\documentclass[a4paper]{article}
\usepackage{amsmath, amsthm, amsfonts, amssymb, mathrsfs}
\usepackage[english]{babel}
\usepackage[utf8]{inputenc}
\usepackage[margin=1 in]{geometry}

\usepackage{graphicx}
\usepackage{cancel}

\DeclareMathOperator{\rank}{rank}
\DeclareMathOperator{\im}{im}

\def\C{\mathbb{C}}
\def\F{\mathbb{F}}
\def\Q{\mathbb{Q}}
\def\N{\mathbb{N}}
\def\R{\mathbb{R}}
\def\SW{\mathbb{SW}}

\def\dgm{\mathrm{dgm}}
\def\bcd{\mathrm{bcd}}

\usepackage{enumitem}
\usepackage{subfig}
\usepackage{colortbl}
\usepackage{array}
\usepackage{tikz-cd}
\usepackage{algorithm}
\usepackage{algpseudocode}
\usepackage[colorlinks = true, citecolor = blue, urlcolor = blue]{hyperref}

\definecolor{lightgray}{RGB}{230, 230, 230}
\definecolor{darkgray}{RGB}{204, 204, 204}
\definecolor{navyblue}{RGB}{137, 207, 240}

\newtheorem{theorem}{Theorem}[section]
\newtheorem{lemma}[theorem]{Lemma}
\newtheorem{corollary}[theorem]{Corollary}
\newtheorem{definition}[theorem]{Definition}
\newtheorem{example}[theorem]{Example}
\newtheorem{proposition}[theorem]{Proposition}

\newtheoremstyle{remarkstyle}
  {3pt} 
  {3pt} 
  {} 
  {} 
  {\itshape} 
  {.} 
  {.5em} 
  {} 

\theoremstyle{remarkstyle}
\newtheorem{remark}{Remark}
\usepackage{appendix}

\title{Estimation of Persistence Diagrams Via the Three Gap Theorem\thanks{ This work was partially supported by the National Science Foundation through   CAREER award \# DMS-2415445.}}

\author{Luis Suarez Salas \ \ \ \ \ \ \ \ \ \ \ \ Jose A. Perea}

\date{ }
\begin{document}
\maketitle

\begin{abstract}
The time delay (or Sliding Window) embedding is a technique from dynamical systems to reconstruct attractors from time series data.
Recently, descriptors from Topological Data Analysis (TDA) —
specifically, persistence diagrams — have been used to measure the
shape of said reconstructed attractors in applications including periodicity and
quasiperiodicity quantification. Despite their utility, the fast computation
of persistence diagrams of sliding window embeddings
is still poorly understood. In this work, we present theoretical and
computational schemes to approximate the persistence diagrams of
sliding window embeddings from quasiperiodic functions. We do
so by combining the Three Gap Theorem from number theory with
the Persistent K\"unneth formula from TDA, and derive fast and provably
correct persistent homology approximations. The input to our
procedure is the spectrum of the signal, and we provide numerical as well as theoretical
evidence of its utility to capture the shape of toroidal attractors.
\end{abstract}

\section{Introduction}
\label{sec:intro}
The Sliding Window embedding is a powerful technique for reconstructing attractors of a dynamical system from observed time series data.
When combined with tools from Topological Data Analysis (TDA), such as persistence diagrams, one can measure the topological properties of the state space
to detect patterns in the original signal.  In general, combining sliding windows with persistence can be used to probe for the presence of non-trivial recurrent behavior. For example, Figure \ref{fig:IlustPendulum} depicts a system with behavior more complex than periodicity.
\begin{figure}[!htb]
    \centering
    \includegraphics[width=0.44\textwidth]{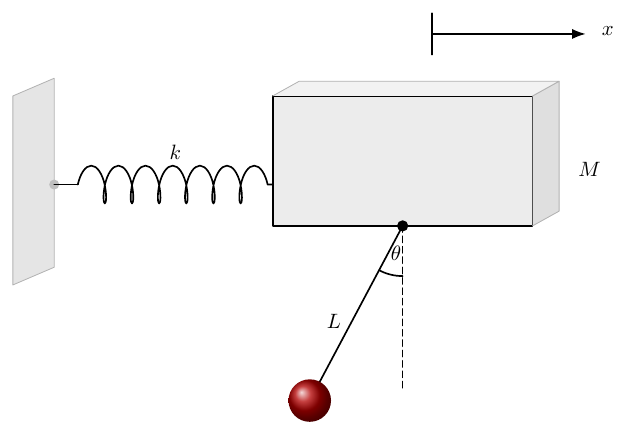} \ \ \ \ \ \
    \includegraphics[width=0.45\textwidth]{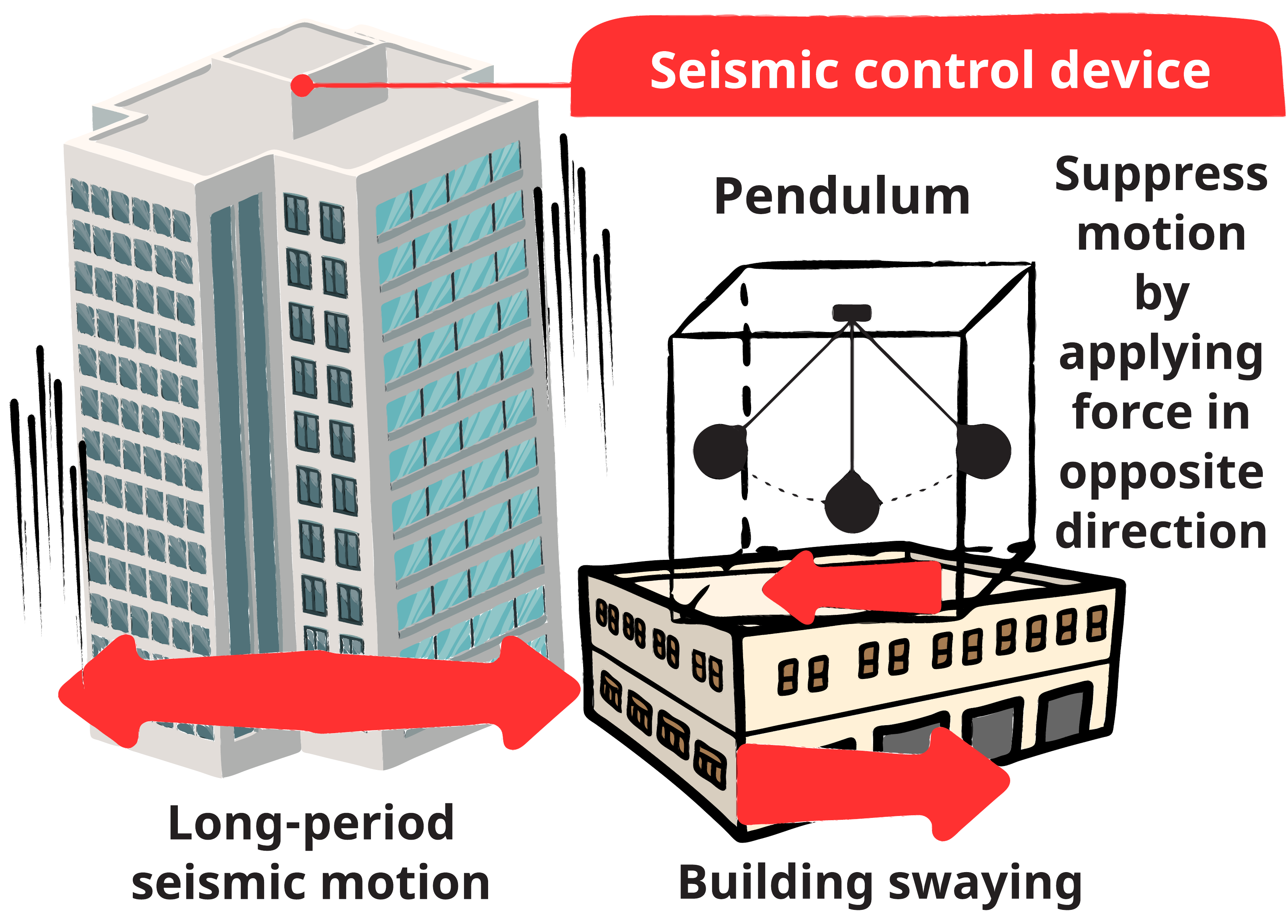}
    \caption{Left: Schematic model for a pendulum on a sliding block. Right: Illustration of anti-earthquake technology (see also  \cite{PhysOrg2013Quake}), whose dynamics are captured by the model on the left.}
\label{fig:IlustPendulum}
\end{figure}

\noindent Indeed, leveraging measurements of the system, such as the horizontal position $x(t)$ of the block or the angular position $\theta(t)$ of the pendulum, one can determine (as we will explain below) that the system exhibits a type of recurrent behavior known as \emph{quasiperiodicity} (Figure \ref{fig:pipeline} and Section \ref{sec:dynamical}).
In dynamical systems this behavior emerges when  oscillators are superimposed, also in the transition between stable and chaotic dynamics \cite{weixing1993quasiperiodic}, and has been observed, for instance, in fMRI scans obtained from mice \cite{belloy2017dynamic}, in the formation of quasicrystals \cite{levine1984quasicrystals}, and in the study of biphonation in  mammal vocalization \cite{wilden1998subharmonics}.

Mathematically, we say that $f:\mathbb{R} \rightarrow \mathbb{C}$ is quasiperiodic (see Definition~\ref{def:quasiperiodic}) with frequency vector $\omega = (\omega_1,\ldots, \omega_N)$, if $\{\omega_i \}^N_{i=1}$ are positive real numbers which are linearly independent over $\mathbb{Q}$ (i.e., incommensurate)  and
\[
f(t)=F(\omega_1t,\dots,\omega_Nt), \] where $F: \mathbb{T}^N \rightarrow \mathbb{C}$ is a complex-valued continuous function
on the $N$-torus $\mathbb{T}^N =  \mathbb{R}^N / \mathbb{Z}^N$, called the parent function of $f$ \cite{gakhar2021sliding}.
Note that in the case of a single frequency, $f$ is just a periodic function.

Periodicity or quasiperiodicity in a signal with frequency vector of length $N$  corresponds  to observing the traversal of a (topological) circle or an $N$-torus attractor, respectively, in the phase space of the underlying dynamics \cite{skraba2012topological}.
Takens’ embedding theorem \cite{takens1981detecting}, on the other hand, and under appropriate conditions \cite{xu2019twisty}, ensures that the sliding window embedding of the signal (Section \ref{sec:dynamical} and middle panels of Figure \ref{fig:SW})  provides a topological reconstruction of said attractors. Finally, periodicity/quasiperiodicity can be quantified with shape
descriptors from TDA: the Rips persistence diagrams (defined in Section \ref{sub: PH} and shown in the rightmost panels of Figure \ref{fig:SW}) of the reconstructed attractors via sliding window embeddings.
In short, they measure the prominence and dimension of holes in a point cloud, providing a convenient summary of its multi-scale geometry and topology.
\begin{figure}[htb!]
    \centering
    \includegraphics[width=0.32\textwidth]{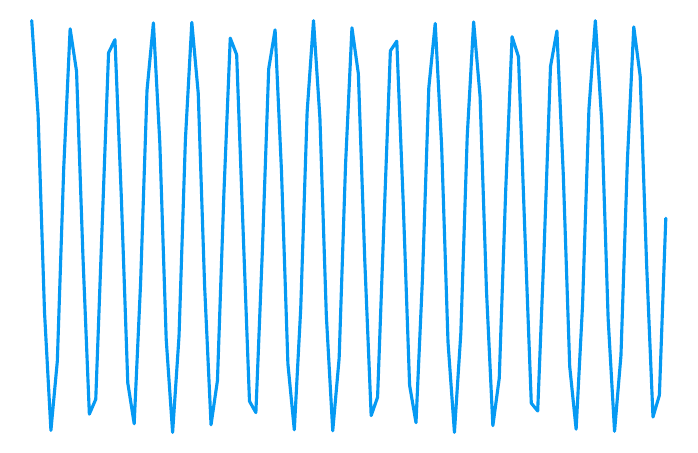}
    \hfill
    \includegraphics[width=0.22\textwidth, trim=9.9cm 10.5cm 9.9cm 9.9cm, clip]{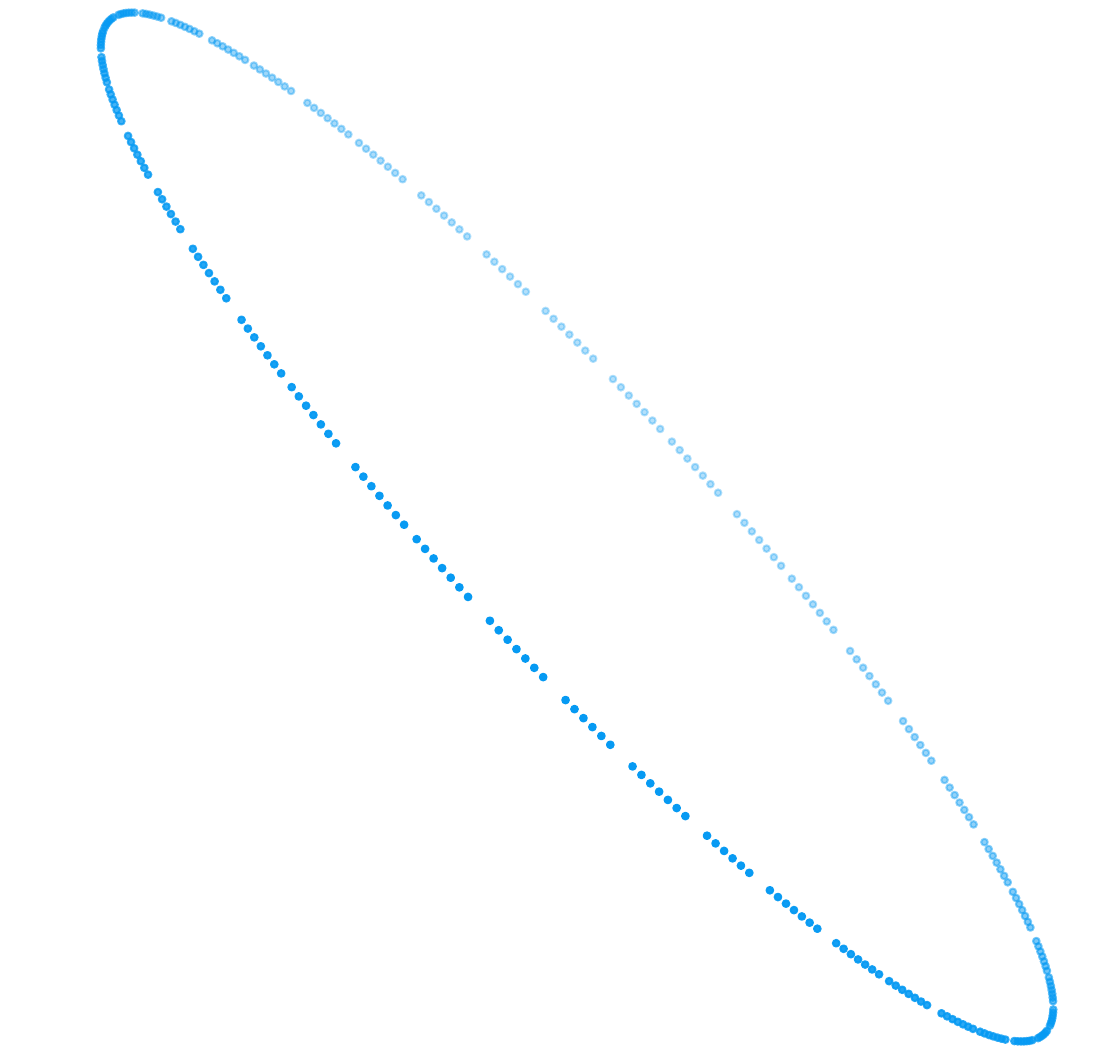}
    \hfill
    \includegraphics[width=0.29\textwidth]{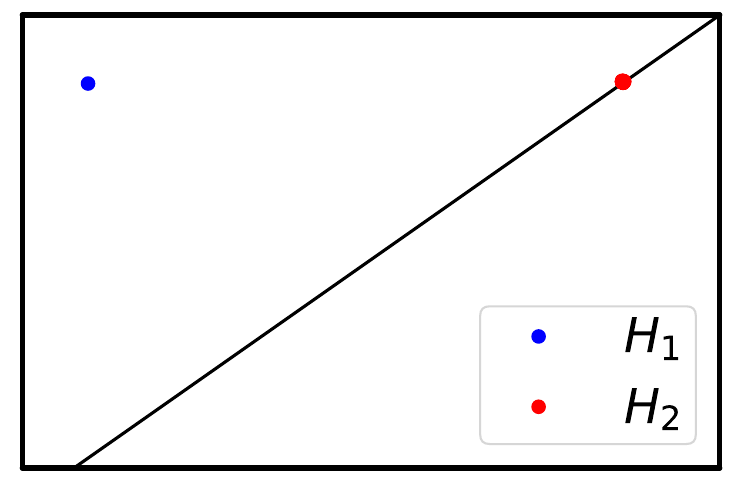}
    \includegraphics[width=0.32\textwidth]{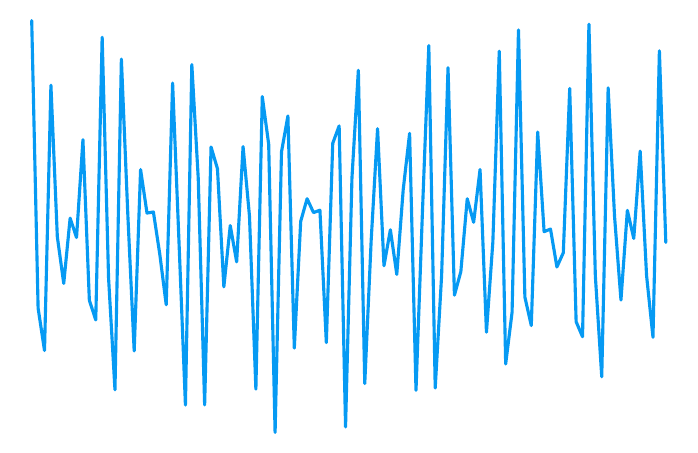}
    \hfill
    \includegraphics[width=0.28\textwidth, trim=8cm 9.9cm 8cm 8cm, clip]{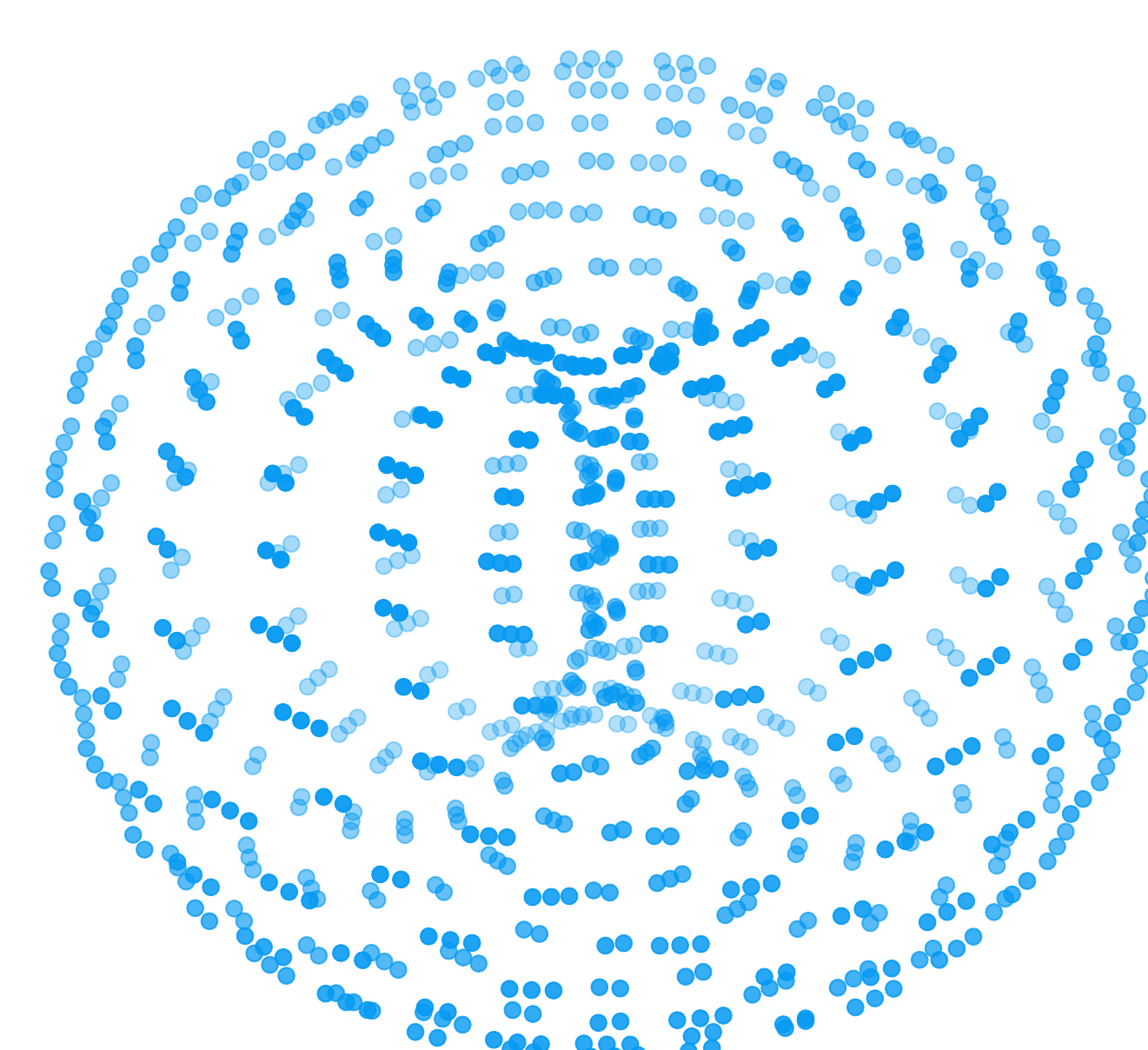}
    \hfill
    \includegraphics[width=0.29\textwidth]{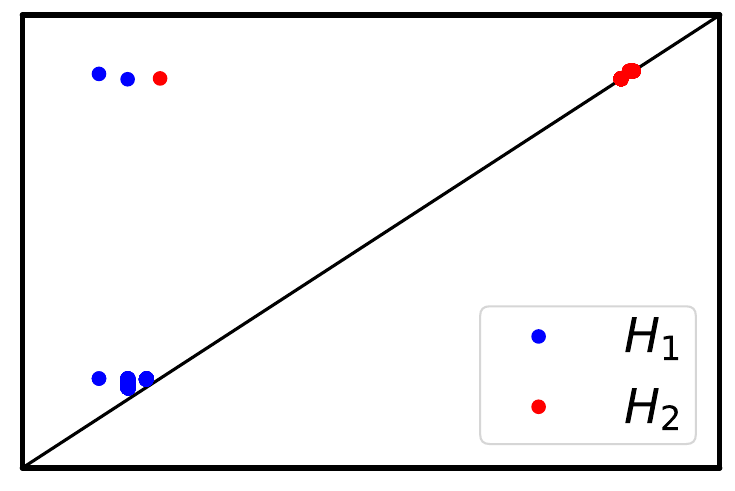}
    \caption{Given a time series (left column), we reconstruct the underlying attractor via its
    sliding window embedding (middle column), and then compute its persistent homology (right column). Top row: Signal from the periodic motion of an ideal pendulum (left). The sliding window point cloud parametrizes a circle (middle), so its persistence diagram exhibits exactly one nontrivial 1-dimensional homology class (right). Bottom row: A quasiperiodic function with two incommensurate frequencies (left); its sliding window embedding (center) recovers a 2-torus. The corresponding persistence diagrams (right) confirm this topology by showing two independent 1-dimensional classes and one 2-dimensional class.}
\label{fig:SW}
\end{figure}

Much work has been done to establish the usefulness of sliding windows and persistence, with
applications including   parameter estimation in damping mechanisms \cite{myers2022damping}, stability of stochastic delay equations \cite{khasawneh2016chatter}, quantifying the presence of biphonation in vibrating vocal folds \cite{tralie2018quasi} and periodicity detection in gene expression data \cite{perea2015sw1pers}.
Furthermore, there are well-defined methodologies for optimizing the embedding parameters of the sliding window embedding, in such a way that the topological features in the persistence diagrams are maximally amplified \cite{gakhar2021sliding}.
However, the algorithmic complexity of computing the persistent homology of Rips filtrations remains a challenge in real-world applications.
This is the case even with highly-optimized libraries like \verb"Ripser" \cite{bauer2021ripser}, which exploits cohomology duality, apparent pairs and clearing optimizations to achieve substantial speed-ups.
The main difficulty lies in that the worst-case complexity of the matrix-reduction algorithm  \cite{zomorodian2005computing} for computing persistent homology in degrees $j=0,\ldots, J$ is cubic in the number of simplices of dimensions $0,\ldots, J+1$ \cite{morozov2005}; the case relevant to this paper is $J= N$, i.e. the number of $\Q$-linearly independent frequencies in the signal $f$.
Moreover, the number of $j$-dimensional simplices for \verb"Ripser" grows as $\binom{n_{SW}}{ j+1}$, where $n_{SW}$ is the number of points in the sliding window embedding.
In practice, this proves to be computationally taxing even when $N= 2$, $n_{SW} = 1000$, and motivates our efforts to develop   approximation methods with known error bounds (see Figure \ref{fig:pipeline}), that would provide a  faster and computationally accessible alternative (see Figure \ref{fig:example 1 3G} and Table \ref{tab:RT}). Our method relies on the Three‑Gap Theorem. This theorem uses continued‑fraction expansions from number theory to describe how points—sampled at irrational rotation angles—are distributed around a circle. We also use the K\"{u}nneth formula in persistent homology. Thanks to the results in Section \ref{sec: FS}, our pipeline applies to a general quasiperiodic function $f$. It suffices to compute the spectrum of $f$ -- e.g. via the fast Fourier Transform (FFT) -- to obtain the input to our algorithm (Section \ref{sec: 3G}). Specifically, we replace the standard Ripser‑based step 3 by our two‑part alternative (3b) shown below (we keep the original procedure as step 3a for comparison).
\begin{center}
 \begin{tikzpicture}
  \draw[rounded corners] (0, 0) rectangle (10, -4.4);

  \node at (6.5, -2.2) {
    \begin{minipage}{0.8\linewidth}
      \begin{tabular}{l}
        \textbf{1.} Start with a time series \(f\). \\
        \textbf{2.} Reconstruct the phase space by computing \\
        \hspace{5mm} $\{SW_{d,\tau}f(t)\}_{t=0}^T$. \\
        \textbf{3a.} Compute dgm$^R_j(\{SW_{d,\tau}f(t)\}_{t=0}^T,d_2)$ using \verb"Ripser". \\
        \textbf{3b.} (i) Use the FFT to retrieve the frequencies of $f$ and \\
        \hspace{5mm} then compute their continued fraction expansion (CFE). \\
        \hspace{5mm} (ii) Use the CFE as shown in Section \ref{MR section} and then apply \\
        \hspace{5mm} the results from Section \ref{sec:AM} to approximate \\
        \hspace{5mm} dgm$^R_j(\{SW_{d,\tau}f(t)\}_{t=0}^T,d_2)$ for $j=1,2$.  \\
      \end{tabular}
    \end{minipage}
  };
\end{tikzpicture}
\end{center}
In step 3b (ii), the continued fraction expansions allow us to compute exact persistence diagrams on each frequency circle (via the Three Gap Theorem, Section \ref{MR section}), which we then assemble—using the Künneth formula as detailed in Section~\ref{sec:AM}—to approximate the persistence diagrams of interest.
\begin{figure}[tbh!]
    \centering
    \hfill
    \includegraphics[width=0.31\textwidth]{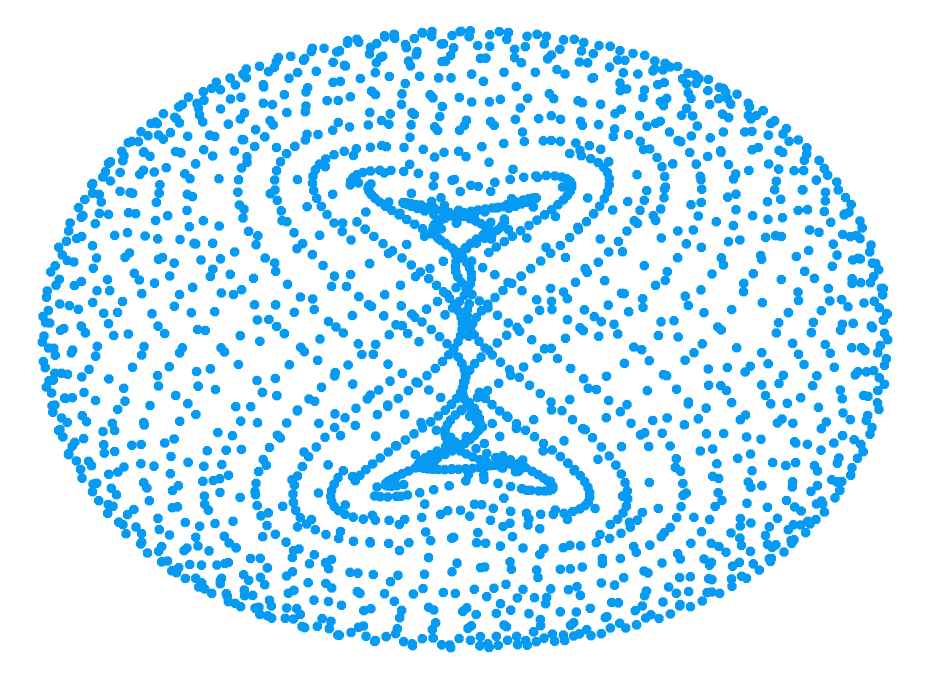}
    \hfill
    \includegraphics[width=0.32\textwidth]{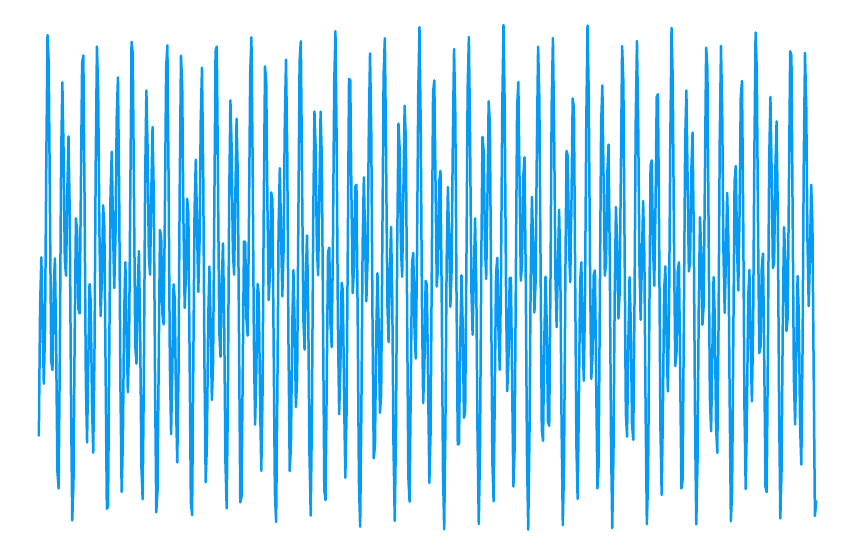}
    \hfill
    \includegraphics[width=0.32\textwidth]{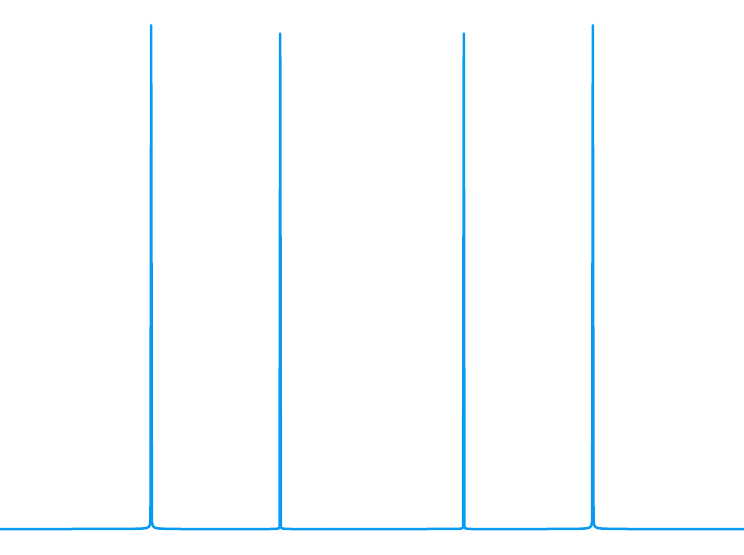}
    \hfill
    \includegraphics[width=0.48\textwidth]{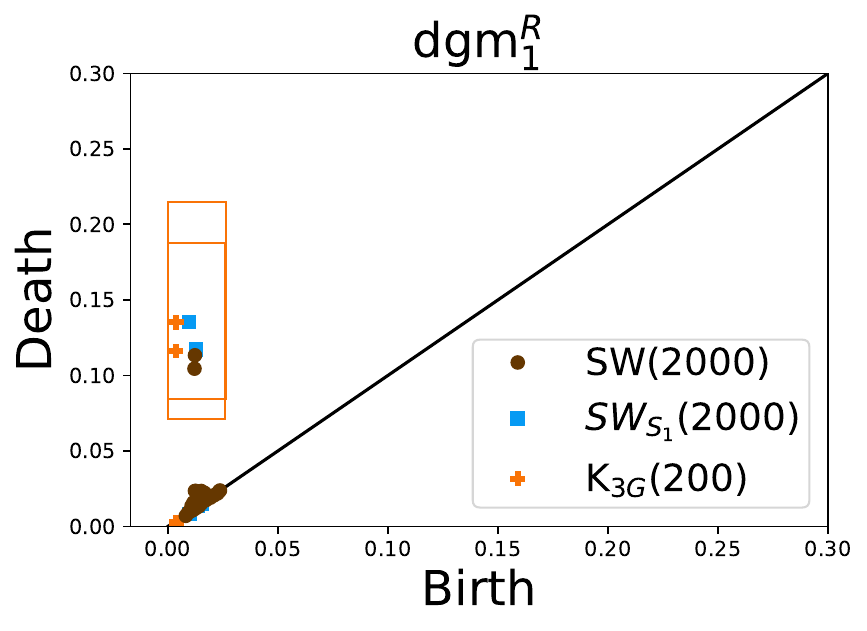}
    \hfill
    \includegraphics[width=0.48\textwidth]{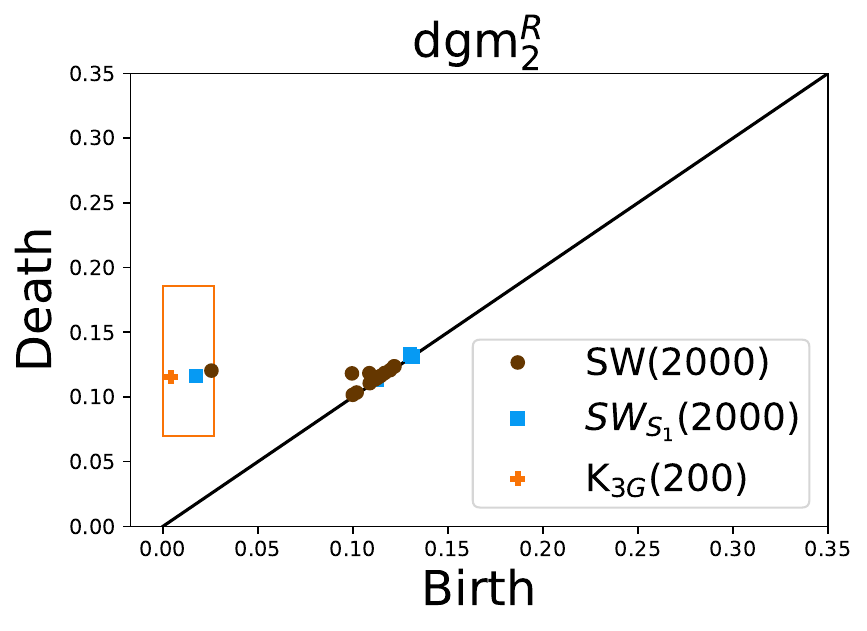}
    \label{fig:applications}
    \caption{Top row: The plot on the left is the phase space $(x,\dot{x}) $ from the pendulum attached to a sliding block shown in Figure \ref{fig:IlustPendulum}. The solution $x(t)$ we used for the sliding window embedding is plotted in the middle, followed by the modulus of its Discrete Fourier Transform (right). Bottom row: Persistence diagrams of the sliding window  embedding of $x(t)$(brown circles), of a truncated Fourier series of $x(t)$ (blue squares), and our proposed approximation $K_{3G}$ (orange crosses). The orange rectangles depict the theoretical approximation bound.}
\label{fig:pipeline}
\end{figure}

\begin{remark}
The idea of using Fourier analysis together with the K\"unneth formula to compute persistent homology is not unique to our work. Indeed, in \cite{kim2022exact}, these tools are combined to construct a multiparameter filtration whose persistent homology can be computed exactly. As the authors state:
\begin{quote}
“Our main idea is to transform time‑series data into a barcode through the Liouville torus without utilizing sliding‑window embedding.”
\end{quote}
Thus their method depends on reconstructing the full Liouville torus.

By contrast, our work approximates the persistence diagram via a single‑parameter sliding‑window Rips filtration:
\[
\begin{tikzcd}[
  column sep=small,   
  row sep=small,      
  every label/.append style={font=\footnotesize, align=center}  
]
  f
    \arrow[r,]
  & R_{\epsilon}(SW_{d,\tau}f)
    \arrow[r,]
  & \text{dgm}^{R}_{j}(SW_{d,\tau}f). \\[-6pt]
  (\mathrm{Time\text{-}series\ data}) & (\mathrm{Rips \ complex\ of \ sliding \ window\ embedding}) & (\mathrm{persistence \ diagram})
\end{tikzcd}
\]
Meanwhile, \cite{kim2022exact} provides exact formulas for the persistence diagram of the Liouville torus \(\Psi_f\), obtained from the multiparameter filtration
\[
\begin{tikzcd}[
  column sep=small,   
  row sep=small,      
  every label/.append style={font=\footnotesize, align=center}  
]
  f
    \arrow[r,]
  & R_{\epsilon_1}(\pi_1\Psi_f)\times\cdots\times R_{\epsilon_N}(\pi_N\Psi_f)
    \arrow[r,]
  & \text{dgm}^{R,\ell}_{j}( \Psi_f ). \\[-6pt]
  (\mathrm{Time \ series\ data}) & (\mathrm{Multi-parameter\ filtration\ of\ the\ Liouville\ torus}) & (\mathrm{persistence \ diagram})
\end{tikzcd}
\]
We summarize the distinguishing features of these two methods in Table \ref{tab:compare}.
\end{remark}

\begin{table}[h]
  \centering
  \caption{At‑a‑glance comparison of our 3G pipeline vs.\ the exact multiparameter method of \cite{kim2022exact}.}
  \label{tab:compare}
  \rowcolors{2}{lightgray}{darkgray}
  \renewcommand{\arraystretch}{1.5}
  \begin{tabular}{|l|c|c|}
    \hline
    \rowcolor{navyblue}
      {{Feature}} &
      {{Our 3G method}} &
      {{Method of \cite{kim2022exact}}} \\
    \hline
    FFT‑based decomposition          & Yes & Yes \\ \hline
    Künneth‑type product formula     & Yes & Yes \\ \hline
    Takens' Theorem     & Yes & Yes \\ \hline
    Sliding window embedding         & Yes & No  \\ \hline
    Liouville torus & No  & Yes \\ \hline
    Continued fraction expansion     & Yes & No \\ \hline
    Filtration type                  & 1‑parameter & Multi-parameter \\ \hline
    Type of signal & Quasiperiodic & Periodic \\ \hline
    Exactness vs.\ approximation     & Approximate (error bound) & Exact \\ \hline
  \end{tabular}
\end{table}

\subsection*{Acknowledgements}
The authors would like to thank Facundo Memoli and  the anonymous
referees for their detailed comments; they  were invaluable in fixing multiple imprecisions found in the original draft.

\section{Background and Definitions}
\label{sec:Back}
In this section we cover the mathematical concepts needed for our Three Gap (3G) method. We first define persistence modules, their barcodes and stability theory, as well as the special case of persistent homology for filtered simplicial complexes.
Next, we introduce dynamical systems and highlight Takens’ reconstruction theorem, which underlies the sliding window embedding. After that, we show how the Fourier transform recovers the frequency vector of a general quasiperiodic signal. We then recall continued fraction expansions and their fundamental properties, leading directly to the Three Gap Theorem --- the number‑theoretic pillar of our work.  We conclude with an overview of our 3G method (named for its use of the Three Gap Theorem), emphasizing how these components fit together.
 \subsection{Persistent Homology}
 \label{sub: PH}
Our presentation of persistence modules and their interleaving distance follows \cite[Chapters~1, 3, and 4]{chazal2016structure}. These concepts provide an algebraic account of stability, which we use in this work.

\subsubsection{Persistence Modules}
\begin{definition}
A persistence module $ \mathbb{V} $ over the   real numbers $\mathbb{R} $ is  an indexed family of vector spaces
\[
\left\{V_t \mid t \in \mathbb{R} \right\},
\]
together with a doubly-indexed family of linear maps
\[
\left\{v_s^t \colon V_s \to V_t \mid s \le t\right\}
\]
which satisfy the composition law
\[
v_s^t \circ v_r^s = v_r^t
\]
whenever \(r \le s \le t\), and where \(v_t^t\) is the identity map on \(V_t\) for every $t\in \mathbb{R}$.
\end{definition}

\begin{definition}
A morphism \(\Phi\) between two persistence modules \(\mathbb{U}, \mathbb{V}\) is a collection of linear maps
\[
\varphi_t : U_t \to V_t
\]
such that the diagram
\[
\begin{tikzcd}
U_s \arrow[r, "u_s^t"] \arrow[d, "\varphi_s"'] & U_t \arrow[d, "\varphi_t"] \\
V_s \arrow[r, "v_s^t"] & V_t
\end{tikzcd}
\]
commutes for all \(s \le t\). Composition is defined in the obvious way, as are identity  morphisms. This makes the collection of persistence modules into a category. The set of all  morphisms from \(\mathbb{U}\) to \(\mathbb{V}\) is denoted
\[
  \mathrm{Hom}\bigl(\mathbb{U},\mathbb{V}\bigr).
\]
When \(\mathbb{U}=\mathbb{V}\), these are called \emph{endomorphisms}, and we write
\[
  \mathrm{End}\bigl(\mathbb{V}\bigr) = \mathrm{Hom}\bigl(\mathbb{V},\mathbb{V}\bigr).
\]
\end{definition}

We have the following generalization of morphism,

\begin{definition}
Let \(\mathbb{U}=\{U_t, u_s^t\}\) and \(\mathbb{V}=\{V_t, v_s^t\}\) be persistence modules indexed by \(t\in \R \).
A \emph{morphism of degree \(\delta \geq 0\)} is a family of linear maps
\[
  \varphi_t \colon U_t \;\longrightarrow\; V_{t+\delta},
  \quad t\in \R,
\]
such that for all \(s\le t\) the square
\[
\begin{tikzcd}
U_s \arrow[r,"u_s^t"] \arrow[d,"\varphi_s"']
  & U_t \arrow[d,"\varphi_t"] \\
V_{s+\delta} \arrow[r,"v_{s+\delta}^{t+\delta}"]
  & V_{t+\delta}
\end{tikzcd}
\]
commutes.  We write
\[
  \operatorname{Hom}^\delta(\mathbb{U},\mathbb{V})
  = \{\varphi\text{ of degree }\delta\}.
\]
When \(\mathbb{U}=\mathbb{V}\), such a map is called a \emph{degree \(\delta\) endomorphism}, and we write
\[
  \operatorname{End}^\delta(\mathbb{V})
  = \operatorname{Hom}^\delta(\mathbb{V},\mathbb{V}).
\]
\end{definition}

\begin{definition}
For any persistence module \(\mathbb{V}=\{V_t,v_s^t\}\) and \(\delta\ge0\), the \emph{shift map}
\[
  1^\delta_{\mathbb{V}}
  = \{\,v_t^{t+\delta}\colon V_t\to V_{t+\delta}\}_{t\in \R}
\]
is a degree \(\delta\) endomorphism in \(\operatorname{End}^\delta(\mathbb{V})\).
\end{definition}

\begin{definition}
Let \(\mathbb{U}\) and \(\mathbb{V}\) be persistence modules and fix \(\delta\ge0\).  A \(\delta\)–\emph{interleaving} consists of
\[
  \Phi\in\operatorname{Hom}^\delta(\mathbb{U},\mathbb{V}),
  \quad
  \Psi\in\operatorname{Hom}^\delta(\mathbb{V},\mathbb{U}),
\]
such that
\[
  \Psi\Phi = 1^{2\delta}_{\mathbb{U}}
  \quad\text{and}\quad
  \Phi\Psi = 1^{2\delta}_{\mathbb{V}}.
\]
\end{definition}

\begin{definition}
The \emph{interleaving distance} between two persistence modules \(\mathbb{U}\) and \(\mathbb{V}\) is
\[
  d_I(\mathbb{U},\mathbb{V})
  = \inf\{\delta\ge0 \mid \mathbb{U}\text{ and }\mathbb{V}\text{ admit a }\delta\text{-interleaving}\}.
\]
If no finite \(\delta\) exists, set \(d_I(\mathbb{U},\mathbb{V})=\infty\).
\end{definition}

Under appropriate finitenes conditions, which we describe below, a persistence module can be uniquely determined up to isomorphism by an invariant called the \emph{barcode}.
Indeed, first
note that the direct sum of persistence modules can be defined component-wise; that is, $\mathbb{U\oplus V}$ has vector spaces $U_t \oplus V_t$ (direct sum of vector spaces) and linear maps $v_{s}^t \oplus u_s^t$.

\begin{definition}
\label{def:dgm-module}
Let $\mathbb{V}$ be a persistence module  such that  \(V_{t}\) is finite‑dimensional for each $t\in \R$ (i.e., $\mathbb{V}$ is pointwise finite).
By \cite{crawley2015decomposition}, there is a
unique multiset $\bcd(\mathbb{V})$ of intervals $I\subseteq \R$, called the barcode of $\mathbb{V}$,
so that
\[
  \mathbb{V}\;\cong\;\bigoplus_{I \in \bcd(\mathbb{V})} \mathbb{I}_I,
\]
where  $\mathbb{I}_I$ is the persistence module with the following vector spaces: the field $\F$ for each $t\in I$, the zero vector space otherwise, and the following linear maps: the identity whenever possible, and the zero map when not.
By multiset we mean that intervals may appear with finite repetition (multiplicity).
The persistence diagram of $\mathbb{V}$ is the following multiset of points in in the extended plane $\overline{\R}^2 = [-\infty, \infty] \times [-\infty, \infty]$,
\[
\dgm(\mathbb{V}) =
\left\{
\big(\inf(I), \sup(I)\big) \mid I \in \bcd(\mathbb{V})
\right\}.
\]
\end{definition}

The space of persistence diagrams can be endowed with an extended pseudo-metric called the bottleneck distance $d_B$ \cite{cohen2007stability, chazal2016structure}. In the following definition we consider the points in the diagonal $y=x$ as part of the persistence diagrams;
moreover, each point $(x,x)$
is included with  countable infinite multiplicity.
\begin{definition}
\label{bottleneck def}
Let $\dgm, \dgm' \subseteq \overline{\R}^2$ be two persistence diagrams.
Their bottleneck distance  is
\[
d_B\left(\dgm ,\dgm'\right)
=
\inf_{\phi:\dgm  \rightarrow \dgm'}\;
 \sup_{y\in \dgm }\{ \lVert y-\phi(y) \rVert_{\infty} \} ,
\]
where $\phi$ is a bijection of multisets, and  $\|(a,b)\|_\infty = \max\{|a|, |b|\}$ is the extended sup norm on $\overline{\R}^2$.
\end{definition}
Conceptually, given two persistence diagrams we consider all multiset bijections that pair points between them.
The distortion of a given pairing can be quantified using the largest infinity norm  over all paired points.
The bottleneck distance is then the smallest possible distortion over all bijective pairings. In essence, this captures how similar two persistence diagrams are to each other in the plane. We note that including all points \((x,x)\) along the diagonal  ensures that any two diagrams admit a bijective matching (so that unmatched off‑diagonal points may be paired with diagonal points).

The interleaving distance between persistence modules and the bottleneck distance between their persistence diagrams satisfy the celebrated isometry theorem \cite{lesnick2015theory},

\begin{theorem}
\label{isometry theorem}
 Let \(\mathbb{U}, \mathbb{V}\) be pointwise finite persistence modules. Then
\[
d_I(\mathbb{U}, \mathbb{V}) = d_B(\operatorname{dgm}(\mathbb{U}), \operatorname{dgm}(\mathbb{V})).
\]
\end{theorem}

\subsubsection{Persistent Homology of Filtered Simplicial Complexes}
In this section we describe a formal notion of “shape” for discrete data sets.
Unless otherwise noted, definitions follow \cite{MR1867354,gakhar2021sliding}.
\begin{definition}
Given a set $X$, an abstract simplicial complex (or a simplicial complex, for short) with vertices in $X$  is a  subset $K\subseteq \mathcal{P}(X)$ of the power set of $X$, such that
\begin{enumerate}
\item Every $\sigma \in K$ is finite and nonempty,
\item For any $\sigma \in K$,  if  $\tau \neq \varnothing$  and $ \tau \subseteq \sigma$, then  $\tau \in K$.
\end{enumerate}
\end{definition}

The elements of  $K$ are called simplices and the dimension of a simplex is   one less than its cardinality: $\dim(\sigma)=|\sigma|-1.$
We also define $\dim(K)=\max\{\dim(\sigma) \mid \sigma \in K\}.$
A face of $\sigma$ is any nonempty proper subset of it.
The $n$-th skeleton of $K$ is denoted   $K^{(n)}:=\{\sigma \in K \mid \dim(\sigma) \leq n \}$ for $n\in\mathbb{N} =\{0,1,\ldots \}$;
in particular, $K^{(0)}$ is called the set of vertices of $K$ and $K^{(1)}\smallsetminus K^{(0)}$ is its set of edges.
A subcomplex of  $K$  is a subset   $L\subseteq K$ which is also an abstract simplicial complex.

We can construct an abstract simplicial complex from any  set of points in a metric space.
Indeed, this is achieved via the \emph{Rips complex} which treats each  point in a metric space $(X,d)$, say $x \in X $, as a vertex; edges $\{x_0,x_1\}$ are created between two vertices if their distance is at most  some fixed $\epsilon \geq 0$, i.e., $ d(x_0,x_1) \leq \epsilon$; connecting three vertices creates a triangle $\{x_0,x_1,x_2\}$, four a tetrahedron  $\{x_0,x_1,x_2,x_3\}$, and so on.
Formally,
\begin{definition}
Given a  metric space $(X,d)$ and a real number $\epsilon$, the Rips complex $R_\epsilon(X,d)$ is the simplicial complex
\[
R_\epsilon(X,d) = \big\{ \,\{x_0,\ldots,x_n\} \subseteq X \mid   d(x_i,x_j) \leq \epsilon, \; \forall \; 0\leq i,j \leq n \,  \big\}. \]
\end{definition}
With this construction we can leverage  a powerful shape descriptor from algebraic topology: simplicial homology, which assigns algebraic objects to an abstract simplicial complex in a deformation-invariant manner.
In particular, the 0-th dimensional homology measures the number of connected components, the $1$-st dimensional homology detects loops, and the $2$-nd dimensional homology detects cavities.
In general, the $n$-th dimensional homology detects $n$-dimensional holes.  The relevant definitions are below.

\begin{definition}
Let $K$ be a simplicial complex and $\mathbb{F}$ a field.
For $n\in \mathbb{N}$, let  $C_n(K;\mathbb{F})$ denote the $\mathbb{F}$-vector space with basis the set of $n$‑simplices of $K$:
\[
  C_n(K;\mathbb{F}) = \left\{\sum_{\sigma\in K^{(n)}\smallsetminus K^{(n-1)}} a_\sigma\,\sigma
    \;\bigg|\; a_\sigma\in \mathbb{F},\
    a_\sigma\neq0\text{ for only finitely many }\sigma\right\}.
\]
For $n < 0$, we let $C_n(K;\mathbb{F})$ be the zero vector space.
\end{definition}
We call an element $\tau\in C_n(K;\mathbb{F})$ an \emph{$n$‑chain}.
If
\[
  \tau = \sum_{\sigma\in K^{(n)}\smallsetminus K^{(n-1)}} a_\sigma\,\sigma
\]
and $\sigma_0$ is one of the simplices in the sum (i.e.\ $a_{\sigma_0}\neq0$), we write $\sigma_0 \in \tau$.

\begin{definition}
Fix a partial (e.g., a total) order $ \preceq$ on the vertices of $K$ so that each simplex $\sigma \in K$ is totally ordered.
The $n$-th boundary map $\partial_n:C_n(K;\mathbb{F}) \longrightarrow C_{n-1}(K;\mathbb{F})$ is the linear transformation defined for any $\sigma = \{v_0,\dots,v_n\} \in K^{(n)}\smallsetminus K^{(n-1)}$  as
\[
\partial_n(\sigma)= \sum_{i=0}^n (-1)^n\{v_0,\dots,\hat{v_i}, \dots, v_n\},
\]
where  $v_i \preceq v_{i+1}$ for $i=0,\ldots, n-1$, and $\{v_0,\dots,\hat{v_i}, \dots, v_n\}$ denotes the $(n-1)$-th face of $\sigma$ obtained by removing the $i$-th vertex $v_i$ from the totally ordered set $\{v_0,\dots,v_n\}.$
\end{definition}
One can show that $\partial_{n} \circ \partial_{n+1} = 0$ for all $n\in \N$.
\begin{definition}
For $n\in \N$, a field $\mathbb{F}$ and a simplicial complex $K$, let
\[
\ker(\partial_n) = \{\tau \in C_n(K;\mathbb{F}) \mid \partial_n(\tau) = 0\}
\hspace{.5cm}
\mbox{ and }
\hspace{.5cm}
\im(\partial_{n+1}) = \{\partial_{n+1}(\beta) \mid \beta \in C_{n+1}(K ;\mathbb{F})\} \subseteq \ker(\partial_{n}).
\]
The $n$-th homology of  $K$ with coefficients in $\mathbb{F}$ is the quotient vector space
\[
H_n(K;\mathbb{F}) = \frac{\ker(\partial_n)}{\im(\partial_{n+1})}.
\]
\end{definition}
We note that
the isomorphism type of $H_n(K;\F)$ is independent of the partial order on the vertices of $K$.
Moreover,
when $K=R_{\epsilon}(X,d)$ and $\epsilon \geq 0$, the nonzero elements of  $H_n\bigl(R_\epsilon(X,d);\mathbb{F}\bigr)$  correspond to nontrivial $n$‑cycles (i.e., elements of $\ker(\partial_n)$) not bounding any $(n+1)$‑chain in the Rips complex \(R_\epsilon(X,d)\), and thus measure  $n$-dimensional holes in $(X,d)$ at scale $\epsilon$.
We note that
\[
  R_{\epsilon}(X,d)\;\subseteq\;R_{\epsilon'}(X,d)
  \quad\text{whenever}\;\epsilon\le\epsilon',
\]
which is a particular case of the following definition.

\begin{definition}
\label{def fcom}
A filtered simplicial complex $\mathcal{K} = \{K_\epsilon\}_{\epsilon \in \R}$ is a collection of simplicial complexes $K_\epsilon$ such that
\[
K_\epsilon \subseteq K_{\epsilon'}
\]
whenever $\epsilon \leq \epsilon'$. We refer to $\mathcal{K}$ as a filtration.
\end{definition}

We note that the family
\[
  \mathcal{R}(X,d)=\{R_{\epsilon}(X,d)\}_{\epsilon\in \R}
\]
is a filtered complex in the sense of Definition \ref{def fcom}, called the \emph{Rips filtration} of the metric space $(X,d)$.
The main idea behind persistent homology is to track the evolution of homology classes in a filtration as $\epsilon$ changes.
For instance, as $\epsilon$ increases one may see several connected components (0‑cycles) merge into one, or an empty spherical cavity (a 2‑cycle) become filled.
The resulting multiscale summary is recorded in the persistence diagram of the following persistence module.

\begin{definition}
\label{def bcd}
Let  $\mathcal{K}=\{K_\epsilon\}_{\epsilon\in \R}$  be a filtration.  For each  $n\in \mathbb{N}$, define the  $n$-th persistence module
\[
H_n(\mathcal{K};\mathbb{F}) = \big\{  H_n(K_\epsilon;\mathbb{F}), \ T_{\epsilon,\epsilon'}:H_n(K_\epsilon;\mathbb{F}) \\ \longrightarrow H_n(K_{\epsilon'};\mathbb{F}), \ \epsilon \leq \epsilon' \ \big\}
\]
as the family of \ $\mathbb{F}$-vector spaces $H_n(K_\epsilon;\mathbb{F})$ and linear transformations $T_{\epsilon,\epsilon'}$ induced by the inclusion maps $K_\epsilon \hookrightarrow K_{\epsilon'}$, for $\epsilon \leq \epsilon'$. The $n$-th persistent homology groups are
\[
H^{\epsilon,\epsilon'}_n(\mathcal{K};\mathbb{F})=\im(T_{\epsilon,\epsilon'})
\]
and their dimension over $\mathbb{F}$ are the persistent Betti numbers
\[
\beta^{\epsilon,\epsilon'}_n(\mathcal{K};\F):=\rank(T_{\epsilon,\epsilon'}) = \dim_{\mathbb{F}} \big( H^{\epsilon,\epsilon'}_n(\mathcal{K};\mathbb{F}) \big).
\]
\end{definition}

If $\mathcal{K}=\{K_\epsilon\}_{\epsilon\in \R}$ and  $n \in \mathbb{N}$ are such that $ H_n(\mathcal{K};\mathbb{F})$
is pointwise finite‑dimensional, then
Definition \ref{def:dgm-module}
ensures the existence of the barcode
\[
\bcd_n(\mathcal{K})
= \bcd(H_n(\mathcal{K}; \F))
\]
and the persistence diagram
$\dgm_n(\mathcal{K}) = \dgm(H_n(\mathcal{K}; \F))$.

\begin{remark}
\label{rmk:diagramAltSum}
When the $n$-th homology of a filtration changes at only finitely many filtration values,
the resulting persistence diagram has a particularly simple description in terms of  persistent Betti numbers   (see \cite[Page 106]{cohen2007stability}).
Indeed, let $\mathcal{K} = \{K_\epsilon\}_{\epsilon \in \R}$
be a filtration whose $n$-th persistence module is pointwise finite and changes at only finitely many values of $\epsilon$,
namely $\epsilon_1 < \epsilon_2 < \cdots < \epsilon_J $ (i.e., finitely many homological critical values), then given interleaved values
$\epsilon_j < r_j < \epsilon_{j+1}$
with $r_0 = \epsilon_{-1} = -\infty$ and $r_{J+1} = \epsilon_{J+1} = \infty$,
the persistence diagram
$\dgm_n(\mathcal{K})$ is the collection of points
$(\epsilon_i , \epsilon_j)$ that have \emph{positive} multiplicity
\begin{equation}
\label{eq:BettiMult}
\beta^{r_{i-1}, r_{j}}_n(\mathcal{K} ; \F)
\;-\;
\beta^{r_{i}, r_{j}}_n(\mathcal{K}; \F)
\;+ \;
\beta^{r_{i}, r_{j-1}}_n(\mathcal{K} ; \F)
\;- \;
\beta^{r_{i-1}, r_{j-1}}_n(\mathcal{K}; \F).
\end{equation}
\end{remark}

Throughout this paper, and unless otherwise stated,  we assume that every metric space $(X,d)$ is finite.
Hence, for each $\epsilon\in \R$ the Rips complex $R_{\epsilon}(X,d)$ is a finite simplicial complex and the persistence module
\[
  H_n\bigl(\mathcal{R}(X,d);\mathbb{F}\bigr)
\]
is pointwise finite.
The corresponding barcode and persistence diagram are denoted  $\bcd_n^R(X,d)$ and $\dgm_n^R(X,d)$, respectively, and the intervals $I \in \bcd_n^R(X,d)$ are closed on the left and open on the right.
An interval $[a,b) \in \bcd_n^R(X,d)$  indicates that an $n$‑cycle is born at $\epsilon=a$ and becomes trivial at $\epsilon=b$; its persistence (lifetime)  \(b-a\)  provides a measure of robustness for this homological feature.
Plotting $\dgm^R_n(X,d)$ in the extended plane  $  \overline{\R}^2$ yields a visualization of the persistent homology of the Rips filtration $\mathcal{R}(X,d)$.
Points $(a,b) \in \dgm_n^R(X,d)$ that lie far from the diagonal \(y = x\) represent topological features of $(X,d)$ with greater robustness, as the Stability Theorem \ref{thm:Pstability}  implies. We describe this theorem next.

When working with metric spaces and the Rips filtration, we can connect the bottleneck distance of the corresponding persistence diagrams with the Gromov-Hausdorff distance. The latter measures the similarity between bounded metric spaces and is defined below.

\begin{definition}
Given two non-empty bounded subsets \(A\) and \(B\) of a metric space \( (X, d) \), the \emph{Hausdorff distance} \( d_H(A, B) \) is defined as:
\[
d_H(A, B) = \max \left\{ \sup_{a \in A} \inf_{b \in B} d(a, b), \sup_{b \in B} \inf_{a \in A} d(b, a) \right\},
\]
where \(d(a, b)\) denotes the distance between points \(a\) and \(b\) in \(X\).
\end{definition}

\begin{definition}
For two bounded metric spaces \((X_1,d_1)\) and \((X_2,d_2)\), the Gromov–Hausdorff distance
\[
  d_{GH}(X_1,X_2)
\]
is defined as follows. Take any common metric space \(Z\) and any pair of isometric embeddings
\(\phi_1\colon X_1\to Z\) and \(\phi_2\colon X_2\to Z\); compute the Hausdorff distance between
\(\phi_1(X_1)\) and \(\phi_2(X_2)\) in \(Z\). Then \(d_{GH}(X_1,X_2)\) is the infimum of
all such Hausdorff distances:
\[
d_{GH}(X_1, X_2) = \inf_{Z, \phi_1, \phi_2} d_H(\phi_1(X_1), \phi_2(X_2)).
\]
\end{definition}
The aforementioned connection is the content of the well-known Rips stability theorem \cite{cohen2007stability, chazal2016structure}.
\begin{theorem}
\label{thm:Pstability}
Let $X_1$ and $X_2$ be finite metric spaces. Then
\[
d_B(\dgm^R_n(X_1,d_1),\dgm^R_n(X_2,d_2)) \leq 2d_{GH}(X_1,X_2).
\]
\end{theorem}
In addition to stability, one can also describe  the Rips persistence of metric products.
 \begin{definition}
Let $(X,d_X)$ and $(Y,d_Y)$ be metric spaces. The maximum metric $d_\infty$ is given by
\[
d_\infty\big((x,y),(x',y') \big) = \max\{d_X(x,x'),d_Y(y,y') \},\]
where $(x,y),(x',y') \in X \times Y.$
\end{definition}
 We note that $(X \times Y,d_\infty)$ is a metric space. Furthermore, its barcodes are given by \cite{gakhar2019k}:
 \begin{theorem}[\textbf{Persistent K{\"u}nneth Formula }]
 \label{kunneth formula}
Let $(X,d_X)$ and $(Y,d_Y)$ be finite metric spaces. Then, $$\bcd_n^R(X \times Y,d_\infty) = \bigcup_{i+j=n} \big\{I \cap J \ | \ I \in \bcd_{i}^R(X ,d_{X}), \ J \in \bcd_{j}^R(Y ,d_{Y}) \big\}, $$ for all $n\in \mathbb{N}.$
\end{theorem}

Finally, we present an observation which we use repeatedly in this work.

\begin{proposition}
\label{isomorphic diagrams}
Let $(X,d_X)$ and $(Y,d_Y)$ be two finite metric spaces.
Suppose there exists a bijection $f: X \rightarrow Y$ such that for some $\lambda > 0$ $$d_X(t_1,t_2) = \lambda d_Y(f(t_1),f(t_2)). $$ Then,
$$ \dgm^R_n(X,d_X) = \left\{(\lambda a, \lambda b) \, : \, (a,b) \in \dgm_n^R(Y,d_Y) \right\} .$$
\end{proposition}
\begin{proof}
Let us define the metric $d_Y^\lambda=\lambda d_Y$ on $Y$. By assumption, $f$ is an isometry between $(X,d_X)$ and $(Y,d_Y^\lambda)$. This implies that $f$ induces simplicial isomorphisms between $R_\epsilon(X,d_X)$ and $R_\epsilon(Y,d_Y^\lambda)$  for all $\epsilon \geq 0$. This, in turn, induces isomorphisms on the corresponding persistence modules, which implies they have the same persistence diagrams. \end{proof}

\subsection{Dynamical Systems}
 \label{sec:dynamical}
Having presented persistent homology as a tool  to measure shape in discrete data, we now introduce the framework that models the source of said data.
Indeed, several real-world scientific measurements arise from underlying deterministic mechanisms.
Although said mechanism may be unknown or highly complex, the theory of dynamical systems enables its mathematical description.
The framework consists of a phase space $M$ that represents all of the relevant states of the system and a function $\Phi$ that keeps track of the evolution of the system.
The following definitions are taken from \cite{perea2019topological}.
\begin{definition}
A global continuous time dynamical system is a pair $(M,\Phi)$, where $M$ is a topological space and $\Phi: \mathbb{R} \times M \longrightarrow M$ is a continuous map so that $\Phi(0,p) = p,$ and $\Phi(s,\Phi(t,p)) = \Phi (s+t,p)$ for all $p \in M$ and all $t,s \in \mathbb{R}$.
\end{definition}

Note that any autonomous system of ordinary differential equations
\[
  \dot x(t) \;=\; \Lambda \bigl(x(t)\bigr),
  \quad x(t)\in M,
\]
on a smooth manifold \(M\) defines a dynamical system.  Its flow
\(\Phi_t\colon M\to M\) is given by following the integral curves (solutions) of the ODE from each initial condition.  Hence the evolution of the system is fully determined by the vector field \(\Lambda\) and the chosen initial state.

Representing a model in this fashion allows for a topological understanding of the system. This is done by looking at trajectories in the phase space $M$; indeed,
although most systems can only be solved numerically, given equations,  a general understanding of the shape of trajectories in $M$ provides qualitative information of the system \cite{skraba2012topological}.
In particular, knowing the topology of an \emph{attractor} --- i.e., a subset of $M$ that pulls nearby trajectories into it --- is of great significance. Formally,
\begin{definition}
A set $A\subseteq M$ is called an attractor if
\begin{enumerate}
\item It is compact,
\item It is an invariant set, i.e., if $a\in A$ then $\Phi(t,a) \in A$ for all $t \geq 0,$
\item There is an invariant open neighborhood $U$ of $A$, so that $A=\bigcap\limits_{t \geq 0} \{ \ \Phi(t,p) \mid p \in U \ \}.$
\end{enumerate}
\end{definition}

\begin{example}
\label{ex:radial eq}
Consider the dynamical system given by the radial equation (in polar coordiantes) $$\dot{r}=r(1-r^2), \ \ \ \ \ \ \ \dot{\theta}=1$$ where $r\geq 0.$ This system has a circle as an attractor in the $xy$-plane, shown in red in Figure \ref{fig:radial 2} (Left).
\end{example}

\begin{example}
\label{ex:van 2}
Consider the van der Pol equation
$$ \ddot{x} + \mu (x^2-1)\dot{x}+x=0$$ where $\mu \geq 0$ is a parameter. This system also has an attractor in the $x \dot{x}$-plane as shown in Figure \ref{fig:radial 2} (Right). In this case, it is not a round unit circle, yet it is topologically equivalent to it.
\end{example}

\begin{figure}[!htb]
    \centering
    \includegraphics[width=0.37\textwidth, trim=3cm 3cm 3cm 3cm, clip]{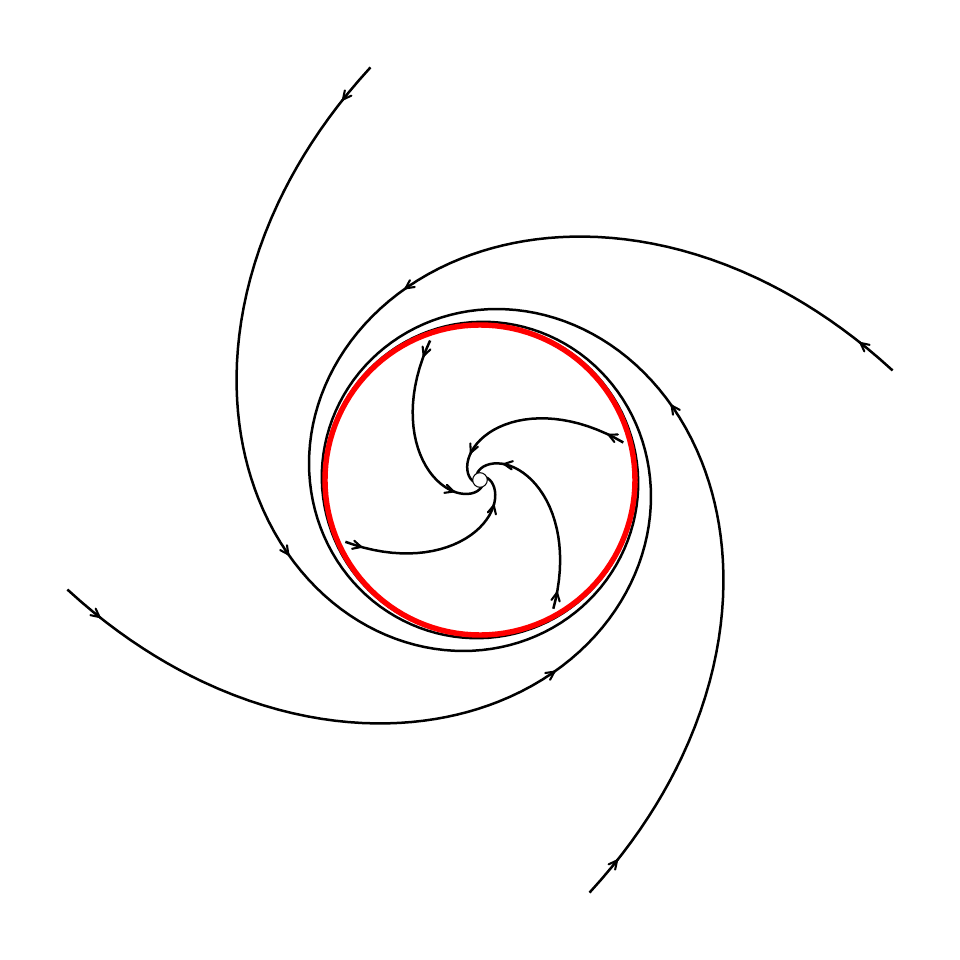} \ \ \ \ \ \ \ \ \ \ \ \ \
     \includegraphics[width=0.31\textwidth, trim=0cm 0cm 0cm 0cm, clip]{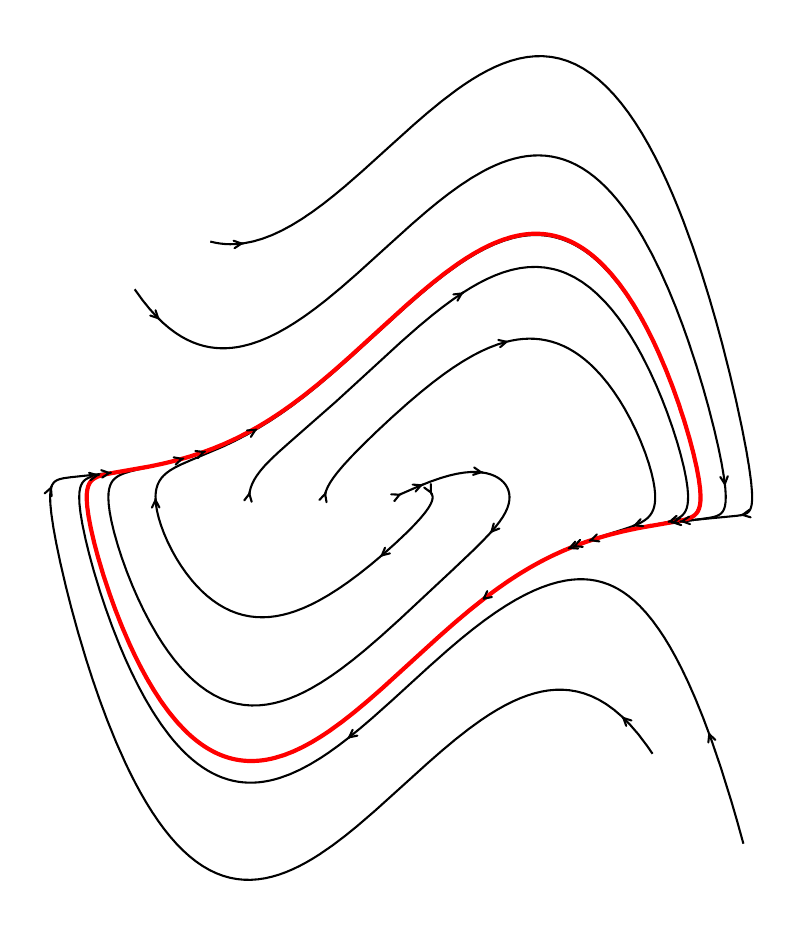}
    \caption{Trajectories of the system from Example \ref{ex:radial eq}, plotted in the \(xy\)-plane (Left) and from Example \ref{ex:van 2} in the $x\dot{x}$ plane (Right). In each case, the attractor is a topological circle highlighted in red.}
    \label{fig:radial 2}
\end{figure}

An attractor that is homeomorphic to a circle corresponds to a system exhibiting periodicity. In general, an attractor that is homeomorphic to an $N$-dimensional torus, a toroidal attractor, comes from a quasiperiodic system \cite{gakhar2021sliding}.
In practice, we often only  have a measurement of the dynamical system in the form of  time series data.
The underlying equations governing the system are unknown—i.e.\ neither \(M\) nor \(\Phi\) are given.
Nevertheless, there is a way to reconstruct attractors of this unknown system while preserving qualitative information. Concretely, Takens’ landmark 1981 theorem \cite{takens1981detecting} shows that a single time series measurement—sampled from a generic observation function—suffices to reconstruct the topology of the underlying attractor.

 \begin{theorem}[\textbf{Takens' Embedding }]
Let M be an $m$-dimensional compact Riemannian manifold. For pairs $(\phi,\bar{f})$, where $\phi \in C^2(M,M)$ and $\bar{f} \in C^2(M,\mathbb{R})$, it is a generic property that the map $\Psi_{\phi,\bar{f}}:M\rightarrow \mathbb{R}^{2m+1}$ defined as $$\Psi_{\phi,\bar{f}}(p)= \big(\bar{f}(p),\bar{f}(\phi(p)),\bar{f}(\phi^2(p)),\dots ,\bar{f}(\phi^{2m}(p)) \big) $$ is an embedding.
\end{theorem}

The reconstruction given by the image of $\Psi_{\phi, \bar{f}}$ may look different from the original state space $M$, but we are assured it will be topologically equivalent.
This is a powerful guarantee that validates the \emph{Sliding Window embedding} (defined below) as a means to reconstructing observed dynamics from  time series data.
\begin{definition}
\label{SW def}
For a function $f: \mathbb{R} \rightarrow \mathbb{C},$ an integer $d >0$ called the embedding dimension, and a real number $\tau > 0$ called the time delay, the sliding window embedding of $f$ at $t$ is given by:
$$ SW_{d,\tau}f(t) =
\begin{bmatrix}
          f(t) \\
          f(t+\tau) \\
           \vdots \\
           f(t+d\tau) \\
\end{bmatrix}   \in \mathbb{C}^{d+1}. $$
For \(T\subseteq\mathbb{R}\) a set of time points, the sliding window point cloud of $f$ sampled at \(T\) is the set
\[
\SW_{d,\tau}(T) =   \{\,SW_{d,\tau}f(t)\mid t\in T\}.
\]
\end{definition}

In the context of a dynamical system
$\Phi: \R\times M \longrightarrow M$, an observation function $\bar{f} : M \longrightarrow \R$,
and an initial condition $p_0 \in M$ so that the trajectory $t\mapsto \Phi(t,p_0)$ densely samples an attractor $A \subseteq M$, Takens' theorem implies that the sliding window of the
time series $t \mapsto \bar{f}\circ\Phi(t,p_0)$ (with appropriate parameters $\tau > 0$ and $d \geq \dim(M)$), produces a point cloud whose Rips persistence diagrams can be used to infer the homology of $A$.

In particular, for toroidal attractors and selecting \((d,\tau)\) according to \cite{gakhar2021sliding}, we are detecting homological features of an $N$-dimensional torus, as depicted in Figures \ref{fig:pipeline} and \ref{fig:SW 2}.
\noindent In summary, here are steps for detecting the presence of toroidal attractors given time series data:
\begin{enumerate}
\item Start with a time series pertaining to a measurement of a system,
\item Construct the sliding window point cloud from it,
\item Compute its persistence diagrams to determine if the point cloud samples an $N$-torus.
\end{enumerate}
Although sliding windows and persistence are well established and have found multiple applications \cite{tralie2018quasi, perea2015sw1pers, gakhar2019k}, computing persistence diagrams remains computationally taxing.
We now move onto the tools that will enable the fast approximation of Rips persistence diagrams of sliding window point clouds from quasiperiodic functions.

\subsection{Fourier Analysis of Quasiperiodic Signals}
\label{sec: FS}
\begin{definition}
\label{def:quasiperiodic}
We say that $f:\mathbb{R} \rightarrow \mathbb{C}$ is quasiperiodic   with frequency vector $\omega = (\omega_1,\ldots, \omega_N)$, if $\{\omega_i \}^N_{i=1}$ are positive real numbers which are linearly independent over $\mathbb{Q}$ (i.e., incommensurate)  and
\[
f(t)=F(\omega_1t,\dots,\omega_Nt), \] where $F: \mathbb{T}^N \rightarrow \mathbb{C}$ is a complex-valued continuous function
on the $N$-torus $\mathbb{T}^N =  \mathbb{R}^N / \mathbb{Z}^N$, called the parent function of $f$.
\end{definition}

As we will see below, one can leverage the spectrum of a quasiperiodic function  to infer the underlying frequency vector, as well as to approximate the Rips persistence diagrams of the sliding window point cloud.
This is enabled by the following   result  \cite[Theorem 1.7 and Corollary 1.13]{gakhar2021sliding}.

\begin{figure}[!htb]
    \centering
    \includegraphics[width=0.40\textwidth]{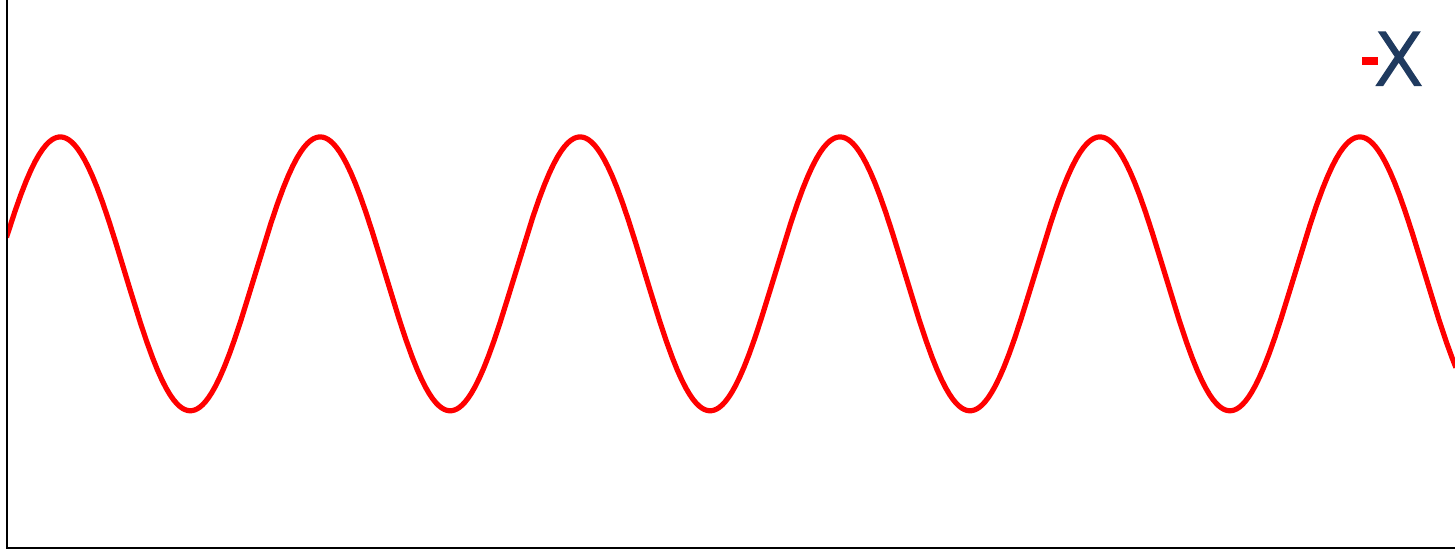}
    \hfill \
    \includegraphics[width=0.28\textwidth, trim=.2cm 2cm 2cm 2cm, clip]{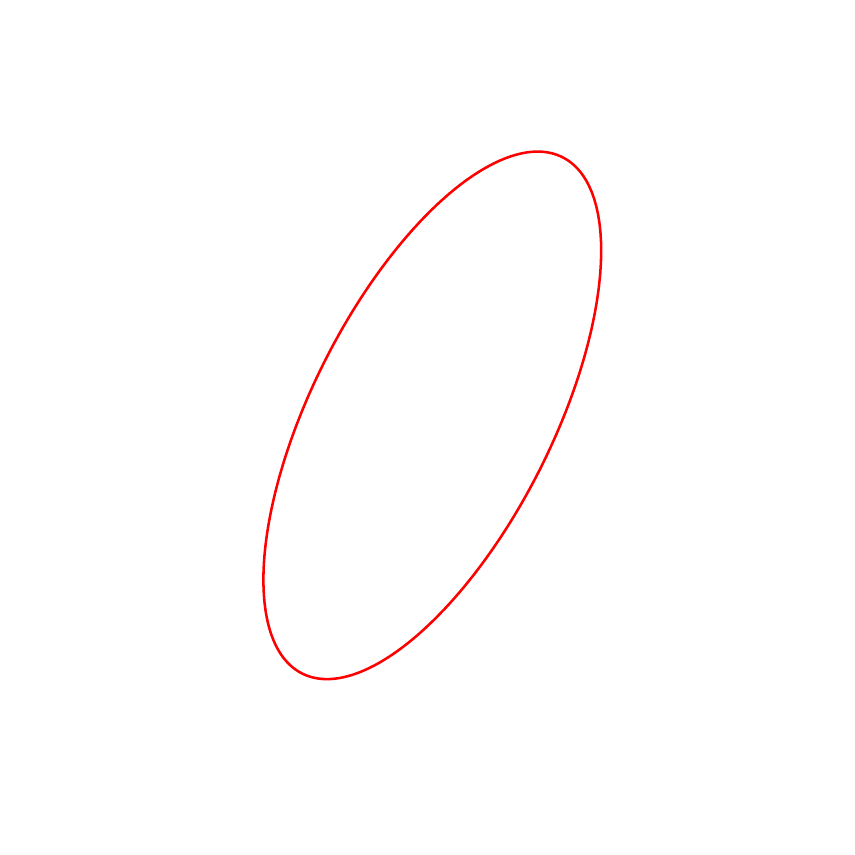}
    \hfill
    \includegraphics[width=0.27\textwidth]{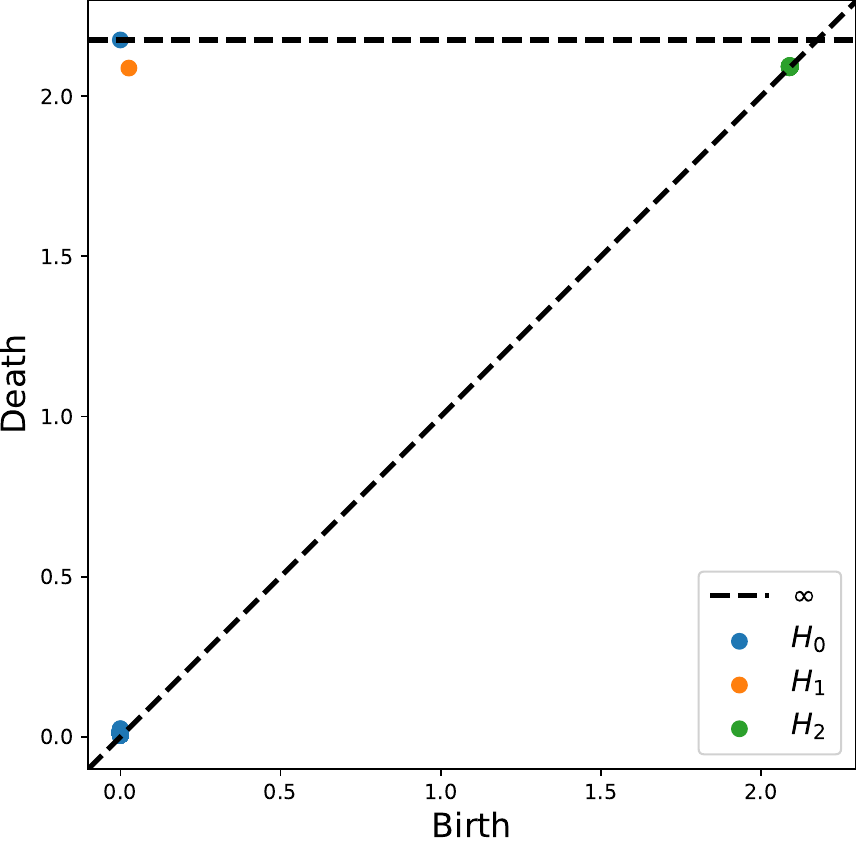}
    \includegraphics[width=0.40\textwidth]{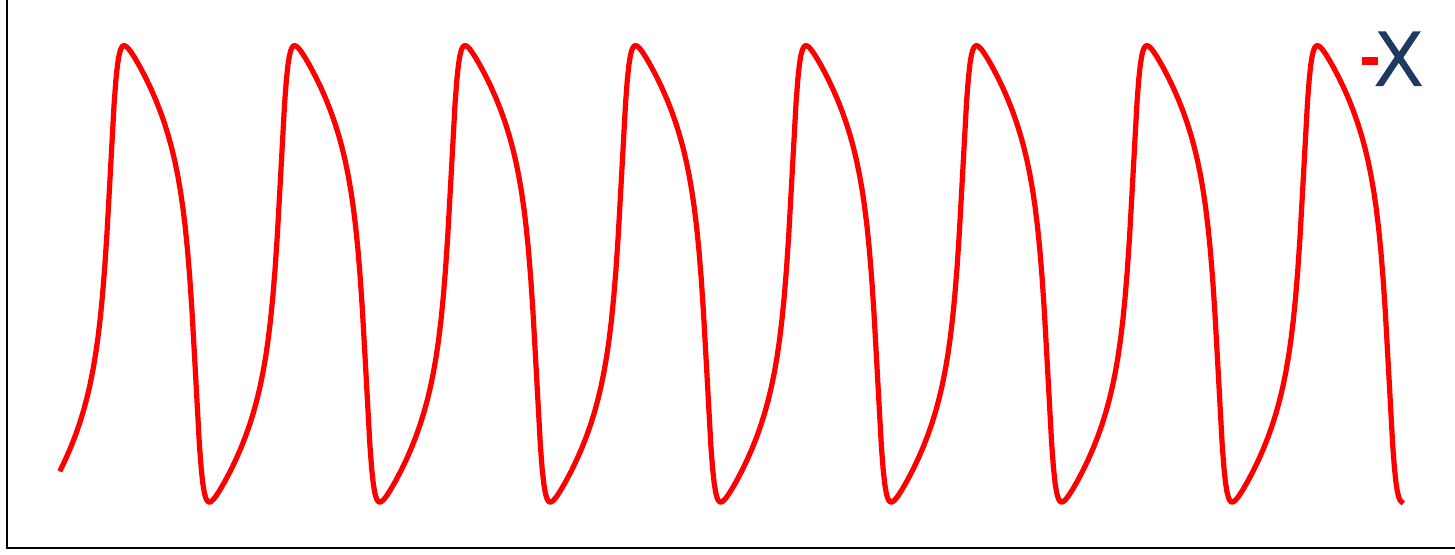}
    \hfill \
    \includegraphics[width=0.22\textwidth, trim=1.7cm 1.7cm 1.7cm 1.7cm, clip]{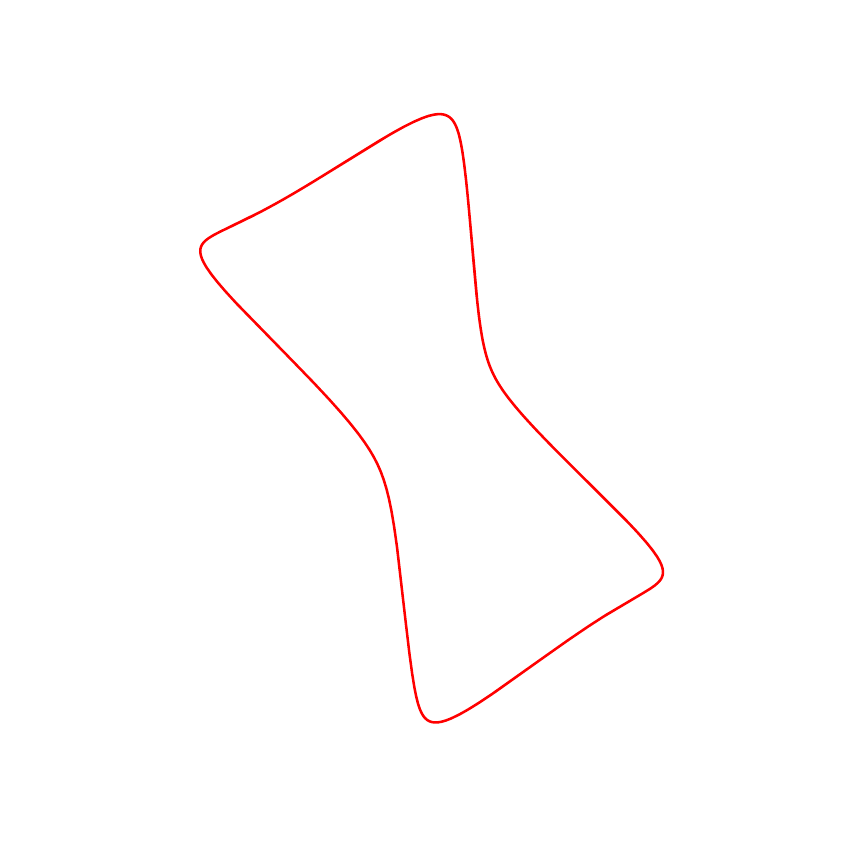}
    \hfill
    \includegraphics[width=0.27\textwidth]{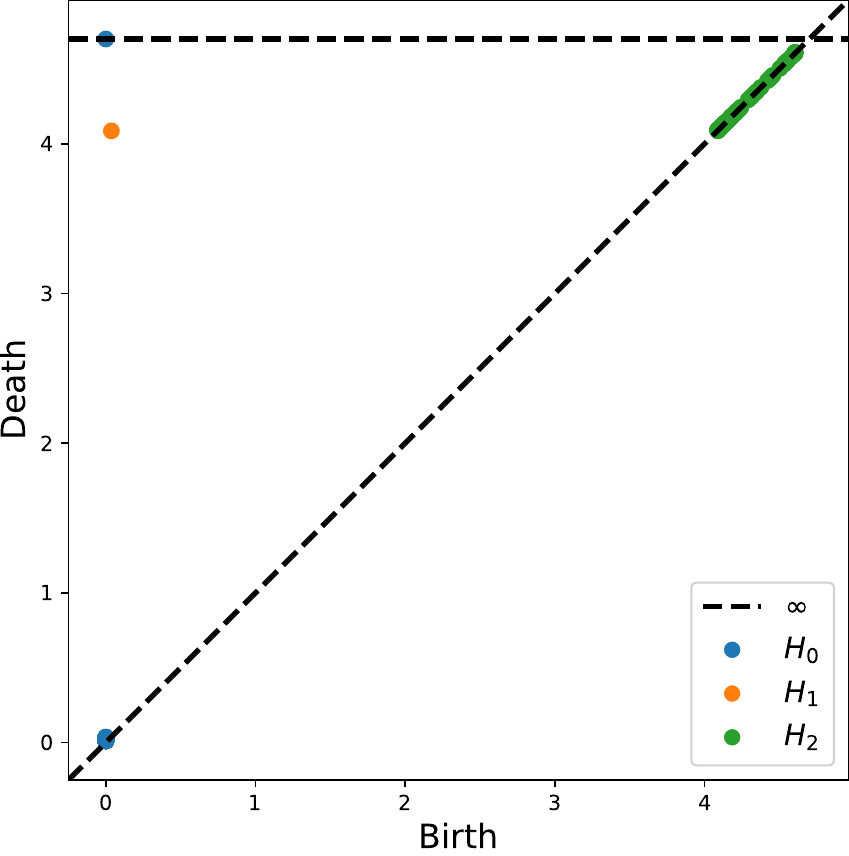}
    \caption{Left column: the \(x(t)\) coordinate of the solution to the radial system from Example \ref{ex:radial eq} (top row) and the van der Pol system from Example \ref{ex:van 2} (bottom row). Middle column: The sliding window embedding of $x(t)$ with appropriate parameters $d$ and $\tau$. Right column: The Rips persistence diagrams in dimensions 0,1, and 2 for the sliding window point cloud.}
    \label{fig:SW 2}
\end{figure}

\begin{theorem}
\label{thm:fourier-approx}
Let $f:\mathbb{R}\rightarrow\mathbb{C}$ be a quasiperiodic function with frequency vector $\mathbf{\omega}$ and parent function $F$.
For each  $\mathbf{k} \in \mathbb{Z}^N$,
the $\mathbf{k}$-th Fourier coefficient of $F$ can be computed as
\[
\hat{F}(\mathbf{k})
=
\lim_{\lambda \rightarrow \infty} \frac{2\pi}{\lambda}\int_{0}^{\lambda}
f(t)e^{- 2\pi i \langle \mathbf{k},t\mathbf{\omega} \rangle} dt,
\]
where $i=\sqrt{-1}$.
Furthermore, if
\[
S_{K}f(t)=  \sum_{\| \mathbf{k}\|_{\infty} \leq K } \hat{F}(\mathbf{k}) e^{2\pi i \langle \mathbf{k},t\mathbf{\omega} \rangle}
\;\;\; ,\;\;\;
K \in \N
\]
then  for each $ j,T \in \mathbb{N}$, $  \dgm^R_j(\SW_{d,\tau}f(T)) $ can be approximated in bottleneck distance by
\[
\dgm^R_j(\SW_{d,\tau}S_Kf(T))
\]
as $K\to \infty$.
The approximation is of order $O\big( K^{\frac{N}{2}-r} \big) $ when $|\hat{F}(\mathbf{k})| = O( \|\mathbf{k} \|^{-r}_2)$ and $r > N/2.$
\end{theorem}
Theorem \ref{thm:fourier-approx} shows that, in bottleneck distance, the Rips persistence diagrams of a quasiperiodic signal \(f\) can be arbitrarily closely approximated by those of its truncated Fourier series \(S_Kf(t)\).
Moreover,  the coefficients of \(S_Kf\) are determined by the spectrum of \(f\), which can be estimated via the Discrete Fourier Transform (implemented in the FFT algorithm).

This result justifies reducing general quasiperiodic functions to finite exponential sums of the form
\[
  f(t)=\sum_{j=1}^M c_j\,e^{2\pi i\,\omega_j t},
  \quad M\in\mathbb{N},\;c_j\in\mathbb{C}\smallsetminus\{0\},\;\omega_j\in\mathbb{R},
\]
where $\omega_1,\ldots, \omega_M$ are not necessarily $\Q$-linearly independent.
Set \(\tilde{c}_j=\sqrt{d+1}\,c_j\), choose an embedding dimension \(d\in\mathbb{N}\), a delay \(\tau > 0\) and let
\[
  A = \frac{1}{\sqrt{d+1}} \begin{pmatrix}
    1 & \cdots & 1 \\[3pt]
    e^{2\pi i\,\omega_1\tau} & \cdots & e^{2\pi i\,\omega_M\tau} \\[3pt]
    \vdots & \ddots & \vdots \\[3pt]
    e^{2\pi i\,\omega_1 d\tau} & \cdots & e^{2\pi i\,\omega_M d\tau}
  \end{pmatrix}
  \;\in\;\mathbb{C}^{(d+1)\times M},
  \quad
  v_t = \begin{pmatrix}
     \tilde{c}_1\,e^{2\pi i\,\omega_1 t} \\[3pt]
    \vdots \\[3pt]
    \tilde{c}_M\,e^{2\pi i\,\omega_M t}
  \end{pmatrix}
  \;\in\;\mathbb{C}^M.
\]
In this setting, the sliding window embedding can be written succinctly as
\[
  SW_{d,\tau}f(t)
  = A\,v_t,
\]
and according to \cite[Corollary 1.3]{gakhar2021sliding},
the set  $\{v_t\}_{t\in \R}$ is a dense subset of a space homeomorphic to the $N$-torus, for  $N \leq M$ equal to the dimension of $\mathrm{span}_\Q\{ \omega_1,\ldots, \omega_M\}$ as a $\Q$-vector space.
The multi-dimensional version of Kronecker's approximation theorem, as presented in \cite[Chapter 2, Theorem 2]{hlawka2012geometric}, allows one to take $t\in \N$ and retain density.
Furthermore, $d$ and $\tau$ can be chosen so that the same holds true for $\left \{SW_{d,\tau}f(t)\right\}_{t\in \N}$ \cite[Theorem 1.15]{gakhar2021sliding}.

With this in mind, the Cartesian product
\begin{equation}
\label{eq:Grid}
G_T= \prod_{j=1}^M
\left\{\tilde{c}_je^{2\pi  i\omega_jt} \mid t = 0, 1, \ldots, T\right\},
\end{equation}
offers a way to approximate $\{v_t\}_{t=0}^T $ in Hausdorff distance, and hence the resulting Rips persistence diagrams in bottleneck distance, by Theorem \ref{thm:Pstability}, as $T\to\infty$.
See also Figure \ref{fig:Torus}.
This is useful since the persistence diagrams of $G_T$ can be readily computed using the Persistent K\"{u}nneth Formula (Theorem \ref{kunneth formula}).
We will show in  Theorem \ref{AM theorem1} that this approach can   be used to approximate the Rips persistence diagrams of the sliding window embedding of $f$.

\subsection{Continued Fraction Expansion}
The continued-fraction expansion of a real number lies at the heart of our work. All definitions, propositions, and theorems in this section are classical and can be found in standard references—most notably \cite[Chapter 3]{fractions1963cd}, and for Theorem \ref{BA Lagrange} see, e.g., \cite[p.~22]{rockett1992continued}. As we will show, continued-fraction data enables us to infer persistent-homology information of sliding-window point clouds once the frequencies of the signal are estimated --- e.g. with the FFT.

\begin{definition}
\label{def:contFrac}
Let $\omega$ be a  real number. The continued fraction expansion of $\omega$ is given by $$ \omega=a_1+\frac{1}{a_2+\frac{1}{a_3+\frac{1}{a_{4}+\cdots}}} := [a_1,a_2,a_3,\cdots],$$ where $a_1 \in \mathbb{Z}$ and $a_k \in \mathbb{N}$, for $k >1.$
\end{definition}
Note that the expansion, which is obtained by repeated applications of the division algorithm, is infinite if and only if $\omega$ is irrational. From here onward, we let $\omega$ denote a positive irrational number.
\begin{definition}
Let $k \in \mathbb{N}.$ The $k$-th convergent of $\omega$ is given by
\[
\frac{p_k}{q_k}=[a_1,a_2,\cdots,a_k]
\in \Q.
\]
\end{definition}
The terms $p_k$ and $q_k$ play a special role in the theory, and can be obtained recursively as follows.
\begin{proposition}
\label{computation convergent}
The numerator $p_k$ and the denominator $q_k$ of the $k$-th convergent of $\omega$ satisfy the equations
\begin{equation*}
  \begin{split}
    p_k &=a_kp_{k-1}+p_{k-2},\\
    q_k &=a_kq_{k-1}+q_{k-2}. \\
  \end{split}
\end{equation*} for $k\geq 1$, where
\begin{equation*}
  \begin{split}
    p_0 &=1,\\
    q_0 &=0, \\
  \end{split}
  \ \ \ \ \ \ \ \
  \begin{split}
    p_{-1} &=0,\\
    q_{-1} &=1.
  \end{split}
\end{equation*}
\end{proposition}
Proposition \ref{computation convergent} provides a convenient computational method for obtaining the \(k\)-th convergent of \(\omega\), and implies the following result.
\begin{proposition}
$p_k$ and $q_k$ are relatively prime.
\end{proposition}
We can quantify the quality of the convergent \(p_k/q_k\) as an approximation to \(\omega\) by examining the absolute error
\[
  \bigl|\omega - \tfrac{p_k}{q_k}\bigr|.
\]
\begin{proposition}
Let $ k \geq 1$. Then $$\frac{1}{2q_kq_{k+1}} < \Big|\omega - \frac{p_k}{q_k} \Big| < \frac{1}{q_kq_{k+1}} < \frac{1}{q_k^2}.$$
\end{proposition}
Furthermore, this result allows us to establish the existence and uniqueness of infinite continued fraction expansions.
Indeed, this follows by noting that $ q_{k} < q_{k+1}$ for all $k \geq 1$,
and  justifies the equality sign in Definition \ref{def:contFrac}.
The next result further characterizes their convergence behavior.
\begin{proposition}
\label{odd even}
Let $k \geq 1$ be odd. Then $$ \frac{p_k}{q_k} < \omega <\frac{p_{k+1}}{q_{k+1}}.$$
\end{proposition}
\begin{example}
The results presented here can be used to argue that the golden ration $\varphi $ is the ``most" irrational number.
Indeed, by computing
\[
\varphi= \frac{1 + \sqrt{5}}{2}= 1+\frac{1}{1+\frac{1}{1+\frac{1}{1+\cdots}}} = [1,1,1,\cdots],
\]
one can show that this expression corresponds to the slowest possible rate of convergence.
The Klein diagram shown in  Figure \ref{fig:Klein 2} provides a geometric picture of this observation.
\end{example}
\begin{figure}[htb!]
    \centering
    \includegraphics[width=0.45\textwidth]{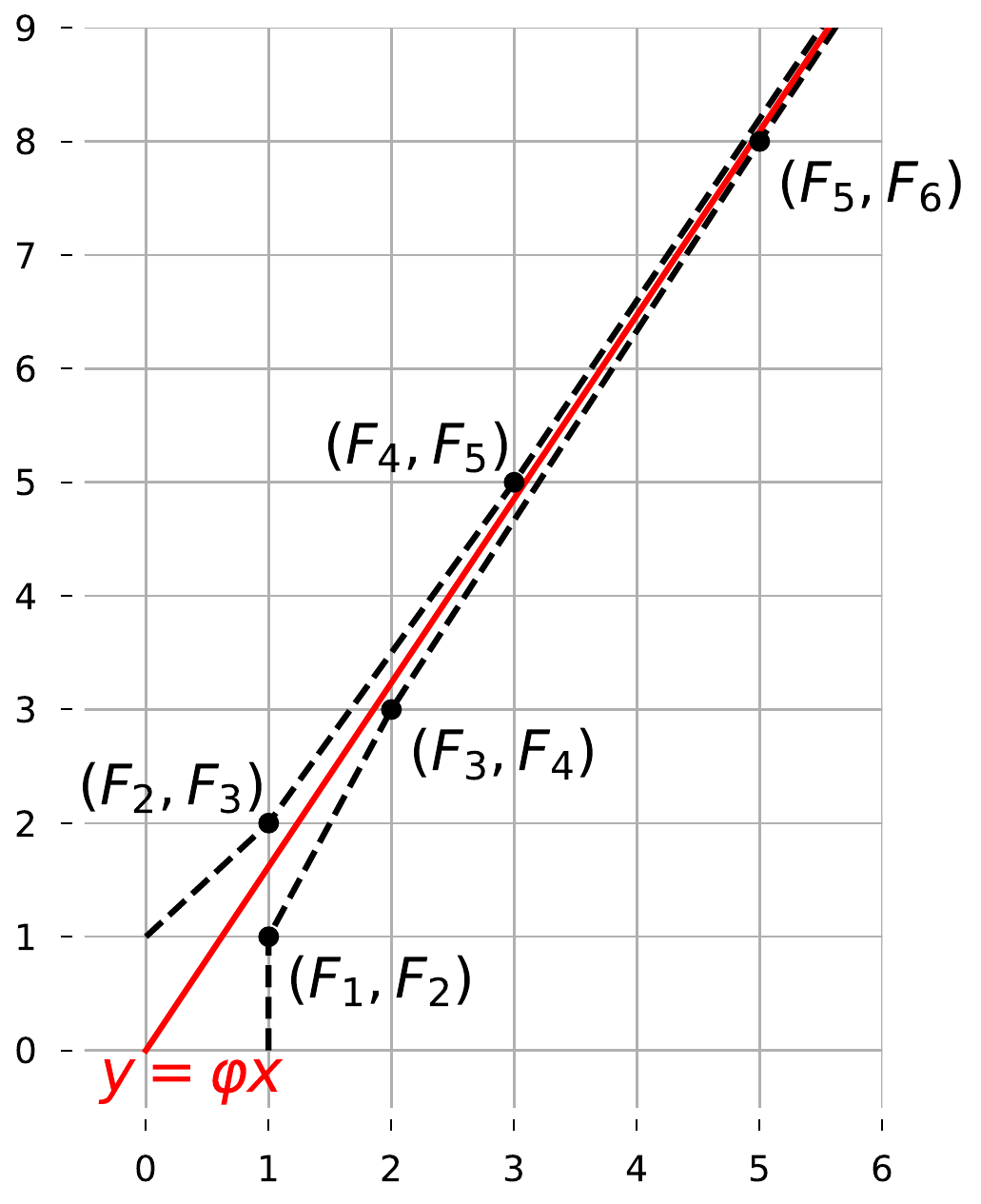}
    \caption{Klein diagram for the golden ratio \(\varphi = \frac{1 + \sqrt{5}}{2}\). Here \(F_k\) denotes the \(k\)-th Fibonacci number, which can be used to compute the $k$-th convergent $\frac{p_k}{q_k}$ of $\varphi$; specifically,   $\frac{p_k}{q_k} = \frac{F_k}{F_{k-1}}$.}
    \label{fig:Klein 2}
\end{figure}
We end this section by stating a critical property of the $k$-th convergent of $\omega$  attributed to Lagrange. Namely, that the $k$-th convergent provides a best rational approximation with bounded denominator to $\omega$ (see, \cite[p.~22]{rockett1992continued}).
\begin{theorem}
\label{BA Lagrange}
Let $a/b$ be different from $p_{k+1}/q_{k+1}$ with $ 0 < b \leq q_{k+1}.$ Then $$ |b \omega - a | \ \geq \ |q_k \omega -p_k | \ > \ |q_{k+1} \omega - p_{k+1}|.  $$
\end{theorem}

\subsection{The Three Gap Theorem Method}
\label{sec: 3G}
Thus far we have shown that exponential functions are fundamental in the study of quasiperiodic signals, and that their frequencies can be optimally approximated with their $k$-th convergents.
If we let
\[
  f(t) = e^{2\pi i \omega t},
\]
then viewed as a map \(f\colon\R\to S^1\subseteq \mathbb{C}\), \(f(t)\) parametrizes the unit circle. Fix a natural number \(T\in\mathbb{N}\) and consider the times \(t=0,1,\ldots,T\).  This yields the sample
\[
  \{\,f(t)\mid t=0,1,\dots,T\}
  = \left\{\,e^{2\pi i \omega t}\mid t=0,1,\dots,T\right\}
\]
which in flat coordinates (i.e., modulo 1), corresponds to
\[
  \{\,t\omega \bmod 1\mid t=0,1,\ldots,T\}
  \;\subseteq\;[0,1].
\]
\begin{definition}
For any real number \(x\in\mathbb{R}\), write
\[
  [x] \;=\; x \bmod 1
\]
for its equivalence class on the flat circle $[0,1]/0\sim 1$.
For \(T\in\mathbb{N}\) and \(\omega\in\mathbb{R}\smallsetminus\mathbb{Q}\) we let
\[
  S_{\omega,T}
  := \bigl\{[0],\,[\omega],\,[2\omega],\,\ldots,\,[T\omega]\bigr\}.
\]
\end{definition}

By identifying the endpoints of \([0,1]\), we view \(S_{\omega,T}\) as a finite sequence of points on the  circle (see Figure \ref{fig:3G circles}). To measure distances in this space we use the quotient metric
\[
  \bar{d}([x],[y])
  = \min\bigl\{|x-y|,\;|1-(y-x)|,\;|1-(x-y)|\bigr\},
\]
which turns \(\bigl(S_{\omega,T},\bar{d} \,\bigr)\) into a metric space.
When we write \(S_{\omega,T}\)  we will require  $\omega \in \R\smallsetminus \Q$ unless otherwise stated.
The asymptotic sampling properties of this set are well documented \cite{zhigljavskykronecker} and follow in large part from the Three Gap Theorem (Steinhaus conjecture) \cite{van1988three}, which we state below.

\begin{theorem}[\textbf{Three Gap Theorem}]
Let $T\in \mathbb{N}$ and $\omega \in \R\smallsetminus \Q$.
The points in $ S_{\omega,T}$ partition $[0,1]/0 \sim 1$ into $T+1$ intervals, such that their lengths take at least two and at most three different values.
If three distinct lengths occur --- denoted in increasing order  by \(\delta_A\), \(\delta_B\), and \(\delta_C\) --- then
\[
\delta_C \;=\;\delta_A + \delta_B.
\]
\end{theorem}
Importantly, the conclusion of the Three Gap Theorem --- that there are at most three distinct interval lengths --- holds for every irrational \(\omega\) and every sample size \(T\).
We illustrate the content of the Three Gap theorem in Figure \ref{fig:3G circles} for $\omega= \pi -1$ (left), $\omega = \sqrt{5}$ (right) and $T=17$.

\begin{figure}[!htb]
    \centering
    \includegraphics[width=0.30\textwidth]{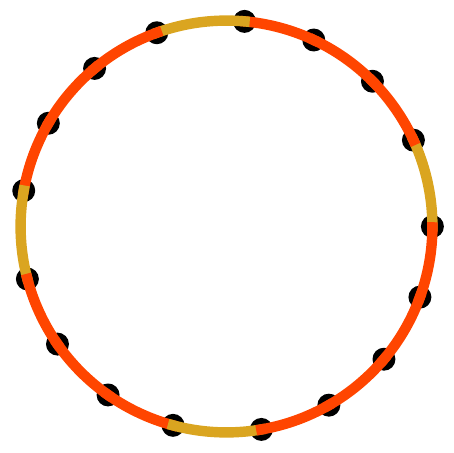} \ \ \ \ \ \ \ \ \ \ \ \ \ \ \ \ \ \ \ \
    \includegraphics[width=0.30\textwidth]{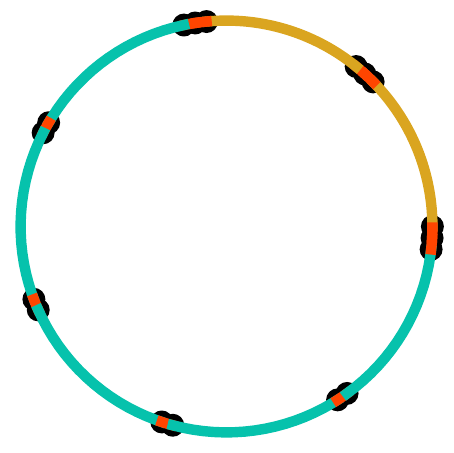}
    \caption{On the left, we depict \(S_{\pi-1,17}\) in black. Each gap size is shown in a different color; in this case, there are only two distinct gaps (red and yellow). On the right, we show \(S_{\sqrt{5},17}\) in black, which has three distinct gap sizes (green, red and yellow).}
    \label{fig:3G circles}
\end{figure}

The lengths $\delta_A,\delta_B$ and $\delta_C$ in the Three Gap theorem are in fact governed by the convergents of \(\omega\) \cite{leong2017sums}; see  Proposition \ref{three gaps computation} for explicit formulas for the gaps and their multiplicities.
This observation yields a complete description of pairwise distances between neighboring points in \(S_{\omega,T}\), and we will show in the next section how they can be used to compute the persistence diagrams of the Rips filtration $\mathcal{R}(S_\omega, \bar{d}\, )$; see Figure \ref{fig:3gap} for an example.
This leads to our approximation scheme for the persistence diagrams of sliding window point clouds from quasiperiodic functions: The Three Gap Theorem Method (3G), depicted  schematically in Figure \ref{fig:diagram}.

\begin{figure}[htb!]
    \centering
    \includegraphics[width=1\textwidth]{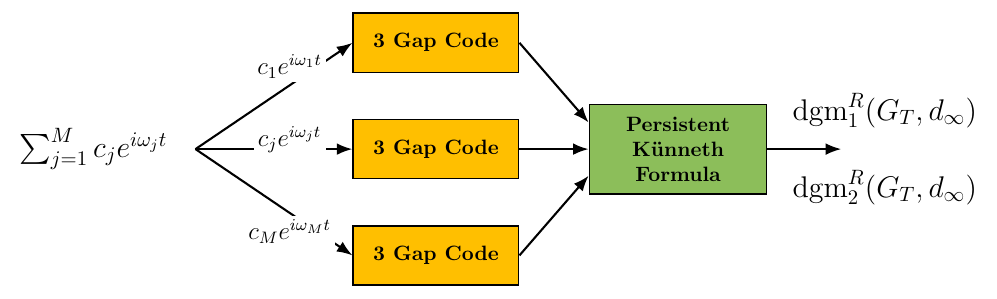}
    \hfill
    \caption{Schematic representation of the 3G method. Starting from a sum of exponentials approximating  a quasiperiodic function, we extract its frequency parameters via the FFT. We then apply the Three Gap construction independently to each frequency circle and compute the resulting persistence diagrams.  Finally, we assemble these individual diagrams using the Persistent K\"unneth Formula, yielding an approximation to the Rips persistence of the sliding window point cloud of the original signal.}
    \label{fig:diagram}

\end{figure}

In short, we first extract the independent frequencies $\omega_j$ of the input quasiperiodic signal; then, for each frequency $\omega_j$, we compute its continued fraction expansion and use its convergents to find the lengths $\delta_A, \delta_B$ (and $\delta_C = \delta_A + \delta_B$ when appropriate) in $(S_{\omega_j, T},\bar{d}\,)$, as well as their multiplicities (see Proposition \ref{three gaps computation}); the lengths and multiplicities of these gaps are used to derive the barcodes
$\bcd_n^R(S_{\omega_j, T}, \bar{d}\,)$ and persistence
diagrams $\dgm^R_n(S_{\omega_j ,T}, \bar{d}\,)$
 --- this is the content of Corollary \ref{0th bcd}, Corollary \ref{1st dgm}, and the output of Algorithm \ref{alg:3GC} below;
finally, we use the Persistent K\"unneth Theorem
to compute   $\dgm_n^R(G_T , d_\infty)$   --- see Equation (\ref{eq:Grid}) for the definition of $G_T$ and Theorem \ref{GT Kunneth} for the results.
The quality of the approximation of
$\dgm^R_n(\SW_{d,\tau} f(T), d_2)$
via $\dgm_n^R(G_T , d_\infty)$
is discussed in detail in Section \ref{sec:AM}, for $d_2$ the standard Euclidean distance in $\C^{d+1}$.
The relevant approximation results are given in
Theorem \ref{AM theorem1} and further clarified in Remark \ref{Ebound}.

\begin{algorithm}[!htb]
\caption{3 Gap Code} \label{alg:3GC}
\begin{algorithmic}[1]  
\State Input frequency $\omega_j$
\State Scale frequency $\omega_j=\omega_j/(2\pi)$
\State Obtain continued fraction expansion, C.F.E., of $\omega_j$
\State Use C.F.E. as shown in Proposition \ref{three gaps computation} to obtain gaps \& their multiplicities
\State Scale gaps as detailed in Theorem \ref{GT Kunneth} to obtain the barcode  of $\mathcal{R}(S_{\omega_j,T}, \bar{d}\,)$
\end{algorithmic}
\end{algorithm}

\section{Main Results}
\label{MR section}
We now move on to detailing the connection between the Three Gap Theorem and persistence diagrams.
The theorem allows us to compute persistence diagrams by leveraging information from the continued fraction expansion of the underlying frequencies estimated with the FFT of a quasiperiodic signal.
With this at hand, we show how, using the Persistent K\"unneth Formula, we can establish an approximation method for the persistence diagrams of sliding window embeddings from quasiperiodic functions.
Finally, we show how to obtain error bounds for our method.

\subsection{Persistence Diagrams}
Computing $\text{dgm}^R_0(S_{\omega,T},\bar{d})$ via the Three Gap Theorem
reduces to relating the lengths of the gaps to the convergents of $\omega$.
Although the connection between gap sizes and convergents is classical (see \cite{van1988three,beresnevich2017sums}), we include an alternative proof of the Three Gap Theorem that highlights how each continued fraction coefficient contributes to the $0$-th persistence diagram.
Throughout, let \(\{0,1,\ldots, T\}\subseteq\mathbb{N}\) denote the (discrete) set of time values, and let \(\omega\in\mathbb{R}\smallsetminus\mathbb{Q}\) be an irrational frequency.
We begin by finding the distance from a point in  $S_{\omega,T}$ to its nearest neighbor:
\[
Cl(j)=\min\big\{\bar{d}([j\omega],x) \ | \ x \in S_{\omega,T}\smallsetminus \{[j\omega]\}\big\}
\;\;\; ,\;\;\;
0 \leq j \le T
.
\]
\begin{lemma}
\label{closest point}
Let $ T\in \mathbb{N} $ and $ \omega \in \mathbb{R}\smallsetminus \mathbb{Q} $ with $i$-th convergent $\frac{p_i}{q_i}$. Fix $0 \leq j_0 \leq T$ and let $$ \bar{k}(j_0)=\max\{j\ |\ q_j \leq \max\{j_0,T-j_0\}\}. $$ If $0 \leq j_1 \leq T$ satisfies $$|j_0-j_1|=q_{\bar{k}(j_0)} $$ then $$ Cl(j_0)= \bar{d}([j_0\omega],[j_1\omega]).$$
\end{lemma}

\begin{proof}
Let us first show the result in the case $j_0=0.$ In this case, $Cl(j_0)= \bar{d}([j_0\omega],[j_1\omega])$ if and only if $$ [j_1\omega] = \min \{ [j\omega],1-[j\omega] \ | 1 \leq j \leq T \}.$$ The latter is equivalent to finding $p \in \mathbb{N}$ such that for all $ p_0,q_0\in \mathbb{N}$, with $q_0 \leq T$ $$ |j_1\omega - p | \leq |q_0\omega - p_0 |.$$ By Theorem \ref{BA Lagrange} this is satisfied by the convergents of $\omega,$ namely by $j_1=q_{\bar{k}(0)}$ and $p=p_{\bar{k}(0)}.$ This shows the result for the case $j_0=0.$ For the general case, we note that since $$ \bar{d}([j_0\omega],[j_1\omega]) = \bar{d}(0,[|j_0-j_1|\omega]) $$ we can apply the previous argument but now restricting $p_0 \leq \max\{j_0,T-j_0\} $ since $|j_0-j_1| \leq \max\{j_0,T-j_0\}.$ Similarly, the minimum is achieved when $|j_0-j_1|= q_{\bar{k}_j(j_0)}$, and the result follows.

\end{proof}

\begin{proposition}
\label{three gaps computation}
Let $T\in \mathbb{N}$ and $\omega\in \mathbb{R}\smallsetminus \mathbb{Q}$ with continued fraction expansion $[a_1,a_2,a_3,\cdots]$,
$i$-th convergent $\frac{p_i}{q_i}$,
and arrange the elements
of $S_{\omega,T} = \{[0], [\omega], \ldots, [T\omega] \}$
in increasing order as
\[
  \lbrack n_i\omega\rbrack < \lbrack n_{i+1}\omega\rbrack
  \quad\text{for }0\le i< T.
\]
Let $k,r,s$ be the unique integers for which
$$ q_k+q_{k-1} \leq T < q_k + q_{k+1} $$
 $$ T=rq_k+q_{k-1}+s, \ \ \ \ \ 1\leq r \leq a_{k+1}, \ \ \ \ \ 0 \leq s \leq q_k-1$$
and let $  D_i = q_i\,\omega - p_i$ for each $i$.
If
$$
\begin{aligned}
    \delta_A &= |D_k|, \quad & \delta_B &= |D_{k+1}|+(a_{k+1}-r)|D_k|, \quad & \delta_C &= \delta_A+\delta_B,
\end{aligned}
$$
and  \(\oplus\)  denotes addition modulo \(T+1\),
then for $0\leq i \leq T$
\[
\left\{
\bar{d}([n_i\omega], [n_{i\oplus1}\omega])
\right\}
\subseteq
\{\delta_A,\delta_B,\delta_C\}.
\]
Furthermore, denoting the multiplicity of gaps by $N_I=\#\{i \ | \ \bar{d}([n_i\omega], [n_{i\oplus1}\omega]) = \delta_I \},$ we have
$$
\begin{aligned}
    N_A &= T+1-q_k, \quad & N_B &= s+1, \quad & N_C &= q_k-s-1.
\end{aligned}
$$

\end{proposition}

\begin{proof}
We first establish the existence and uniqueness of $k,r,$ and $s.$ For all $i \geq 0$ $$q_i+q_{i-1} < q_i+q_{i+1}.$$ Thus, $T\in [q_k+q_{k-1},q_k+q_{k+1})$ for some $k,$ that this $k$ is unique follows by the monotonicity of $\{q_i+q_{i-1} \}_i.$
We also note that by assumption $T-q_{k-1}\geq q_k,$ thus by the division algorithm there exists unique $r,s$ such that $$T-q_{k-1} = rq_k +s, $$ where $1\leq r$ and $0 \leq s \leq q_k.$ On the other hand, since $T-q_{k-1} < q_k + q_{k+1} - q_{k-1} = (a_{k+1}+1)q_k$, by Proposition \ref{computation convergent}, we conclude $r \leq a_{k+1}.$ This shows the existence and uniqueness of $k,r,$ and $s$ as stated in the result.

For convenience of presentation in what follows, we define the notion of to the ``right" and to the ``left" of a point.
For $x,y \in S_{\omega,T}$, we say that $y$ is to the right of $x$ if $\bar{d}(x,y)$ is achieved by traversing from $x$ clockwise to $y$ in this representation. Similarly, achieving the distance while traversing counterclockwise is used to define to the left of. Since translation is an isometry under $\bar{d}$ we note $[i_1\omega]$ is to the right or to the left of $[i_2\omega]$ if and only if for all $j \geq 0,$ \ $[(i_1+j)\omega]$ is to the right or to the left, respectively, of $[(i_2+j)\omega].$ We make use of this fact repeatedly.

Let us now show that $$\{\bar{d}(0,[n_1\omega]),\bar{d}(0,[n_T\omega]) \} = \{ \delta_A,\delta_B\}.$$ This equality will apply to general points in $S_{\omega,T}$ by the isometry property of the translation map. In particular, it will allow us to compute $ N_A,N_B$, and find the explicit pairs that achieve the corresponding gap. We assume W.L.O.G. that $$[q_k\omega] \leq 1/2,$$ i.e., $[q_k\omega]$ is to the right of $0.$ By Proposition \ref{odd even} $[q_{k-1}\omega]$ and $[q_{k+1}\omega]$ are to the left of $0,$ i.e., $[q_k\omega] < [q_{k-1}\omega],[q_{k+1}\omega].$ Thus, by Lemma \ref{closest point} $n_1=q_k$ and thus $$\bar{d}(0,[n_1\omega]) = [q_k\omega] = \delta_A.$$ In fact, Lemma \ref{closest point} also shows $$ [q_k\omega] < [q_{k-1}\omega] < [q_{k+1}\omega].$$ Moreover, since $$\bar{d}([q_{k-1}\omega],[(q_{k-1}+q_k)\omega])=\bar{d}([q_{k}\omega],0)=[q_k\omega] $$ and, by Theorem \ref{BA Lagrange}, $$\bar{d}([q_{k}\omega],0) < \bar{d}([q_{k-1}\omega],0)$$ we conclude $$ [q_{k-1}\omega] < [(q_{k-1}+q_k)\omega].$$ Repeating this argument and noting $q_{k+1}=a_{k+1}q_k+q_{k-1}$ we can say that for $0 \leq i \leq a_{k+1}$ $$ [q_{k-1}\omega] \leq [(q_{k-1}+iq_k)\omega]  \leq [(q_{k-1}+(i+1)q_k)\omega] \leq  [q_{k+1}\omega].$$ Thus, $$ \bar{d}([(q_{k-1}+iq_k)\omega],0)= \bar{d}([q_{k+1}\omega],0) + \sum_{j=0}^{a_{k+1}-i-1} \bar{d}([(q_{k-1}+jq_k)\omega],[(q_{k-1}+(j+1)q_k)\omega])$$ $$ = |D_{k+1}| + (a_{k+1}-i)|D_k|.$$ By Lemma \ref{closest point} we can conclude $n_T = q_{k-1} + rq_{k} $ and the above computation shows $$ \bar{d}([(q_{k-1}+rq_k)\omega],0)= |D_{k+1}| + (a_{k+1}-r)|D_k| = \delta_B.$$ Thus, we have shown $$\{\bar{d}(0,[n_1\omega]),\bar{d}(0,[n_T\omega]) \} = \{ \delta_A,\delta_B\}.$$

Now consider $n_{j} \in [0,(r-1)q_k+q_{k-1}+s].$ Since $n_j+q_k \leq T$ and $$ \bar{d}([n_j\omega],[(n_j+q_k)\omega])= \bar{d}([q_k\omega],0)=[q_k\omega]=\delta_A $$ we conclude, by the same reasoning as with $n_j=0$, that $[(n_j+q_k)\omega]$ is the closest point to the right of $[n_j\omega],$ i.e., $n_{j\oplus1}=n_j+q_k.$ This shows there are $$(r-1)q_k+q_{k-1}+s+1=(rq_k+q_{k-1}+s)+1-q_k=T+1-q_k$$ different gaps of length $\delta_A$ i.e., $N_A=T+1-q_k.$

Similarly, for $n_{j} \in [rq_k+q_{k-1},T]$ we note $$n_{j}-(rq_k+q_{k-1}) \in [0,s]$$ and thus $$ \bar{d}([n_j\omega],[(n_j-(rq_k+q_{k-1}))\omega])= \bar{d}([(rq_k+q_{k-1})\omega],0)=|D_{k+1}| + (a_{k+1}-r)|D_k|=\delta_B.$$ Carrying over the arguments for the case  $n_j=rq_k+q_{k-1}$, we conclude $[(n_j-(rq_k+q_{k-1}))\omega]$ is the closest point to the left of $[n_j\omega],$ i.e., $n_{j-1}=n_{j}-(rq_k+q_{k-1}).$ Thus, there are $s+1$ different gaps of length $\delta_B$ i.e., $N_B=s+1.$

What remains to consider is the case $n_j\in ((r-1)q_k+q_{k-1}+s,rq_k+q_{k-1}).$ If $r > 1,$ we note $$ \bar{d}([n_j\omega],[(n_j-q_k)\omega])= \bar{d}([q_k\omega],0)=[q_k\omega]=\delta_A.$$ Thus, $n_{j-1}=n_j-q_k$. Since $n_j-q_k < (r-1)q_k+q_{k-1}+s $, we have already accounted for this gap. On the other hand, we note $$\bar{d}([n_j\omega],[n_{j\oplus1}\omega]) \neq \delta_A,\delta_B$$ since $n_j+q_k,n_j+(rq_k+q_{k-1}) > T.$ As a consequence of the inequality shown before
$$ [q_{k-1}\omega] \leq [(q_{k-1}+iq_k)\omega]  \leq [(q_{k-1}+(i+1)q_k)\omega] \leq  [q_{k+1}\omega],$$ we conclude the next smallest gap length is $$ |D_{k+1}| + (a_{k+1}-(r-1))|D_k| =  |D_{k+1}| + (a_{k+1}-r+1)|D_k| = \delta_A+\delta_B=\delta_C.$$ This gap is achieved when $n_{j+1}=n_j-((r-1)q_k+q_{k-1}).$ Thus, there are $q_k-s-1$ different gaps of length $\delta_C,$ i.e., $N_C=q_k-s-1.$

We have shown there are distinct gaps, distances between adjacent points, of length $\delta_A,\delta_B,$ and $\delta_C.$ Furthermore, since $$ N_A+N_B+N_C=T+1$$ we conclude all gaps have been accounted for, i.e.,
$$ \{\bar{d}([n_i\omega], [n_{i\oplus1}\omega]) \} \subseteq \{\delta_A,\delta_B,\delta_C\}.$$ The result follows.
\end{proof}

It is clear that the Three Gap Theorem follows from Proposition \ref{three gaps computation}. Furthermore, Proposition \ref{three gaps computation} details the relation between the length of the gaps and the $k$-th convergents of $\omega$. The next theorem translates what this means in terms of the $0$-th dimensional persistence diagram of $S_{\omega,T}$. We assume \(r < a_{k+1}\) and \(s+1 < p_k\) to streamline the exposition; the other parameter regimes lead to analogous conclusions:
\begin{itemize}
  \item \textbf{Case \(r \ge a_{k+1}\).} In this regime one shows that \(\delta_A > \delta_B\). In Theorem \ref{0th homolgy} below, one then swaps the labels \(\delta_A\) and \(\delta_B\) so that the combinatorial pattern of gap lengths remains identical.
  \item \textbf{Case \(s+1 \ge p_k\).} Here \(N_C = 0\), meaning no third gap appears. Equivalently, the final merge of adjacent points occurs at \(\epsilon = \max\{\delta_A,\delta_B\}\), which then plays the role of \(\delta_C\) in the theorem’s description.
\end{itemize}
In each scenario, relabeling or omitting the appropriate gap lengths shows that the proof of Theorem \ref{0th homolgy} carries over without change.
\begin{theorem}
\label{0th homolgy}
Let $T\in \mathbb{N}$ and $\omega\in \mathbb{R}\smallsetminus \mathbb{Q}$ with continued fraction expansion $[a_1,a_2,a_3,\cdots]$, $i$-th convergent $\frac{p_i}{q_i}$, and let $k,r,s$ be the unique integers for which
$$ q_k+q_{k-1} \leq T < q_k + q_{k+1} $$
and $$ T=rq_k+q_{k-1}+s, \ \ \ \ \ 1\leq r \leq a_{k+1}, \ \ \ \ \ 0 \leq s \leq q_k-1. $$ Let
\[
  D_i = q_i\,\omega - p_i,
\]
and suppose
\[
  r < a_{k+1}
  \quad\text{and}\quad
  s + 1 < p_k.
\]
Then
\[  H_0(R_\epsilon(S_{\omega,T},\bar{d});\mathbb{F}) = \left\{
\begin{array}{ll}
      \mathbb{F}^{^{T+1}} & 0 \leq \epsilon < |D_k| \\ \\
       \mathbb{F}^{^{q_k}} & |D_k| \leq \epsilon < |D_{k+1}|+(a_{k+1}-r)|D_k| \\ \\
     \mathbb{F}^{^{q_k-s-1}} & |D_{k+1}|+(a_{k+1}-r)|D_k| \leq \epsilon  \\ & < |D_{k+1}|+(a_{k+1}-r+1)|D_k| \\ \\
    \mathbb{F} & |D_{^{k+1}}|+(a_{k+1}-r+1)|D_k| \leq \epsilon  \\
\end{array}
\right. \]

\end{theorem}
\begin{proof}
Let $\delta_I,N_I,$ where $I\in \{A,B,C\},$ be as defined in Proposition  \ref{three gaps computation}. Since $ r < a_{k+1} $ we note $\bar{k}(T)=k$ and thus by Lemma \ref{closest point}  $$\delta_A < \delta_B.$$ Furthermore, since $ s+1 < q_k,$ there are three possible lengths for the gaps i.e., $N_C > 0.$

At $\epsilon=0$ $$H_0(R_0(S_{\omega,T},\bar{d});\mathbb{F})=\mathbb{F}^{^{T+1}}$$ by definition. We can interpret this as connected components that are separated by gaps, namely $ N_A+N_B+N_C=T+1$ of them. Since the Rips complex connects two points when $\epsilon$ is greater or equal to their distance, the first time points will be connected is at $\epsilon = \delta_A.$ Now, the newly formed connected components will be separated from each other only at the gaps of length $\delta_B$ and $\delta_C$ i.e., $$H_0(R_{\delta_A}(S_{\omega,T},\bar{d});\mathbb{F})=\mathbb{F}^{^{N_B+N_C}}=\mathbb{F}^{^{q_k}}.$$ Similarly, the next time we connect points is at $\epsilon=\delta_B,$ now the remaining connected components will be separated by only $N_C$ gaps and thus
$$H_0(R_{\delta_B}(S_{\omega,T},\bar{d});\mathbb{F})=\mathbb{F}^{^{N_C}}=\mathbb{F}^{^{q_k-s-1}}.$$ The last time we connect points is at $\delta_C,$ since this means all adjacent points are connected we conclude there is only one connected component i.e.,
$$H_0(R_{\delta_C}(S_{\omega,T},\bar{d});\mathbb{F})=\mathbb{F}.$$ The result follows.
\end{proof}
\begin{corollary}
\label{0th bcd}
Let $T\in \mathbb{N}$ and $\omega\in \mathbb{R}\smallsetminus \mathbb{Q}$ with continued fraction expansion $[a_1,a_2,a_3,\cdots]$, $i$-th convergent $\frac{p_i}{q_i}$, and let $k,r,s$ be the unique integers for which
$$ q_k+q_{k-1} \leq T < q_k + q_{k+1} $$
and $$ T=rq_k+q_{k-1}+s, \ \ \ \ \ 1\leq r \leq a_{k+1}, \ \ \ \ \ 0 \leq s \leq q_k-1. $$ Let
\[
  D_i = q_i\,\omega - p_i,
\]
and suppose
\[
  r < a_{k+1}
  \quad\text{and}\quad
  s + 1 < p_k.
\]
If
$$
\begin{aligned}
    \delta_A &= |D_k|, \quad & \delta_B &= |D_{k+1}|+(a_{k+1}-r)|D_k|, \quad & \delta_C &= \delta_A+\delta_B,
\end{aligned}
$$
then
$$ \dgm^R_0(S_{\omega,T},\bar{d})=
\left\{(0,\delta_A)^{T+1-q_k}, (0,\delta_B)^{s+1}, (0,\delta_C)^{q_k-s-2}, (0,\infty)  \right\} $$
where the upper indices indicate the multiplicity of each point in the diagram.

\end{corollary}
\begin{proof}
For $\epsilon \leq \epsilon'$ let
\[
T_{\epsilon,\epsilon'}:
H_0(R_\epsilon(S_{\omega, T}, \bar{d}\,); \F)
\longrightarrow
H_0(R_{\epsilon'}(S_{\omega, T}, \bar{d}\,); \F)
\]
denote the linear transformation induced by the inclusion of Rips complexes,
and let $\beta_0^{\epsilon, \epsilon'}$ be its rank.
Since $S_{\omega, T}$ is the vertex set of $R_\epsilon(S_{\omega, T}, \bar{d}\,)$ for all $\epsilon \geq 0$,
and a basis for $H_0(K ;\F)$ is given by choosing a vertex in each path-connected component of the simplicial complex $K$, then $T_{\epsilon, \epsilon'}$ is surjective whenever  $0 \leq \epsilon \leq \epsilon'$, and the zero map for $\epsilon < 0$.
Hence $\beta_0^{\epsilon ,\epsilon'}$ is equal to 0 if $\epsilon < 0$, and equal to
$\dim_{\mathbb{F}}(H_0(R_{\epsilon'}(S_{\omega,T},\bar{d}\,);\mathbb{F}))$ if $0 \leq \epsilon \leq \epsilon'$.
The result follows from applying Theorem \ref{0th bcd} to compute $\dim_{\mathbb{F}}(H_0(R_{\epsilon'}(S_{\omega,T},\bar{d}\,);\mathbb{F}))$, and combining the result with the formula from Equation \ref{eq:BettiMult} (Remark \ref{rmk:diagramAltSum}).
\end{proof}

\begin{example}
Let us consider the parameter value $\sqrt{5}$ shown in Figure \ref{fig:3gap} (top row, left column).
To compute the lengths of the gaps in $S_{\sqrt{5},16}$, we follow Proposition \ref{three gaps computation}. We note that $$ \sqrt{5}=[2,4,4,4,4,4,\dots] $$ and thus using Theorem \ref{computation convergent} we obtain $q_1=1, \ q_2=4, \ q_3=17,$ and $q_4=72.$ Since $$ q_1+q_2 \leq 16 < q_2 + q_3 $$ we conclude $k=2.$ One can readily check that $$16 = 3 q_2 +q_1 + 3$$ is the desired decomposition, i.e. $r=3$ and $s=3$.
Hence we have $N_A = 16 + 1 - 4 =13$ gaps of length $\delta_A=|4 \sqrt{5} - 9| \approx 0.05573,$ $N_B = 3 +1 = 4 $ gaps of length $\delta_B=|17 \sqrt{5} - 38| \ + \ (4-3)|4 \sqrt{5} - 9| \approx 0.06888,$ and $N_C=4-3-1=0$ gaps of length $\delta_C=|17 \sqrt{5} - 38|+(4-3+1)|4 \sqrt{5} - 9| \approx 0.12461.$ Thus, by using  Corollary \ref{0th bcd} in the case of two gaps, we have
\[
\dgm^R_0(S_{\sqrt{5},16},\bar{d}) = \{(0,0.05573)^{13},(0,0.06888)^{3},(0, \infty) \}.
\]
\end{example}

\begin{example}
Similarly, we consider the set $S_{\pi-1,16}$ shown in Figure \ref{fig:3gap} (bottom row, left column).
In this case $$\pi-1=[2,7,15,292,1,1,\dots] $$   $p_1=1, \ p_2=7 \ p_3=106,$ and $p_4=113.$
Thus, $k=2$ since $$ p_1 + p_2 \leq 16 < p_2 + p_3.$$ Furthermore, we note that the desired decomposition is given by $$16 = 2p_2 + p_1 + 1,$$ i.e. $r=2$ and $s=1.$
Thus,  we have $N_A = 16 + 1 - 7 =10$ gaps of length $\delta_A=|7 (\pi-1) - 15| \approx 0.00885,$
$N_B = 2 $ gaps of length $\delta_B=|106 (\pi-1) - 227| \ + \ (15-2)|7 (\pi-1) - 15| \approx 0.12389,$ and $N_C=7 -1 -1=5$ gaps of length $\delta_C=|106 (\pi-1) - 227| \ + \ (15-2+1)|7 (\pi-1) - 15| \approx 0.13274.$ Moreover, by Corollary \ref{0th bcd} $$ \dgm^R_0(S_{\pi-1,16},\bar{d}) = \{(0,0.00885)^{10},(0,0.12389)^{2},(0,0.13274)^{4},(0, \infty) \}.$$
\end{example}

\begin{figure}[!htb]
    \centering
    \includegraphics[width=0.25\textwidth]{3GC1fPDF.pdf}
    \hfill
    \includegraphics[width=0.35\textwidth]{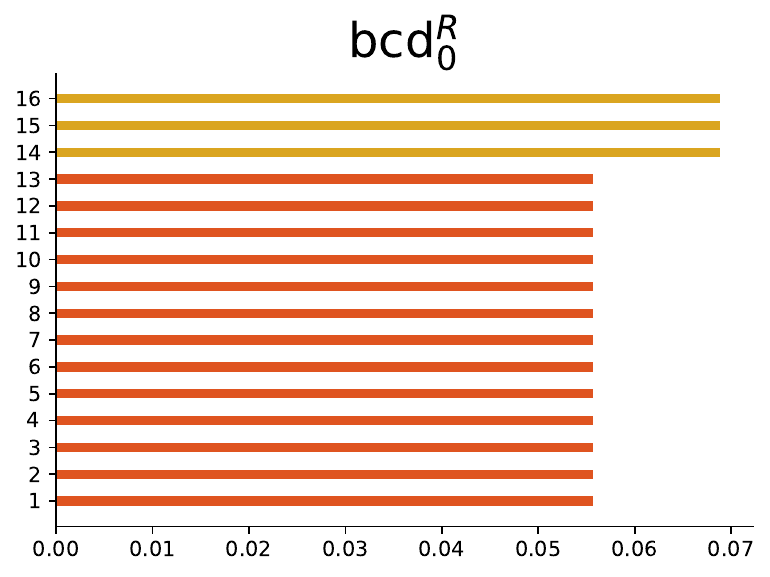}
    \hfill
    \includegraphics[width=0.37\textwidth]{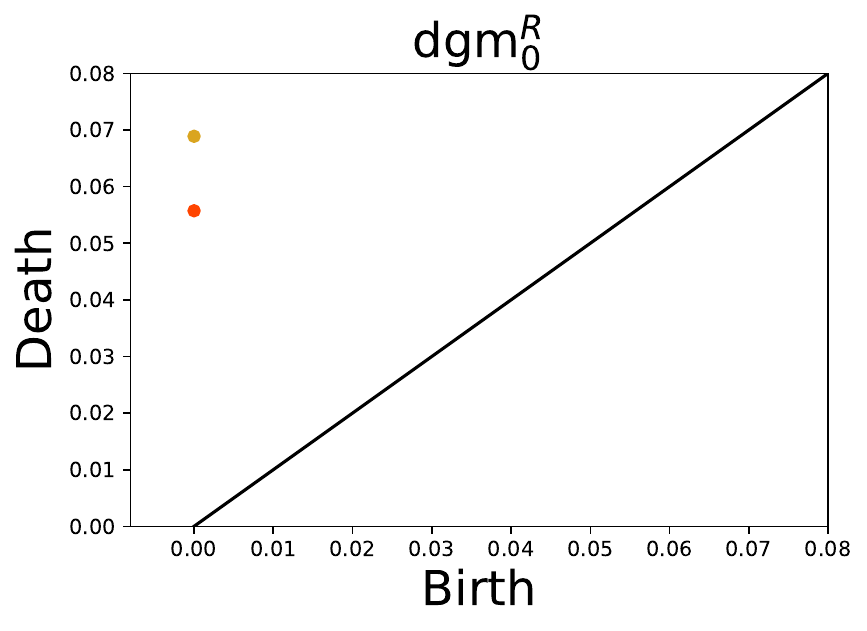}
    \includegraphics[width=0.25\textwidth]{3GC2fPDF.pdf}
    \hfill
    \includegraphics[width=0.35\textwidth]{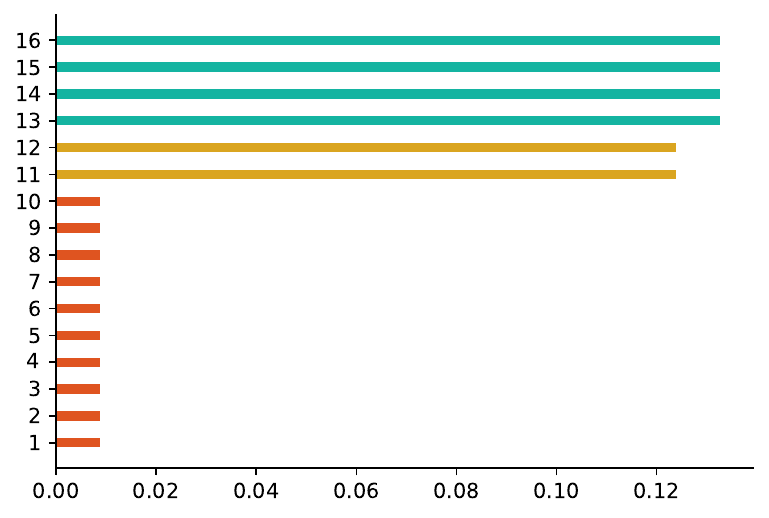}
    \hfill
    \includegraphics[width=0.37\textwidth]{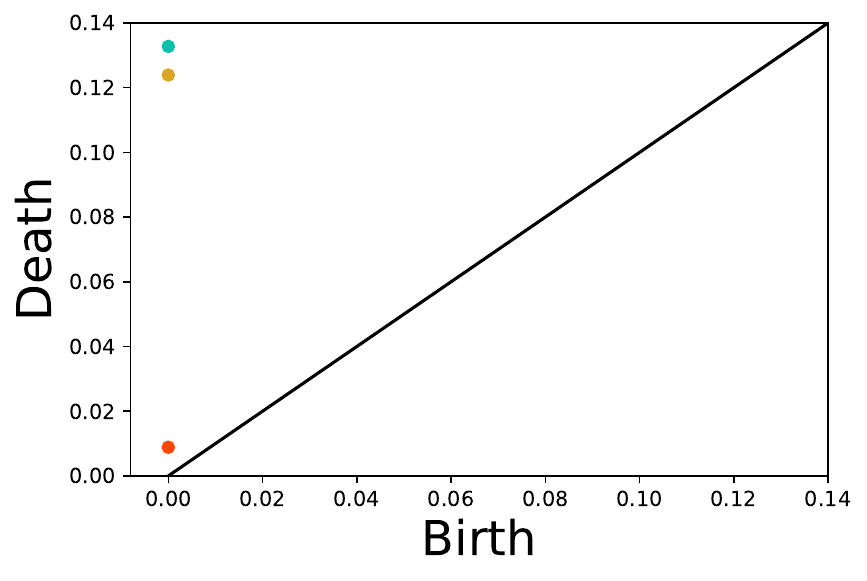}
    \caption{We illustrate above how the 0-th dimensional persistent homology can be computed using the Three Gap Theorem. At $\epsilon=0$, all points in $(S_{\omega,T},\bar{d})$ count as a basis in $H_0(R_{\epsilon}(S_{\omega,T},\bar{d});\mathbb{F})$ (birth time). At $\epsilon = \delta_A$, points get connected for the first time. The three possible gaps are the only instances when components are connected (death time). Top row: The set $S_{\sqrt{5},16}$ is illustrated with the two gaps it creates. These can be traced in the barcodes and persistent diagram as shown. Bottom row: The set $S_{\pi-1,16}$ creates three gaps.}
\label{fig:3gap}
\end{figure}

We also compute
\(\mathrm{dgm}^R_1\bigl(S_{\omega,T},\bar d\bigr)\).  In Theorem \ref{1st homology} we assume
\[
\begin{gathered}
  s + 1 < q_k,\\
  \delta_C < \tfrac13
\end{gathered}
\]
to simplify the presentation. If instead \(s+1 \ge q_k\), then \(N_C=0\) and only the two gap lengths \(\delta_A,\delta_B\) occur—so one may simply set \(\max\{\delta_A,\delta_B\}\) in place of \(\delta_C\). Likewise, if the largest gap \(\delta_M\) satisfies \(\delta_M\ge\tfrac13\), then by Lemma \ref{A5} we have \(\lambda\le\delta_M\), forcing every potential 1-cycle to be $0$ in \(H_1\), i.e. a trivial case. We denote addition modulo \(T+1\) by \(\oplus\). \\

\begin{theorem}
\label{1st homology}
Let $T\in \mathbb{N}$ and $\omega\in \mathbb{R}\smallsetminus \mathbb{Q}$ with continued fraction expansion $[a_1,a_2,a_3,\cdots]$, $i$-th convergent $\frac{p_i}{q_i}$, and let $k,r,s$ be the unique integers for which
$$ q_k+q_{k-1} \leq T < q_k + q_{k+1} $$
and $$ T=rq_k+q_{k-1}+s, \ \ \ \ \ 1\leq r \leq a_{k+1}, \ \ \ \ \ 0 \leq s \leq q_k-1. $$ Let
\[
  D_i = q_i\,\omega - p_i,
\]
and assume
\[
  s + 1 < q_k,
  \quad\text{and}\quad
  \lvert D_{k+1}\rvert + (a_{k+1} - r + 1)\,\lvert D_k\rvert < \tfrac13.
\]
Let $ \Gamma $ be the set containing the pairs $(x,y), \ x,y\in S_{\omega,T} \ \text{and} \ x<y$, for which there exists a $z \in ([0,x)\cup (y,1)) \cap S_{\omega,T}$ such that $ 1-\bar{d}(x,y) \leq 2\max\{\bar{d}(y,z),\bar{d}(x,z)\} \leq 2\bar{d}(x,y)$. If $$ \lambda = \min \{\bar{d}(x,y) | \ (x,y) \in \Gamma\},$$ then
\[   H_1(R_{\epsilon}(S_{\omega,T},\bar{d});\mathbb{F}) = \left\{
\begin{array}{ll}
     \mathbb{F} & |D_{k+1}|+(a_{k+1}-r+1)|D_k|\leq \epsilon < \lambda \\ \\
      0 & else \\
\end{array} .
\right. \]
\end{theorem}

\begin{proof}
Let \(\{\,x_i\}_{i=0}^T\) be the points of \(S_{\omega,T}\) in ascending order, and consider the 1-chain
\[
  \sigma \;=\; \sum_{i=0}^{T} \{\,x_i,\;x_{\,i\oplus1}\},
\]
where \(i\oplus1\) denotes \((i+1)\bmod(T+1)\). By Proposition \ref{three gaps computation}, the maximum distance between consecutive points is
\[
  \delta_C \;=\; \lvert D_{k+1}\rvert \;+\;(a_{k+1}-r+1)\,\lvert D_k\rvert.
\]
Hence
\[
  \sigma \;\in\; C_1\bigl(R_{\epsilon}(S_{\omega,T},\bar d)\bigr)
  \quad\text{iff}\quad
  \epsilon \;\ge\;\delta_C. \]
One checks immediately that \(\partial\sigma = 0\). Lemma \ref{A4} then shows any non-boundary 1-cycle in \[C_1(R_{\delta_C}(S_{\omega,T},\bar d))\] is homologous to \(\sigma\). Consequently \([\sigma]\) generates \(H_1(R_{\delta_C}(S_{\omega,T},\bar d);\mathbb{F})\).

To show that \(\sigma\) is not a boundary, it suffices to prove there is no 2-simplex
\(\tau \in C_2\bigl(R_{\delta_C}(S_{\omega,T},\bar d)\bigr)\)
with \(\partial\tau = \sigma\). We argue by contradiction. Suppose there exists
\(\tau \in C_2\bigl(R_{\delta_C}(S_{\omega,T},\bar d)\bigr)\)
such that \(\partial\tau = \sigma.\)
Let us first show we may assume, without loss of generality, that
\[
  \bar d(0,x_1) = \delta_C.
\]
To see this, we construct an explicit isometry on \((S_{\omega,T},\bar d)\) which reindexes the points accordingly. By hypothesis, there is an index \(i_0\) with
\[
  \bar d\bigl(x_{i_0},x_{\,i_0\oplus1}\bigr)=\delta_C.
\]
Define
\[
  \phi(x_i)
  =
  \begin{cases}
    x_i - x_{i_0}, & i \ge i_0,\\[6pt]
    1 - \bigl(x_{i_0} - x_i\bigr), & i < i_0.
  \end{cases}
\]
A direct check shows
\[
  \bar d\bigl(\phi(x_i),\phi(x_j)\bigr)
  = \bar d(x_i,x_j)
  \quad\text{for all }i,j,
\]
so \(\phi\) is an isometry. By Theorem \ref{isomorphic diagrams}, it induces identical persistence diagrams. Finally, set
\[
  y_j = \phi\bigl(x_{\,i_0\oplus j}\bigr),
  \quad j=0,1,\dots,T,
\]
so that
\[
  0 = y_0 < y_1 < \cdots < y_T < 1
  \quad\text{and}\quad
  \bar d(y_0,y_1)=\delta_C.
\]
This relabeled sequence \(\{y_j\}\) therefore has the desired property; we thereafter write \(x_j := y_j\) for all \(j\).

Now, since \(\partial\tau = \sigma\), \(\tau\) must contain a 2-simplex of the form \(\{0,x_1,x_{i_0}\}\) for some \(i_0>1\). If \(x_{i_0}\le\tfrac12\), then
\[
  \bar d(0,x_{i_0})
  = \bar d(0,x_1) + \bar d(x_1,x_{i_0})
  > \delta_C.
\]
If instead \(x_{i_0}>\tfrac12\) and \(\bar d(x_1,x_{i_0})\le\delta_C\), then
\[
  \bar d(0,x_{i_0})
  = \min\bigl\{\bar d(0,x_1)+\bar d(x_1,x_{i_0}),\,
                1 - [\bar d(0,x_1)+\bar d(x_1,x_{i_0})]\bigr\}
  > \delta_C,
\]
since \(\delta_C<\tfrac13\). In either case \(\{0,x_1,x_{i_0}\}\notin\tau\), contradicting \(\partial\tau=\sigma\). Therefore no such \(\tau\) exists, and \(\sigma\) is not a boundary.

To conclude the proof, it suffices to show that for any
\(\epsilon\) with \(\delta_C<\epsilon\le\lambda\), a 2-simplex
\[\tau\in C_2\bigl(R_\epsilon(S_{\omega,T},\bar d)\bigr)\] satisfying
\(\partial\tau=\sigma\) appears precisely at \(\epsilon=\lambda\).
In particular, no such \(\tau\) exists when \(\epsilon<\lambda\), but at
\(\epsilon=\lambda\) one does.  Hence for all
\(\epsilon\ge\lambda\), the class \([\sigma]\) is trivial in
\(
  H_1\bigl(R_\epsilon(S_{\omega,T},\bar d);\mathbb{F}\bigr).
\)

We now demonstrate by explicit construction that at \(\epsilon = \lambda\) there is a 2-chain
\(\tau\in C_2\bigl(R_{\lambda}(S_{\omega,T},\bar d)\bigr)\)
with \(\partial\tau=\sigma\). Choose \((x_{n_1},x_{n_2})\in\Gamma\) so that
\(\lambda = \bar d(x_{n_1},x_{n_2})\), and let \(x_{n_3}\) be the third point paired with \((x_{n_1},x_{n_2})\). By assumption
\[
\max\bigl\{\bar d(x_{n_2},x_{n_3}),\,\bar d(x_{n_1},x_{n_3})\bigr\} \leq \bar d(x_{n_2},x_{n_3}) = \lambda.
\]
Hence \(\{x_{n_1},x_{n_2},x_{n_3}\}\) lies in \(R_{\lambda}(S_{\omega,T},\bar d)\).

Next define the 2-chain
\[
  \tau_1 = \sum_{i=n_1}^{\,n_2-1} \{x_i,\;x_{i+1},\;x_{n_2}\}.
\]
Because for $n_1 \leq i < n_2$, $$\bar{d}(x_i,x_{i+1}) < \bar{d}(x_{n_1},x_{n_2})$$ $\tau_1 \in C_2(R_{\lambda}(S_{\omega,T},\bar{d})).$

To define the analogous chains \(\tau_2\) and \(\tau_3\), we use the fact that \(S_{\omega,T}\) can be identified with the circle.  Let \(\{x^2_i\}_{i=1}^{A_2}\) be the sequence of points obtained by starting at \(x_{n_2}\) and moving clockwise until \(x_{n_3}\), inclusive—so \(x^2_1 = x_{n_2}\) and \(x^2_{A_2} = x_{n_3}\).  We then set
\[
  \tau_2 = \sum_{i=1}^{A_2-1} \{x^2_i,\;x^2_{i+1},\;x_{n_3}\},
\]
interpreting \(\tau_2=0\) if \(A_2=2\).

Similarly, let \(\{x^3_i\}_{i=1}^{A_3}\) be the sequence of points starting at \(x_{n_3}\) and moving clockwise until \(x_{n_1}\), with \(x^3_1 = x_{n_3}\) and \(x^3_{A_3} = x_{n_1}\).  We define
\[
  \tau_3 = \sum_{i=1}^{A_3-1} \{x^3_i,\;x^3_{i+1},\;x_{n_1}\},
\]
again taking \(\tau_3=0\) if \(A_3=2\).  In all cases, each \(\tau_j\) lies in \(C_2\bigl(R_{\lambda}(S_{\omega,T},\bar d)\bigr)\).

Finally set
\[
  \tau = \tau_1 + \tau_2 + \tau_3 + \{x_{n_1},x_{n_2},x_{n_3}\}.
\]
By construction \(\tau\in C_2\bigl(R_{\lambda}(S_{\omega,T},\bar d)\bigr)\), and one checks directly that \(\partial\tau=\sigma\).  This completes the existence argument at \(\epsilon=\lambda\).

To complete the proof, let \(\epsilon\in(\delta_C,\lambda)\).  If one assumes the existence of a 2-simplex
\(\tau_0\in C_2\bigl(R_{\epsilon}(S_{\omega,T},\bar d)\bigr)\)
with \(\partial\tau_0=\sigma\), then by considering
\[
  \max\bigl\{\bar d(x,y),\,\bar d(y,z),\,\bar d(z,x)\;\big|\;\{x,y,z\}\in\tau_0\bigr\}
\]
one obtains a contradiction to the minimality of \(\lambda\); see Lemma \ref{A5}.

\end{proof}

\begin{corollary}
\label{1st dgm}
Let $T\in \mathbb{N}$ and $\omega\in \mathbb{R}\smallsetminus \mathbb{Q}$ with continued fraction expansion $[a_1,a_2,a_3,\cdots]$, $i$-th convergent $\frac{p_i}{q_i}$, and let $k,r,s$ be the unique integers for which
$$ q_k+q_{k-1} \leq T < q_k + q_{k+1} $$
and $$ T=rq_k+q_{k-1}+s, \ \ \ \ \ 1\leq r \leq a_{k+1}, \ \ \ \ \ 0 \leq s \leq q_k-1. $$ Let
\[
  D_i = q_i\,\omega - p_i,
\]
and assume
\[
  s + 1 < q_k,
  \quad\text{and}\quad
  \lvert D_{k+1}\rvert + (a_{k+1} - r + 1)\,\lvert D_k\rvert < \tfrac13.
\]
Let $ \Gamma $ be the set containing the pairs $(x,y), \ x,y\in S_{\omega,T} \ \text{and} \ x<y$, for which there exists a $z \in ([0,x)\cup (y,1)) \cap S_{\omega,T}$ such that $ 1-\bar{d}(x,y) \leq 2\max\{\bar{d}(y,z),\bar{d}(x,z)\} \leq 2\bar{d}(x,y)$. If $$ \lambda = \min \{\bar{d}(x,y) | \ (x,y) \in \Gamma\},$$ then
$$ \dgm^R_1(S_{\omega,T},\bar{d})=
\{(\lvert D_{k+1}\rvert + (a_{k+1} - r + 1)\,\lvert D_k\rvert,\lambda)\}.$$
\end{corollary}
\begin{proof} The 1-cycle $\sigma$ constructed in the proof of Theorem \ref{1st homology} provides
a persistent generator for
\[
H_1(R_\epsilon(S_{\omega, T}, \bar{d}\,); \F )
\]
whenever
$ \lvert D_{k+1}\rvert + (a_{k+1} - r + 1)\,\lvert D_k\rvert = \delta_C \leq  \epsilon < \lambda$.
This shows
that the persistent Betti numbers $\beta_1^{\epsilon, \epsilon'}$ are: Zero for  $\epsilon < \delta_C$ or  $\lambda \leq \epsilon'$,
and equal to one for $\delta_C \leq \epsilon \leq \epsilon' < \lambda$.
The result follows from   Remark \ref{rmk:diagramAltSum}.
\end{proof}

\section{Approximation Method}
\label{sec:AM}
We now detail our approximation method for the sliding window embedding of quasiperiodic functions. To handle general quasiperiodic functions, it is sufficient to consider those that are sums of exponentials, as discussed in Section \ref{sec: FS}. When \(f\) takes this form, its sliding window embedding admits the matrix representation
\[
  SW_{d,\tau}f(t) \;=\; A\,v_t.
\]
Recall that
\[
  \SW_{d,\tau}f(T)
  = \{\,SW_{d,\tau}f(t)\mid t=0,\dots,T\},
\]
denotes the corresponding sliding window point cloud,
and let \(\phi_T\) be the trajectory \(\{v_t\mid t=0,\dots,T\}\);
we assume \(T'\le T\),
and let \(G_T\) be as defined in Equation \ref{eq:Grid}. Throughout  this section  let \( l_0,M,N \in \mathbb{N} \),
we use \( j \) and \( l \) as dummy variables, and we let \( i  = \sqrt{-1}\in \mathbb{C} \).
Our approximation method relates the following metric spaces, and their corresponding persistence diagrams:
$$(\SW_{d,\tau} f(T),d_2) \sim ( \phi_T,d_{\infty}) \sim (G_{T'},d_{\infty}).$$
The first relation is obtained using an eigenvalue argument. Let $\sigma_{min}$ and $\sigma_{max}$ be the smallest and largest singular values of $A$, respectively.

\begin{lemma}
\label{bound 1}
Let $$ f(t) = \sum^M_{j=1} c_je^{2\pi i\omega_jt}, \ \ \ \ \ \ \ \ c_j \in \mathbb{C}\smallsetminus\{0\}, \ \omega_j \in \mathbb{R}. $$
Define $ k = \max\{\sigma_{min}^{-1},\sigma_{max} \sqrt{M} \}$. Suppose $(a_1,b_1) \in \dgm_{l_0}^R( \phi_T,d_{\infty}).$
If $$ \frac{b_1}{a_1} > k^2, $$ then there exists a unique $(a_0,b_0) \in \dgm_{l_0}^R(\SW_{d,\tau} f(T),d_2)$ such that
$$ \frac{1}{k} < \frac{\max\{ b_0,b_1\}}{\min\{ b_0,b_1\} }, \frac{\max\{ a_0,a_1\}}{\min\{ a_0,a_1\}} < k.  $$
\end{lemma}

\begin{proof}
Note that
$$ \sigma_{min} d_{2}(v_{t_1},v_{t_2}) \leq d_2(Av_{t_1},Av_{t_2}) \leq \sigma_{max}d_{2}(v_{t_1},v_{t_2}), $$ and
$$  d_{\infty}(v_{t_1},v_{t_2}) \leq d_{2}(v_{t_1},v_{t_2}) \leq \sqrt{M} d_{\infty}(v_{t_1},v_{t_2}).$$
We conclude that
$$  d_{\infty}(v_{t_1},v_{t_2}) \leq d_{2}(v_{t_1},v_{t_2}) \leq \frac{1}{\sigma_{min}} d_2(Av_{t_1},Av_{t_2})  \leq \frac{\sigma_{max}}{\sigma_{min}} d_2(Av_{t_1},Av_{t_2}) \leq \sqrt{M} d_{\infty}(v_{t_1},v_{t_2}).$$
Since $SW_{d,\tau}f(t)=Av_t$ and $A$ is full-rank (provided $d$ and $\tau$ are chosen according to \cite{gakhar2021sliding}), then
$$
R_{\epsilon}(\SW_{d,\tau}f(T),d_2)
\hookrightarrow
R_{\epsilon/\sigma_{min}}(\phi_T,d_{\infty})
\hookrightarrow
R_{\epsilon\frac{\sigma_{max}}{\sigma_{min}}}(\SW_{d,\tau}f(T),d_2)
$$
under the induced simplicial maps obtained from $\mu:\SW_{d,\tau}f(T) \longrightarrow \phi_T$, $\mu(Av_t)=v_t$, and
$\nu:\phi_T \longrightarrow \SW_{d,\tau}f(T)$, $\nu(v_t)=Av_t$. We have the following diagram:
\begin{center}
\begin{tikzpicture}
  \node (R1) at (0,0) {$R_{\epsilon}(\SW_{d,\tau}f(T),d_2)$};
  \node (R2) at (3,-2) {$R_{k\epsilon}(\phi_T,d_{\infty})$};
  \node (R3) at (6,0) {$R_{k^2\epsilon}(\SW_{d,\tau}f(T),d_2)$};

  \draw[->] (R1) -- (R3) node[midway, above ]{$\iota$};
  \draw[->] (R1) -- (R2) node[midway, left=0.5em] {$\mu$};
  \draw[->] (R2) -- (R3) node[midway, above] {$\nu$};
\end{tikzpicture}
\end{center}
where $\iota$ is the   inclusion $  R_\epsilon(\SW_{d,\tau} f(T), d_2) \hookrightarrow R_{k^2\epsilon}(\SW_{d,\tau} f(T), d_2)$.
Furthermore $\iota$ is contiguous to $\nu\circ \mu$;
that is, for $\sigma \in R_{\epsilon}(\SW_{d,\tau}f(T),d_2)$ we have $\iota(\sigma)\cup\nu\circ  \mu(\sigma)  \in R_{k^2\epsilon}(\SW_{d,\tau}f(T),d_2)$.
Hence we obtain a commutative diagram at the level of homology, which we re-parametrize in  logarithmic scale as
\begin{center}
\begin{tikzpicture}
  \node (R1) at (0,0) {$H_n(R_{\ln(\epsilon)}(\SW_{d,\tau}f(T),d_2);\mathbb{F})$};
  \node (R2) at (3,-2) {$H_n(R_{\ln(k)+\ln(\epsilon)}(\phi_T,d_{\infty});\mathbb{F})$};
  \node (R3) at (6,0) {$H_n(R_{2\ln(k)+\ln(\epsilon)}(\SW_{d,\tau}f(T),d_2);\mathbb{F})$};

  \draw[->] (R1) -- (R3) node[midway, above]{$\iota_*$};
  \draw[->] (R1) -- (R2) node[midway, left=0.5em] {$\mu_*$};
  \draw[->] (R2) -- (R3) node[midway, above, left=0.5em] {$\nu_*$};
\end{tikzpicture}
\end{center}
This gives us a $\ln(k)$ interleaving which implies by Theorem \ref{isometry theorem} that if for $(a_1,b_1) \in \text{dgm}_{l_0}^R( \phi_T,d_{\infty})$, $$\ln(b_1)-\ln(a_1) > 2\ln(k), $$ i.e. $\frac{b_1}{a_1} > k^2$, there exists a unique
$(a_0,b_0) \in \text{dgm}_{l_0}^R(\SW_{d,\tau}f(T),d_2)$ such that $$ | \ln(b_1)-\ln(b_0)| < \ln(k) \ \ \ \ \text{and} \ \ \ \ | \ln(a_1)-\ln(a_0)| < \ln(k).$$ This implies
$$ -\ln(k) < \ln\Big(\frac{\max\{b_0,b_1\}}{\min\{b_0,b_1\}}\Big), \ln\Big(\frac{\max\{a_0,a_1\}}{\min\{a_0,a_1\}}\Big) < \ln(k),$$ from which the result follows.
\end{proof}

To relate the two metric spaces
\[
  \bigl(\phi_T,\,d_{\infty}\bigr)
  \quad\text{and}\quad
  \bigl(G_{T'},\,d_{\infty}\bigr),
\]
we apply Proposition \ref{thm:Pstability} together with the definition of the bottleneck distance.
\begin{lemma}
\label{bound 2}
Let $ \lambda = d_{GH}(\phi_T,G_{T'}) $ and $(a_2,b_2) \in \dgm_{l_0}^R(G_{T'},d_{\infty})$. If $$b_2-a_2 > 4 \lambda,$$ then there exists a unique $(a_1,b_1) \in \dgm_{l_0}^R(\phi_T,d_\infty)$ such that $$|b_1-b_2| < 2\lambda  \ \ \text{and} \ \ |a_1-a_2| < 2\lambda. $$
\end{lemma}

\begin{proof}
Let $B=d_B(\text{dgm}_{l_0}^R(G_{T'},d_{\infty}),\text{dgm}_{l_0}^R(\phi_T,d_\infty)).$ By Theorem \ref{thm:Pstability}, $B<2\lambda$. Thus, $$b_2-a_2 > 4\lambda > 2B$$ and hence by the definition of the bottleneck distance, Definition \ref{bottleneck def}, there exists a unique $(a_1,b_1) \in \text{dgm}_{l_0}^R(\phi_T,d_\infty)$ such that $$|b_1-b_2| < B < 2\lambda  \ \ \text{and} \ \ |a_1-a_2| < B < 2\lambda, $$ as desired.
\end{proof}

We now combine both results to obtain a direct relation between $\text{dgm}_{l_0}^R(\SW_{d,\tau}f(T),d_2)$ and $\text{dgm}_{l_0}^R(G_{T'},d_{\infty})$. We assume the notation used before.

\begin{figure}[t!]
    \centering
    \includegraphics[width=.33\textwidth, trim=3.3cm 3.3cm 3.3cm 3.3cm, clip]{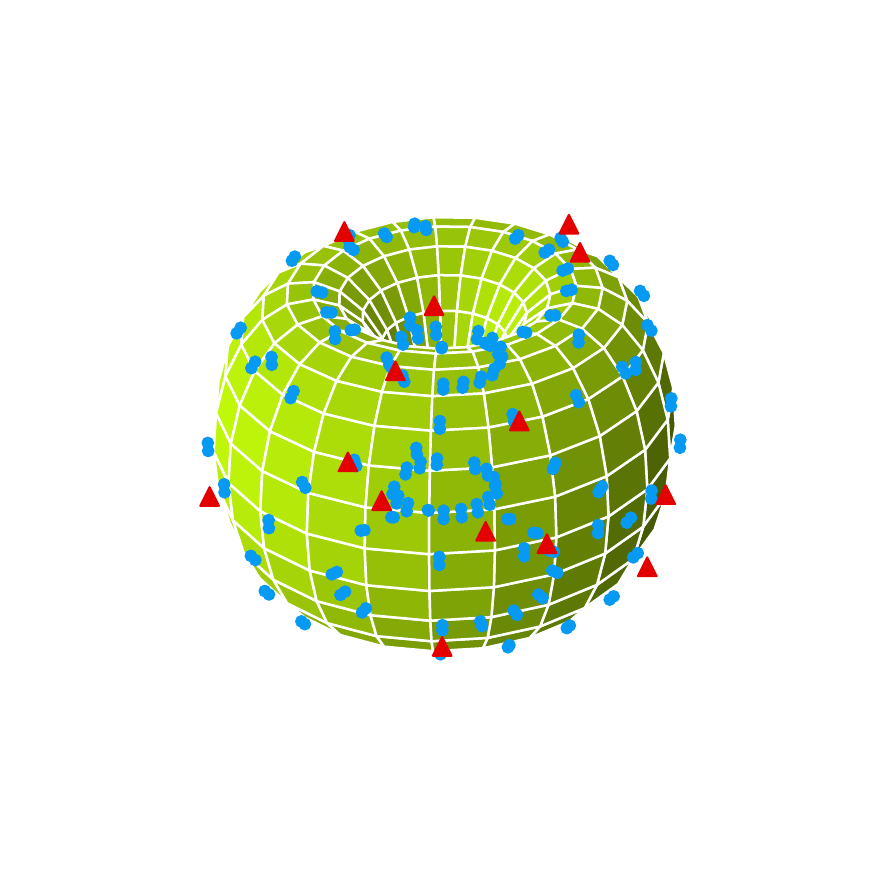}
    \hfill
    \includegraphics[width=0.4\textwidth]{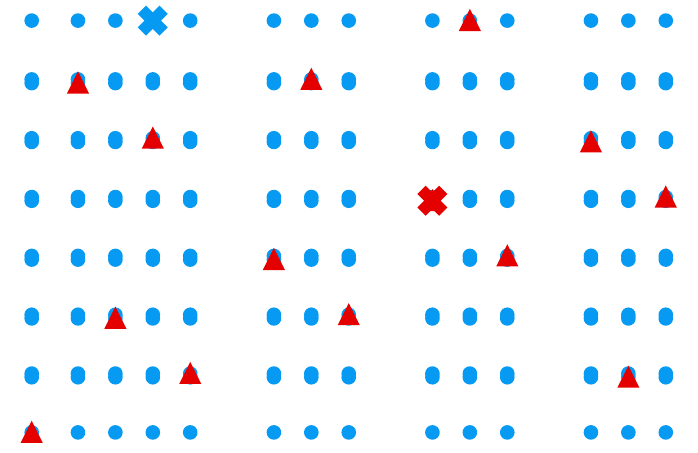}
\caption{Comparison of \(\phi_{13}\) (red triangles) and \(G_{13}\) (blue dots) for a two‐frequency torus with incommensurate \(\omega_1,\omega_2\). Left: Points plotted on the torus \(T^2\subset\C^2\). Right: Same points in the rectangular fundamental domain \([0,1)^2\). The pair marked \(\times\) realizes the Hausdorff distance between \(\phi_{13}\) and \(G_{13}\).}
\label{fig:Torus}
\end{figure}

\begin{theorem}
\label{AM theorem1}
If $(a_2,b_2) \in \dgm_{l_0}^R(G_{T'},d_{\infty}) $ satisfies
$$ \frac{b_2-2\lambda}{a_2+2\lambda} > \max\{k^2,1\}, $$
then there exists a unique $(a_0,b_0) \in \dgm_{l_0}^R(\{SW_{d,\tau}f(t)\}_{t=0}^T,d_2) $ satisfying
$$ \max \biggl\{ 0, \frac{b_2-2\lambda}{k} \biggl\} \leq \ b_0 \ \leq k \ (b_2 + 2\lambda), $$
$$ \max \biggl\{ 0, \frac{a_2-2\lambda}{k} \biggl\} \leq \ a_0 \ \leq k \ (a_2 + 2\lambda). $$
\end{theorem}

\begin{proof}
Let us note that the assumption implies
$$ b_2-a_2>b_2-a_2\max\{k^2,1\} > 2\lambda +2\lambda\max\{k^2,1\} > 4\lambda, $$ and thus by Lemma \ref{bound 2} there exists a unique $(a_1,b_1) \in \dgm_{l_0}^R(\phi_T,d_\infty)$ such that $$b_2-2\lambda < b_1 < b_2+2\lambda$$ and $$a_2-2\lambda < a_1 < a_2+2\lambda.$$ We conclude $$\frac{b1}{a_1} > \frac{b_2-2\lambda}{a_2+2\lambda} > k^2$$ and thus by Lemma \ref{bound 1} there exists a unique $(a_0,b_0) \in \dgm_{l_0}^R(\SW_{d,\tau}f(T) ,d_2)$ such that
$$ \frac{1}{k} < \frac{\max\{ b_0,b_1\}}{\min\{ b_0,b_1\} },\frac{\max\{ a_0,a_1\}}{\min\{ a_0,a_1\}} < k. $$
W.L.O.G. we assume $b_0=\max\{b_0,b_1\}$. Combining our inequalities, we get $$\frac{b_2-2\lambda}{k} < \frac{b_1}{k} < b_0 < kb_1 < k(b_2+2\lambda). $$ An analogous inequality can be obtained for $a_0$ completing the proof.
\end{proof}

There is a similar result that can be obtained for the case of ``symmetric" signals. The following shows it is sufficient to scale our coefficients and consider only half of the frequencies.

\begin{corollary}
\label{thm:3GT}
Let $$ f(t) = \sum^N_{j=1} c_j(e^{2\pi i\omega_jt}+e^{-2\pi i\omega_jt}), \ \ \ \ \ \ \ \ c_j \in \mathbb{C}\smallsetminus\{0\}, \ \omega_j \in \mathbb{R}.$$ Let \(\phi_T\)  denote the trajectory
\[
  \phi_T
  = \bigl\{\,(\tilde{c}_1 e^{2\pi i\omega_1 t}, \dots, \tilde{c}_N e^{2\pi i\omega_N t})^\top
    \;\big|\; t = 0,1, \ldots ,T\bigr\},
\]
and let \(G_{T'}\)   denote the grid
\[
  G_{T'}
  = \bigl\{\,(\tilde{c}_1 e^{2\pi i\omega_1 t_1}, \dots, \tilde{c}_N e^{2\pi i\omega_N t_N})^\top
    \;\big|\; t_n = 0, 1, \ldots , T', \; n =1, \ldots, N\bigr\}.
\]
Let $ k = \max\{\sigma_{min}^{-1},\sigma_{max} \sqrt{2N} \}$ and $\lambda = d_{GH}(\phi_T,G_{T'})$.
If $(a_2,b_2) \in \dgm_{l_0}^R(G_{T'},d_{\infty}) $ satisfies
$$ \frac{b_2-2\lambda}{a_2+2\lambda} > \max\{k^2,1\}, $$
then there exists a unique
$(a_0,b_0) \in \dgm_{l_0}^R(\SW_{d,\tau}f(T),d_2) $ satisfying
$$ \max \biggl\{ 0, \frac{b_2-2\lambda}{k} \biggl\} \leq \ b_0 \ \leq k \ (b_2 + 2\lambda), $$
$$ \max \biggl\{ 0, \frac{a_2-2\lambda}{k} \biggl\} \leq \ a_0 \ \leq k \ (a_2 + 2\lambda). $$
\end{corollary}

\begin{proof}
Let $$ \bar{\phi}_T = \{ 2^{-\frac{1}{2}}(\tilde{c}_1e^{2\pi i\omega_1t},\dots,\tilde{c}_{N}e^{2\pi i\omega_Nt},\tilde{c}_1e^{-2\pi i\omega_1t},\dots,\tilde{c}_{N}e^{-2\pi i\omega_Nt})^\intercal \}_{0\leq t \leq T }. $$
By Lemma \ref{bound 1}, if $(\bar{a}_1,\bar{b}_1) \in \dgm_{l_0}^R( \bar{\phi}_T,d_{\infty})$ and $$ \frac{\bar{b}_1}{\bar{a}_1} > k^2, $$ then there exists a unique $(a_0,b_0) \in \dgm_{l_0}^R(\SW_{d,\tau}f(T) ,d_2)$ such that
$$ \frac{1}{k} < \frac{\max\{ b_0,\bar{b}_1\}}{\min\{ b_0,\bar{b}_1\} },\frac{\max\{ a_0,\bar{a}_1\}}{\min\{ a_0,\bar{a}_1\}} < k.  $$
On the other hand, since
 $$  d_2(c_le^{i\omega_lt_1},c_le^{i\omega_lt_2}) = d_2(c_le^{-i\omega_lt_1},c_le^{-i\omega_lt_2}) $$ the maps $\mu:\bar{\phi}_T\longrightarrow \phi_T, \ \mu(v_t)=\sqrt{2}v_t$  and $\nu: \phi_T \longrightarrow \bar{\phi}_T, \nu(v_t)=2^{-\frac{1}{2}} v_t$ are isometries. Thus, by Proposition \ref{isomorphic diagrams}, $(\bar{a}_1,\bar{b}_1) \in \dgm_{l_0}^R( \bar{\phi}_T,d_{\infty})$ if and only if there exist a unique $(a_1,b_1) \in \dgm_{l_0}^R( \phi_T,d_{\infty})$ such that $$a_1=\bar{a}_1 \ \ \ \text{and} \ \ \ b_1=\bar{b}_1.$$ Combining this fact with Lemma \ref{bound 2}, we can say that if $(a_2,b_2) \in \text{dgm}_{l_0}^R(G_{T'},d_{\infty})$ and $$b_2-a_2 > 4 \lambda$$ then there exists a unique $(\bar{a}_1,\bar{b}_1) \in \text{dgm}_{l_0}^R(\bar{\phi}_T,d_\infty)$ such that $$|\bar{b}_1-b_2| < 2\lambda  \ \ \text{and} \ \ |\bar{a}_1-a_2| < 2\lambda. $$ By the equivalences that have now been shown, we can replicate the arguments in Theorem \ref{AM theorem1} to conclude the result.
\end{proof}

\begin{remark}[Error‐bound confidence region]
\label{Ebound}
Let
\[
  (a_2,b_2)\;\in\;\mathrm{dgm}^R_{l_0}\bigl(G_{T'},d_{\infty}\bigr).
\]
Under the hypotheses of Theorem \ref{AM theorem1}, there is a corresponding
\((a_0,b_0)\in\mathrm{dgm}^R_{l_0}\bigl(\SW_{d,\tau}f(T),d_2\bigr)\)
satisfying
$$ \max \biggl\{ 0, \frac{b_2-2\lambda}{k} \biggl\} \leq \ b_0 \ \leq k \ (b_2 + 2\lambda), $$
$$ \max \biggl\{ 0, \frac{a_2-2\lambda}{k} \biggl\} \leq \ a_0 \ \leq k \ (a_2 + 2\lambda). $$
It follows that \((a_0,b_0)\) lies in the rectangle
\begin{center}
\begin{tikzpicture}
    \draw  (0,0) coordinate (A) -- (3,0) coordinate (B) -- (3,1.5) coordinate (C) -- (0,1.5) coordinate (D) -- cycle;
    \draw (2,1) coordinate (AA) -- cycle;
    \node at (AA) [below left] {$(a_0,b_0)$};
    \node at (A) [below left] {$(x_0,y_0)$};
    \node at (B) [below right] {$(x_1,y_0)$};
    \node at (C) [above right] {$(x_1,y_1)$};
    \node at (D) [above left] {$(x_0,y_1)$};
\end{tikzpicture}
\end{center}
\noindent where $x_0=\max\left\{0,\frac{a_2-2\lambda}{k}\right\}$, $x_1=k(a_2+2\lambda)$, $y_0=\max\left\{0,\frac{b_2-2\lambda}{k}\right\}$, and $y_1=k(b_2+2\lambda)$. This rectangle serves as a confidence region for the approximation.
\end{remark}

We have shown explicitly how $\text{dgm}_{l_0}^R(G_T,d_{\infty})$ approximates $\text{dgm}_{l_0}^R(\SW_{d,\tau}f(T),d_2).$ In order to leverage our results from Section \ref{MR section} we use the K\"unneth Formula, see Theorem \ref{kunneth formula}. It allow us to apply Theorem \ref{0th homolgy} and Theorem \ref{1st homology} to compute $\text{dgm}_{l_0}^R(G_T,d_{\infty})$ for ${l_0}=1,2.$ This approach is detailed in the next result. For ease of presentation, we shall express persistence information using the barcode notation \(\mathrm{bcd}^R_i\) in the following result; these barcodes correspond exactly to the persistence diagrams \(\mathrm{dgm}^R_i\) as detailed in Section \ref{sec:Back}.

\begin{theorem}
\label{GT Kunneth}
Let $ c_j \in \mathbb{C}$ and $ \omega_j \in \mathbb{R}$. For $ a,b,s \in \mathbb{R} $, let $\bar{a} = 2 \  \sin(\pi a) $, $ \omega'_l= \frac{\omega_l}{2\pi} $, and for $I = [a,b)$ we let $\bar{I} = [\bar{a},\bar{b})$ and write $sI$ to denote $[sa,sb)$. Let
$$ G_T = \{ (c_1e^{i\omega_1t_1},\dots,c_{N}e^{i\omega_Nt_N})^\intercal \}_{0\leq t_i \leq T }. $$ Then,
$$ \bcd_{l_0}^R(G_T,d_{\infty}) = \bigcup_{\sum_l r_l={l_0}} \big\{|c_1|\bar{I}_1 \cap\cdots \cap |c_N|\bar{I}_N \ \big| \ I_l \in \bcd_{r_l}^R(S_{\omega'_l,T} ,\bar{d}) \big\}, $$
\end{theorem}

\begin{proof}
For $0 \leq l \leq N$ let $$X_l=\{c_le^{i\omega_lt} \}_{t=0}^T. $$ By Theorem \ref{kunneth formula} $$ \text{bcd}_{l_0}^R(G_T,d_{\infty}) = \bigcup_{\sum_l r_l={l_0}} \big\{J_1 \cap\cdots \cap J_N \ | \ J_l \in bcd_{r_l}^R(X_l,d_2) \big\}. $$
Thus, it is sufficient to show that $ J_l \in \bcd_{r_l}^R(X_l,d_2) $ if and only if there exists a unique $ I_l \in \bcd_{r_l}^R(S_{\omega'_l,T'} ,\bar{d}) $ such that $$ J_l=|c_l|\bar{I_l}.$$
To do so, let us denote $x_j= \omega_lt_j \ \mod \ 2\pi$ and consider the metric $d_l$ on $X_l$ given by $$d_l(c_le^{i\omega_lt_1},c_le^{i\omega_lt_2}) = |c_l| \min \{ |x_1-x_2|, |2\pi-x_2+x_1|, |2\pi-x_1+x_2| \}.$$ It captures the geodesic distance between points along the circle of radius $|c_l|$ centered at $(0,0).$ By trigonometric identities
$$ d_2(c_le^{i\omega_lt_1},c_le^{i\omega_lt_2}) = \frac{|c_l|\sin(|c_l|^{-1}d_l(c_le^{i\omega_lt_1},c_le^{i\omega_lt_2}))}{\sin\big(\frac{\pi-|c_l|^{-1}d_l(c_le^{i\omega_lt_1},c_le^{i\omega_lt_2})}{2}\big)}=2|c_l|\sin\Big(\frac{|c_l|^{-1}d_l(c_le^{i\omega_lt_1},c_le^{i\omega_lt_2})}{2}\Big).$$
Thus, as a consequence of Theorem \ref{isometry theorem}, $ J_l \in \bcd_{r_l}^R(X_l,d_2) $ if and only if there exists a unique $[a,b) \in bcd_{r_l}^R(X_l,d_l)$ such that $$J_l = \Big[2|c_l|\sin\big(\frac{|c_l|^{-1}a}{2}\big),2|c_l|\sin\big(\frac{|c_l|^{-1}b}{2}\big)\Big).$$ Furthermore, noting $$d_l(c_le^{i\omega_lt_1},c_le^{i\omega_lt_2}) = 2\pi |c_l| \bar{d}([\omega_l' t_1],[\omega_l' t_2]) $$ we can similarly say $[a,b) \in \bcd_{r_l}^R(X_l,d_l)$ if and only if there exists a unique $[a_0,b_0) \in \bcd_{r_l}^R(S_{\omega_l',T},\bar{d})$ such that $$[a,b) = [2\pi |c_l| a_0,2\pi |c_l| b_0).$$ Combining these equivalences gives the desired result.
\end{proof}

\subsection{Example}
\label{ssec:Ex}

In this example we execute the workflow illustrated in Figure \ref{fig:diagram}, which takes as input a quasiperiodic function and produces the diagrams, $T'\leq T,$
\[
  \mathrm{dgm}^R_\ell\bigl(G_{T'},\,d_\infty\bigr),
  \qquad \ell = 1,2.
\]
Concretely, we set
\[
  f(t) = \tfrac1{\sqrt2}\,e^{i\sqrt3\,t}
       + \tfrac1{\sqrt2}\,e^{i\sqrt5\,t}.
\]

We then carry out three stages:

\begin{enumerate}
  \item \textbf{Algorithm 1 (3 Gap Code).}
    For each incommensurate frequency \(\omega_j\in\{\sqrt3,\sqrt5\}\), we run Algorithm \ref{alg:3GC} using the results from Section \ref{MR section} to obtain the exact diagrams
    \[
      \mathrm{dgm}^R_\ell\bigl(S_{\omega'_j,T'},\,\bar{d}\bigr),
      \quad \ell = 1,2,
      \quad \omega'_j=\frac{\omega_j}{2\pi}.
    \]

  \item \textbf{Persistent Künneth Formula.}
    We assemble these single‐frequency outputs via Theorem \ref{GT Kunneth} to obtain the diagrams
    \[
      \mathrm{dgm}^R_\ell\bigl(G_{T'},\,d_\infty\bigr),
      \quad \ell = 1,2.
    \]

  \item \textbf{Error Bounds.}
    Finally, Theorem \ref{AM theorem1} provides an explicit bottleneck‐distance bound showing that
    \(\mathrm{dgm}^R_\ell(G_{T'},d_\infty)\) approximates the true sliding‐window diagram
    \(\mathrm{dgm}^R_\ell\bigl(SW_{d,\tau}f(T),d_2\bigr)\)
    within the claimed error.
\end{enumerate}
\begin{remark}
In practical applications to real-world data, the frequencies \(\omega_j\) are first estimated via the FFT of the observed time series (see Section \ref{sec: App}).
\end{remark}
With the roadmap in place, we begin by constructing the sliding‐window embedding matrix as detailed at the end of Section \ref{sec: FS}. In our example one obtains
\[
SW_{d,\tau}f(t)
=\frac{1}{\sqrt{d+1}}
\begin{pmatrix}
1                   & 1                 \\[6pt]
e^{\,i\sqrt3\,\tau} & e^{\,i\sqrt5\,\tau} \\[3pt]
\vdots              & \vdots            \\[3pt]
e^{\,i\sqrt3\,d\tau} & e^{\,i\sqrt5\,d\tau}
\end{pmatrix}
\begin{pmatrix}
\tilde{c}_1\,e^{\,i\sqrt3\,t} \\[3pt]
\tilde{c}_2\,e^{\,i\sqrt5\,t}
\end{pmatrix},
\]
where \(\tilde{c}_1=\tilde{c}_2=\sqrt{d+1}/\sqrt2\).

We take
\[
  d = 1,
  \qquad
  \tau = \frac{\pi}{\sqrt3 - \sqrt5}.
\]
With these parameters \(A\) is orthonormal. Hence
\[
    k = \max\{\,\sigma_{min}^{-1},\,\sigma_{max}\sqrt M\}
      = \max\{\,1,\,\sqrt2\}
      = \sqrt2
      \quad(M=2).
\]

Finally, we approximate the sliding‐window point cloud
\[
\SW_{d,\tau}f(2000)
\]
using \(T'=500\). This sets the stage for fleshing out our 3G approximation method. We next apply Algorithm \ref{alg:3GC} to each frequency component in turn.

First, the scaled frequencies admit the continued‐fraction expansions
\[
\omega_1'=\frac{\sqrt3}{2\pi}=[0;3,1,1,1,2,\dots],
\qquad
\omega_2'=\frac{\sqrt5}{2\pi}=[0;2,1,4,3,1,\dots].
\]
By Corollaries \ref{0th bcd} and \ref{1st dgm}, the resulting persistence diagrams are

\[
\mathrm{dgm}^R_{0}\bigl(S_{\omega_1',500},\bar d\bigr)
= \bigl\{(0,1.577\times10^{-3})^{160},\,(0,2.077\times10^{-3})^{316},\,(0,3.654\times10^{-3})^{24},\,(0,\infty)\bigr\},
\]

\[
\mathrm{dgm}^R_{1}\bigl(S_{\omega_1',500},\bar d\bigr)
= \bigl\{(3.654\times10^{-3},\,334.551\times10^{-3})\bigr\},
\]

\[
\mathrm{dgm}^R_{0}\bigl(S_{\omega_2',500},\bar d\bigr)
= \bigl\{(0,0.368\times10^{-3})^{161},\,(0,2.637\times10^{-3})^{220},\,(0,3.005\times10^{-3})^{119},\,(0,\infty)\bigr\},
\]

\[
\mathrm{dgm}^R_{1}\bigl(S_{\omega_2',500},\bar d\bigr)
= \bigl\{(3.005\times10^{-3},\,334.846\times10^{-3})\bigr\}.
\]
With our single‐frequency persistence diagrams in hand, we now apply the Persistent Künneth Formula (Theorem \ref{GT Kunneth}). Recall
\[
  G_{500} = \bigl\{(e^{i\omega_1 t_1},\,e^{i\omega_2 t_2})\bigr\}_{0\le t_1,t_2\le500}.
\]
To verify the hypotheses of Theorem \ref{AM theorem1}, we identify the diagrams of maximal persistence. From
\(\mathrm{dgm}^R_{1}(S_{\omega_1',500},\bar d)\) the interval
\((a'_1,b'_1)=(3.654\times10^{-3},\,334.551\times10^{-3})\)
induces in \(\mathrm{dgm}^R_{1}(G_{500},d_\infty)\) the interval
\[
  (a''_1,b''_1)
  = \bigl(2\sin(a'_1\pi),\,2\sin(b'_1\pi)\bigr)
  = (0.02296,\,1.73586).
\]
Similarly, from
\(\mathrm{dgm}^R_{1}(S_{\omega_2',500},\bar d)\) with
\((a'_2,b'_2)=(3.005\times10^{-3},\,334.846\times10^{-3})\) we obtain
\[
  (a''_2,b''_2)
  = \bigl(2\sin(a'_2\pi),\,2\sin(b'_2\pi)\bigr)
  = (0.01888,\,1.73678).
\]
In \(\mathrm{dgm}^R_{2}(G_{500},d_\infty)\) the most persistent interval is
\((a''_3,b''_3)=(0.02296,\,1.73586)\).

Since \(\lambda \le d_H\bigl(\phi_{2000},G_{500}\bigr)=0.20975\) and \(\max\{k^2,1\}=2\), we compute
\[
  \frac{b''_1 - 2\lambda}{a''_1 + 2\lambda}
    = \frac{1.73586 - 2\cdot0.20975}{0.02296 + 2\cdot0.20975}
    = 2.9751 > 2,
\quad
  \frac{b''_2 - 2\lambda}{a''_2 + 2\lambda}
    = \frac{1.73678 - 2\cdot0.20975}{0.01888 + 2\cdot0.20975}
    = 3.0049 > 2.
\]
Thus the conditions of Theorem \ref{AM theorem1} are met. Define for any \(x\)
\[
  X(x) = \max\!\Bigl\{0,\frac{x - 2\lambda}{k}\Bigr\},
\quad
  Y(x) = k\,(x + 2\lambda).
\]
We compute the error bounds to be
\[
  (X(a''_1),Y(a''_1)) = (0,\,0.6257),\quad
  (X(b''_1),Y(b''_1)) = (0.9308,\,3.0481),
\]
\[
  (X(a''_2),Y(a''_2)) = (0,\,0.6200),\quad
  (X(b''_2),Y(b''_2)) = (0.9315,\,3.0494),
\]
\[
  (X(a''_3),Y(a''_3)) = (0,\,0.6257),\quad
  (X(b''_3),Y(b''_3)) = (0.9308,\,3.0481).
\]

Therefore Theorem \ref{AM theorem1} guarantees unique diagrams
\[
  (a_1,b_1),\,(a_2,b_2)\;\in\;\mathrm{dgm}^R_{1}\bigl(\SW_{d,\tau}f(2000),d_2\bigr),
\quad
  (a_3,b_3)\;\in\;\mathrm{dgm}^R_{2}\bigl(\SW_{d,\tau}f(2000),d_2\bigr)
\]
satisfying
\[
\begin{aligned}
  0 &\le a_1 \le 0.6257, & 0.9308 &\le b_1 \le 3.0481,\\
  0 &\le a_2 \le 0.6200, & 0.9315 &\le b_2 \le 3.0494,\\
  0 &\le a_3 \le 0.6257, & 0.9308 &\le b_3 \le 3.0481.
\end{aligned}
\]

Finally, we compute the actual persistence diagrams of the sliding‐window embedding for verification. By evaluating
\[
  \mathrm{dgm}^R_{i}\bigl(\SW_{d,\tau}f(2000),d_2\bigr),
  \quad i=1,2,
\]
one finds that the diagrams predicted by our 3G method are
\[
  (a_1,b_1)=(0.193,\,1.771),\quad
  (a_2,b_2)=(0.194,\,1.753),\quad
  (a_3,b_3)=(0.296,\,1.753).
\]
Each \((a_i,b_i)\) lies within its respective error-bound rectangle, as shown in Figure \ref{fig:example 1 3G}. For comparison, we also include two alternative approximations—the sliding-window landmark method \(\mathrm{SW}_L(n)\) and the landmark-based Künneth method \(\mathrm{K}_L(n)\) \cite{gakhar2019k}—together with their corresponding error-bound rectangles and measured computation times for varying sample sizes \(n\).

\begin{figure}[htb!]
    \centering
    \includegraphics[width=0.33\textwidth]{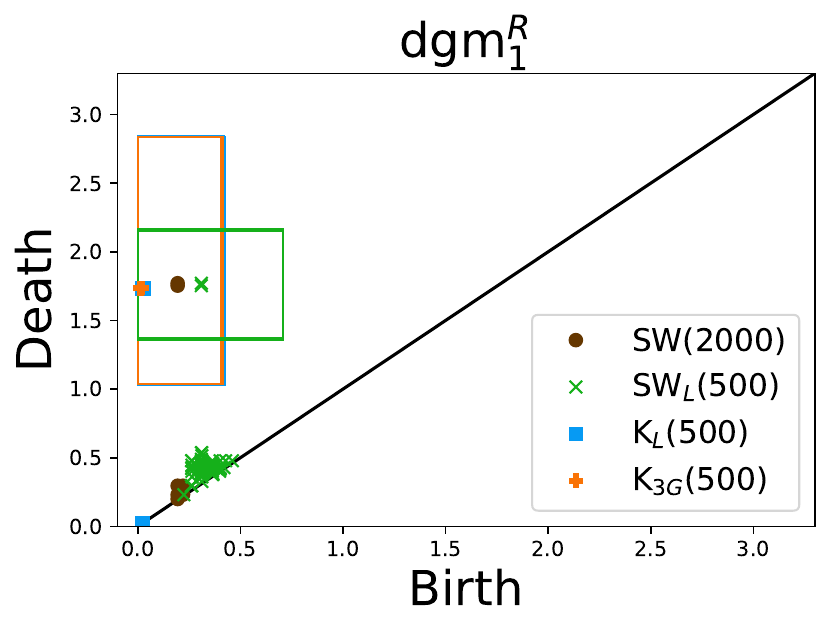}
    \hfill
    \includegraphics[width=0.33\textwidth]{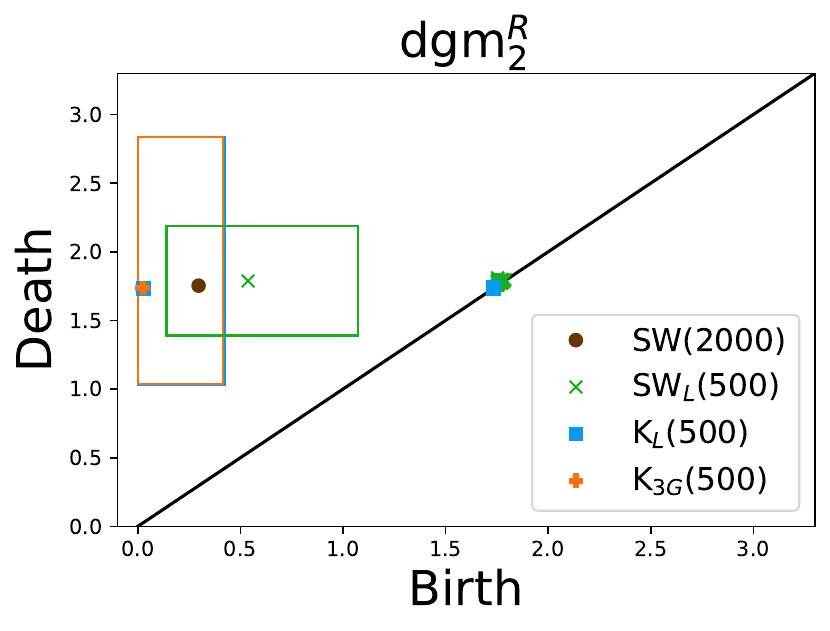}
    \hfill
    \includegraphics[width=0.31\textwidth]{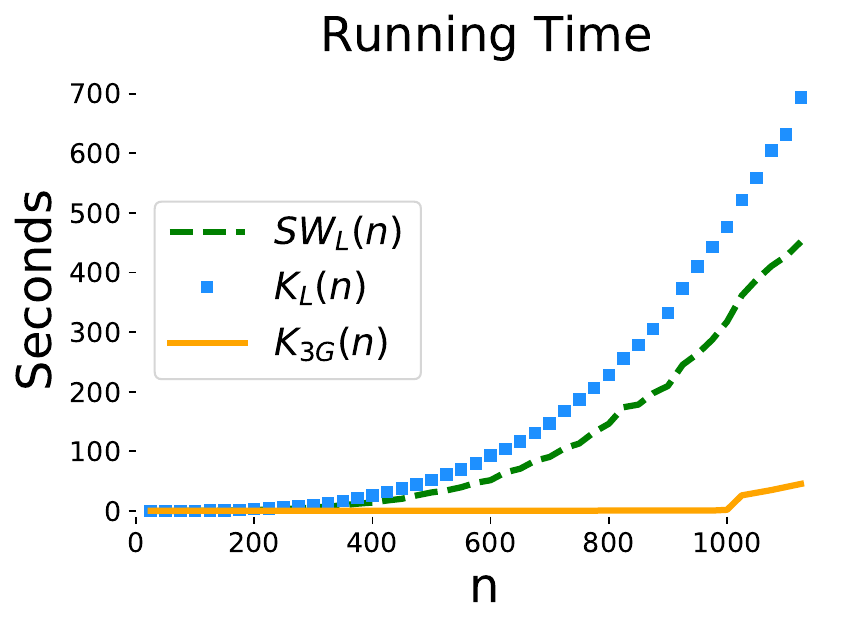}

    \caption{Persistence diagrams from Section \ref{ssec:Ex}. We also report the computation times for Ripser \cite{bauer2016ripser} on the sliding window embedding ($SW$) and the method ($K_L$). Our implementation ($K_{\mathrm{3G}}$) does not invoke Ripser at any stage.}
    \label{fig:example 1 3G}
\end{figure}

\section{Applications}
\label{sec: App}

In this section, we will apply our pipeline, 3G, to the following examples: synthetic tremor signals in nonlinear tuned vibration absorbers as indicators of safe or unsafe operations \cite{detroux2015performance}; neuroscience, where quasiperiodicity is associated with task-specific functions in certain brain areas \cite{kaslik2022stability}; and celestial mechanics, where quasiperiodicity translates to favorable trajectories for missions \cite{gomez2001dynamics}. We contribute to these fields by providing a faster alternative (see Figure \ref{fig:example 1 3G} and Table \ref{tab:RT}) to analyze the reconstructed signals of potential quasiperiodic systems.

\subsection{Methodology}
We focus on dynamical systems whose flow \(\Phi\) arises from a known system of ordinary differential equations. To generate a corresponding time series “observation,” we solve these ODEs from a chosen initial condition and record the resulting solution values as the signal \(f(t)\). We reconstruct the system by computing the sliding window point cloud $\{SW_{d,\tau}f(\epsilon t) \}_{t=0}^T$, where $\epsilon$ is chosen according to the Nyquist-Shannon sampling theorem \cite{shannon1949communication}. The framework detailed in \cite{gakhar2021sliding} provides a method for selecting $d$ and $\tau$ and approximating $ \text{dgm}^R_j(\{SW_{d,\tau}f(\epsilon t) \}_{t=0}^T,d_{2}) $ with $ \text{dgm}^R_j(\{SW_{d,\tau}S_1f(\epsilon t) \}_{t=0}^T,d_{2})$. Since $S_1f$ is a sum of the form shown in Theorem \ref{thm:3GT}, we can apply our result to approximate the latter, thus closely approximating the former. Indeed, in the examples we consider
$$ S_1f(\epsilon t) = \sum^2_{j=1} c_j(e^{i \epsilon \omega_jt}+e^{-i \epsilon \omega_jt}). $$
Thus, 3G can be used to approximate $ \text{dgm}^R_j(\{SW_{d,\tau}S_1f(\epsilon t) \}_{t=0}^T,d_{2})$ with a known error bound as detailed in Remark \ref{Ebound}.

\subsection{Results}
We consider systems with known quasiperiodic behavior, specifically those with two incommensurate frequencies. Hence, the sliding window point cloud $\{SW_{d,\tau}f(\epsilon t) \}_{t=0}^T$ samples a $2$-torus. We label the corresponding persistence diagrams as $SW(T)$ (shown as brown dots). By \(SW_{S_1}(T)\), we denote the persistence diagrams of the sliding‑window embedding \(S_1(f)\) (shown as blue squares). The approximation obtained from 3G is labeled $K_{3G}(T)$ (shown as orange crosses).

\subsubsection{Double Gyre}
\label{ssec:example1}

The driven Double Gyre system provides a model for patterns occurring in geophysical flows \cite{shadden2005definition}. A topological approach to it has shown great success. Indeed, as shown in \cite{charo2020topology}, there are multiple topological classifications that can be observed in the system representing the motion of a fluid particle. In this work, they noted that the motion of fluid can be sparse, meaning some initial conditions will imply a particle will be contained in a small region. Different initial conditions have shown trajectories with a topology of a standard strip, five-handle structure with a torsion, torus, and a m\"{o}bius strip. This result was achieved by calculating homologies in a branched manifold while keeping track of the orientability chains, allowing for the identification of the branches and the localization of twists or torsions. This approach is called Branched Manifold analysis through Homologies (BraMAH). At its core, it follows the principle of reconstructing a dynamical system from a time series. Yet the details at each step are completely different from SW. Nevertheless, for the corresponding initial conditions, we independently corroborated the presence of a torus, i.e., quasiperiodicity in the motion of the fluid particle, see Figure \ref{fig:DG}. This highlights the robustness offered by a topological approach to dynamical systems.
\begin{figure}[!b]
    \centering
    \includegraphics[width=0.29\textwidth, trim=.5cm .3cm .5cm .3cm, clip]{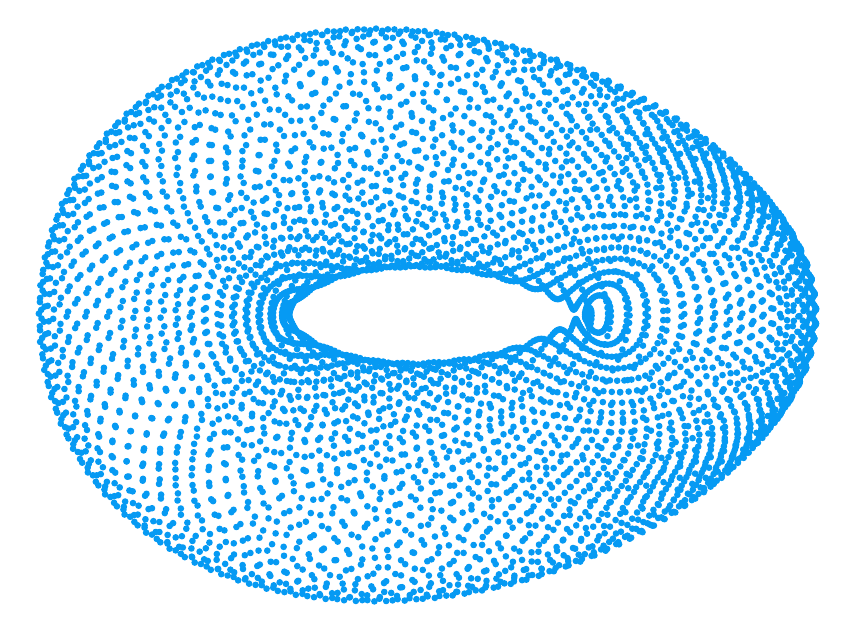}
    \hfill
    \includegraphics[width=0.32\textwidth]{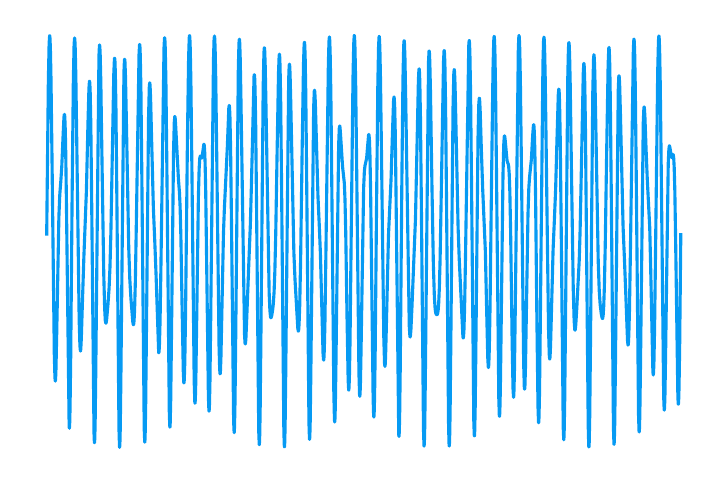}
    \hfill
    \includegraphics[width=0.32\textwidth]{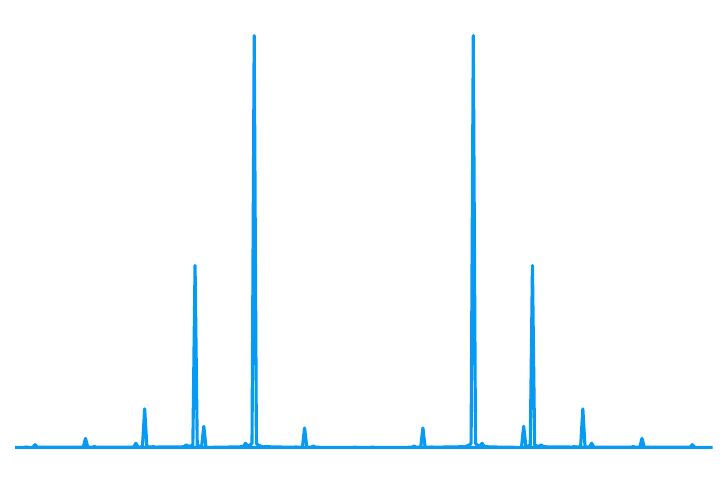}
    \hfill
    \includegraphics[width=0.47\textwidth]{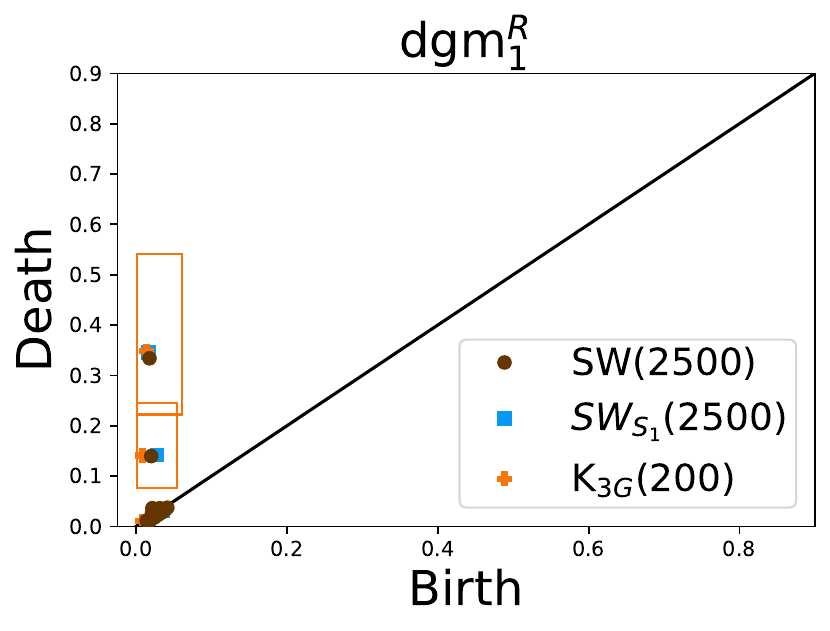}
    \hfill
    \includegraphics[width=0.47\textwidth]{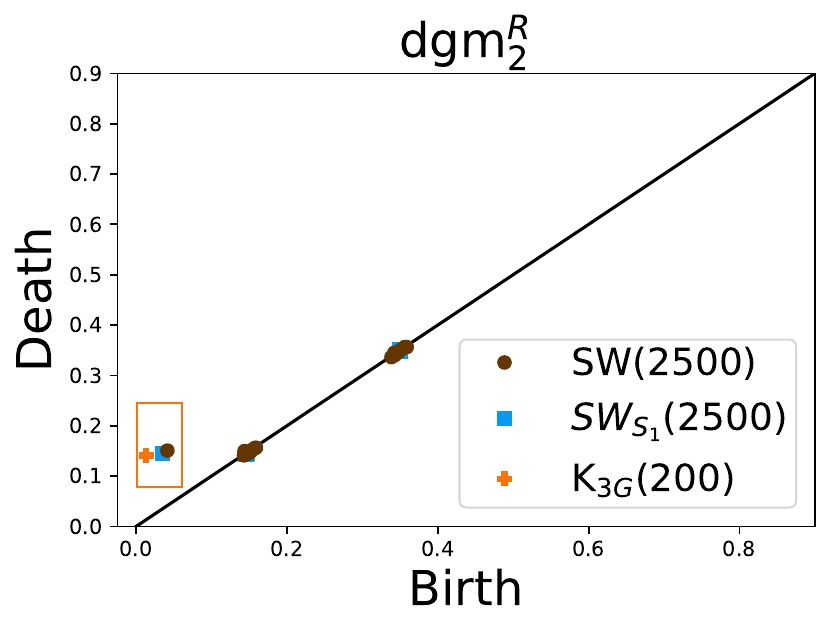}

    \caption{Top row: Phase portrait of the driven Double Gyre system with initial condition \(x_0\) (left), the corresponding trajectory \(x_1(t)\) used for sliding window embedding (center), and its Fourier power spectrum via FFT (right). Bottom row: The persistence diagrams of the resulting Rips filtrations for \(x_1(t)\).}
    \label{fig:DG}
\end{figure}

Concretely, the driven Double Gyre is given by the equation $\dot{x} = (-\frac{\partial \psi}{\partial x_2},\frac{\partial \psi}{\partial x_1} ) $ where
$$ \psi(x_1,x_2,t)=A \ \text{sin}(\pi g(x_1,t)) \ \text{sin}(\pi x_2) $$ with
$$ g(x_1,t) = \eta \ \text{sin}(\lambda t) (x^2_1-2x_1) + x_1,$$
and parameter values $A,\mu=0.1$ and $ \lambda=\pi/5.$ The system has a toroidal attractor for the initial condition $x_0=(0.5,0.625)$ \cite{charo2020topology}. We solve the system using $dt=0.1$ up to $t=800$. Furthermore, SW is done using $ f=x_1, \ d=4, \ \tau=119.03, \ \text{and} \ \epsilon=0.1. $

\subsubsection{The torus}
We now consider a toroidal attractor in $\R^4$ with coordinates $(x,y,z,r)$. The dynamical system is given by the set of equations:

$$ \dot{x} = -y + x(1 - \sqrt{x^2 + y^2}), \quad \dot{y} = x + y(1 - \sqrt{x^2 + y^2}), $$
$$ \dot{z} = -kr + z(4 - \sqrt{z^2 + r^2}), \quad \dot{r} = kz + r(4 - \sqrt{z^2 + r^2}). $$
When $k$ is irrational, the solutions span a torus \cite{penalva2018takens}. We considered the case $k = \sqrt{2}$ with the initial condition $(x_0, y_0, z_0, r_0) = (1, 0, 4, 0)$. We solve the system with $dt = 0.1$ up to $t = 800$. The sliding window is done using $f = x + z$, $d = 4$, $\tau = 87.56$, and $\epsilon = 0.1$, see Figure \ref{fig:Torus R4}.

\begin{figure}[!t]
    \centering
    \includegraphics[width=0.48\textwidth]{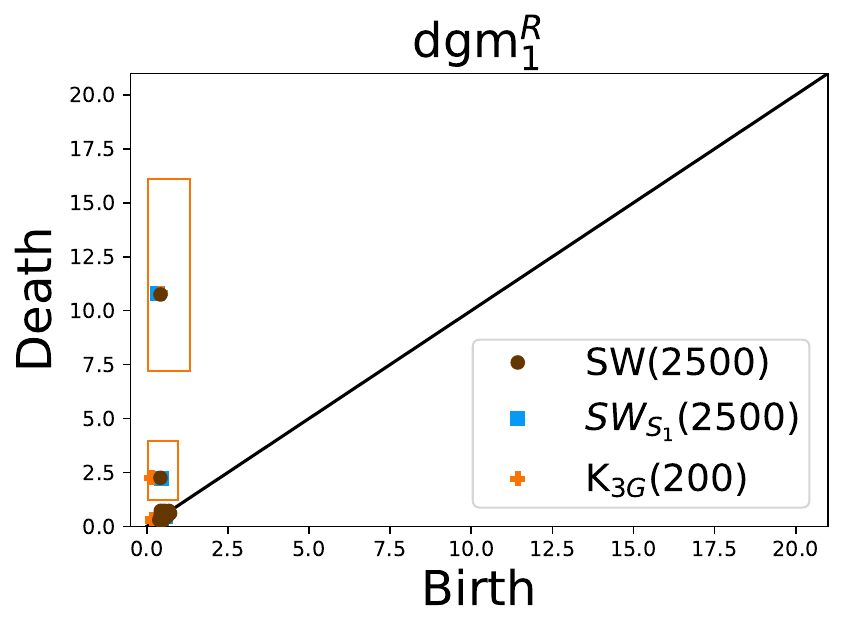}
    \hfill
    \includegraphics[width=0.48\textwidth]{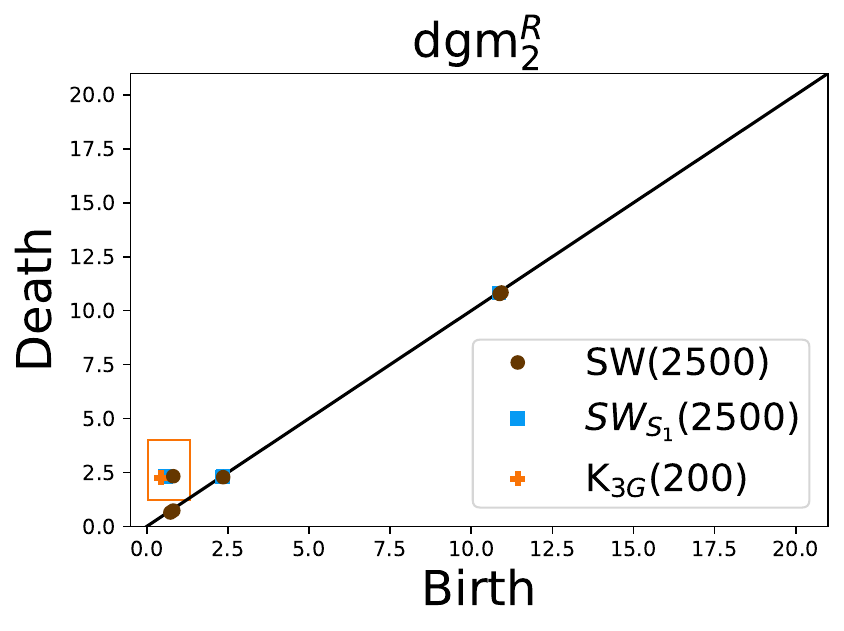}

    \caption{Diagrams obtained for the torus in $\mathbb{R}^4$.}
    \label{fig:Torus R4}
\end{figure}

\begin{figure}[!b]
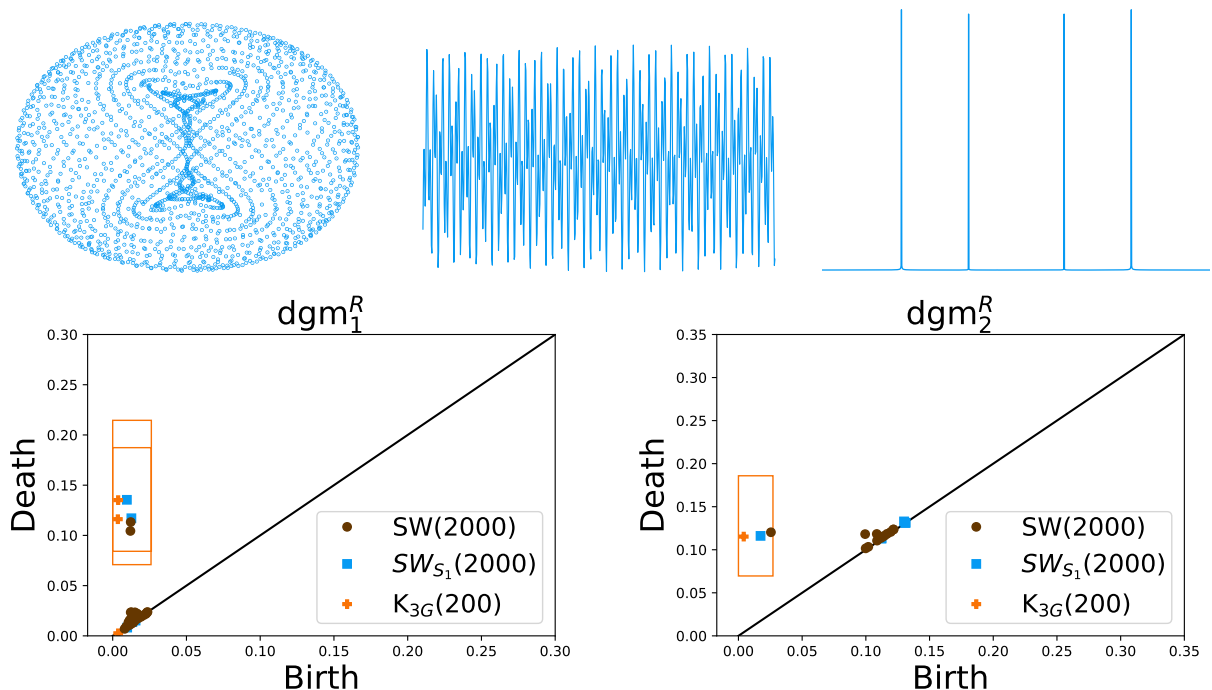

    \centering
    \includegraphics[width=0.31\textwidth]{Pend1PDF.pdf}
    \hfill
    \includegraphics[width=0.32\textwidth]{Pend2PDF.pdf}
    \hfill
    \includegraphics[width=0.32\textwidth]{FFTPendPDF.pdf}
    \hfill
    \includegraphics[width=0.48\textwidth]{R5_3D1PDF.pdf}
    \hfill
    \includegraphics[width=0.48\textwidth]{R5_3D2PDF.pdf}

    \caption{Top row: Phase portrait \((x,\dot x)\) of the pendulum on a sliding-block system (left), the scalar trajectory \(x(t)\) used for sliding window embedding (center), and its Fourier spectrum via FFT (right). Bottom row: Persistence diagrams of the Rips filtration on the sliding window point cloud derived from \(x(t)\).}
    \label{fig:Pendulum}
\end{figure}

\subsubsection{Pendulum Attached to a Sliding Block}
\label{ssec:example3}
We now consider a particular type of tuned mass damper (TMD). A TMD is a device used to suppress vibration by moving a mass attached to the main structure through springs and dampers \cite{wen2019quasi}. This model has been successfully used to dampen the effect of long-duration earthquake ground motions \cite{murudi2004seismic}. Indeed, one can consider the main structure being a skyscraper and the incoming wave being generated by the earthquake. Thus, a better understanding of this system translates to earthquake-resistant technologies.

We consider the case of a pendulum attached to a sliding block, see Figure \ref{fig:Pendulum}. It was shown to exhibit quasiperiodicity in \cite{wen2019quasi}. The governing equations are:

$$ \ddot{x} + \alpha^2 x - \bar{\epsilon} g \theta - \bar{\epsilon} L \dot{\theta}^2 \theta = 0, $$
$$ \ddot{\theta} + (1 + \bar{\epsilon}) \beta^2 \theta - \bar{\epsilon} h \alpha^2 x + \bar{\epsilon} \dot{\theta}^2 \theta = 0, $$
in which
$$ \bar{\epsilon} = \frac{m}{M}, \quad \alpha = \sqrt{\frac{k}{M}}, \quad \beta = \sqrt{\frac{g}{L}}, \quad h = \frac{1}{\bar{\epsilon} L}, $$
and $g$ is the acceleration of gravity. The parameter values are $m = 0.5$, $M = 1$, $L = 1$, and $k = 5$. We solve the system with the initial condition $(x_0, \dot{x}_0, \theta_0, \dot{\theta}_0) = (0.1, 0, -0.1, 0)$ and $dt = 0.27$ up to $t = 540$. The sliding window is done using $f = x$, $d = 4$, $\tau = 108.05$, and $\epsilon = 0.027$.

\begin{figure}[!b]
    \centering
    \includegraphics[width=0.29\textwidth, trim=.2cm .2cm .2cm .2cm, clip]{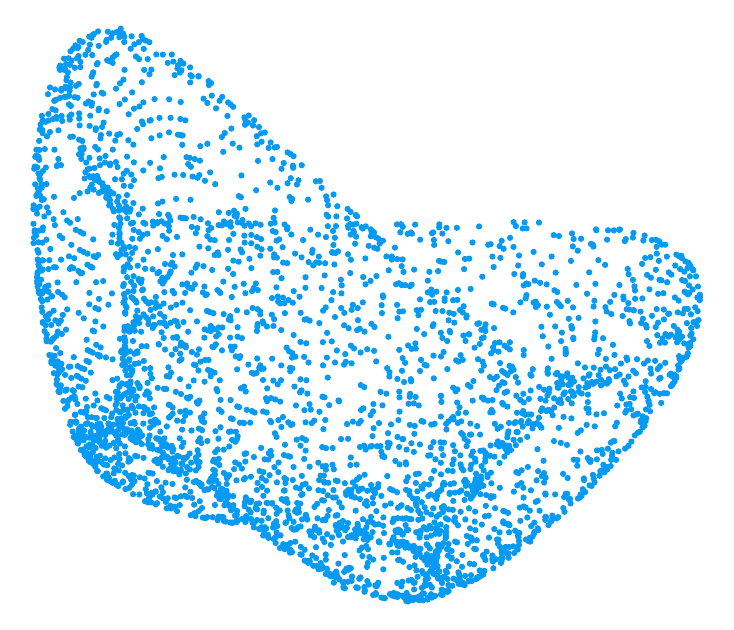}
    \hfill
    \includegraphics[width=0.32\textwidth]{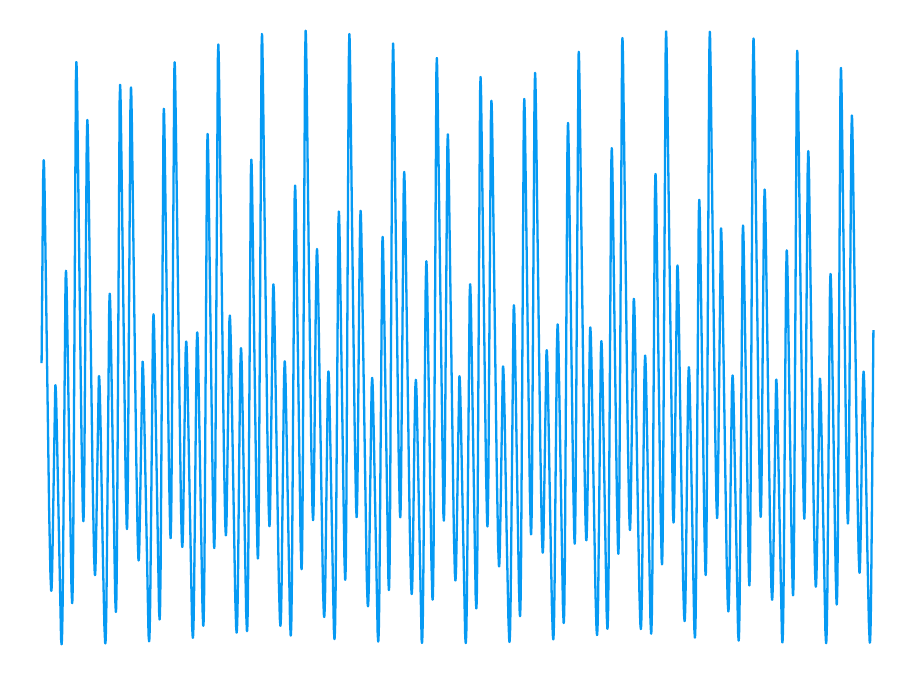}
    \hfill
    \includegraphics[width=0.32\textwidth]{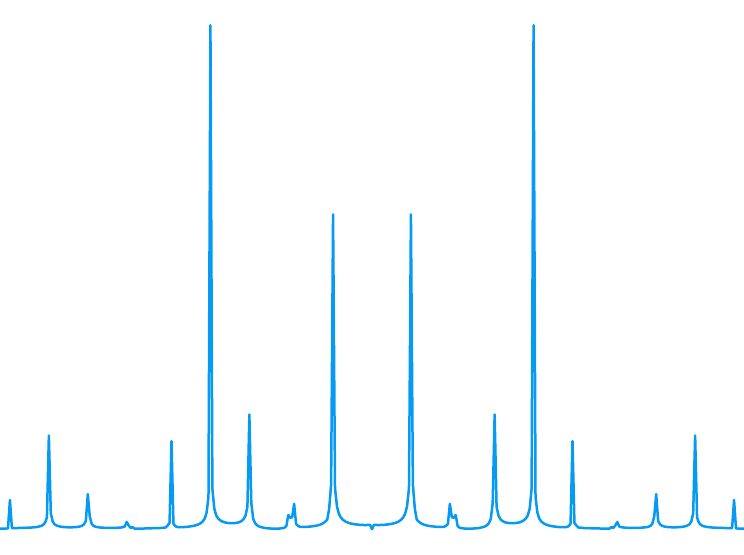}
    \hfill
    \includegraphics[width=0.48\textwidth]{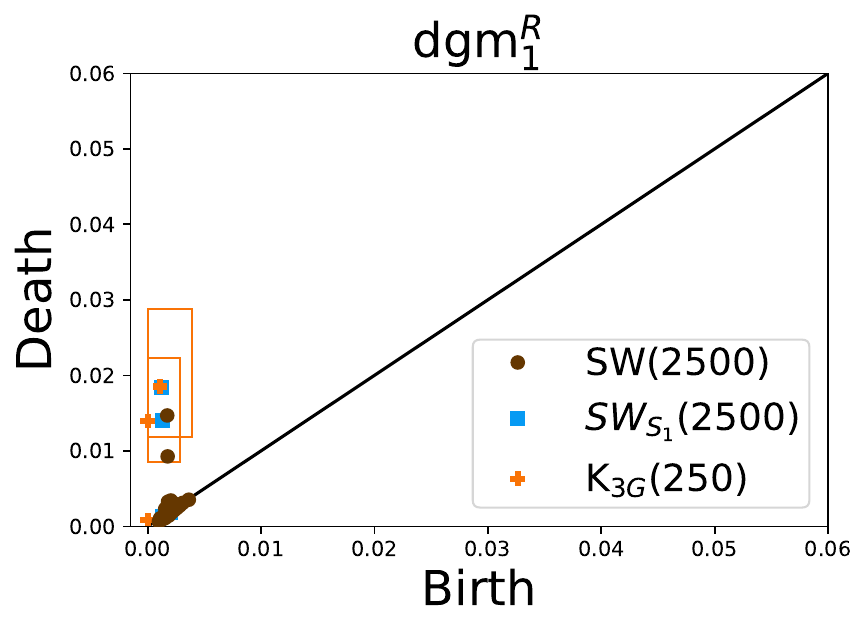}
    \hfill
    \includegraphics[width=0.48\textwidth]{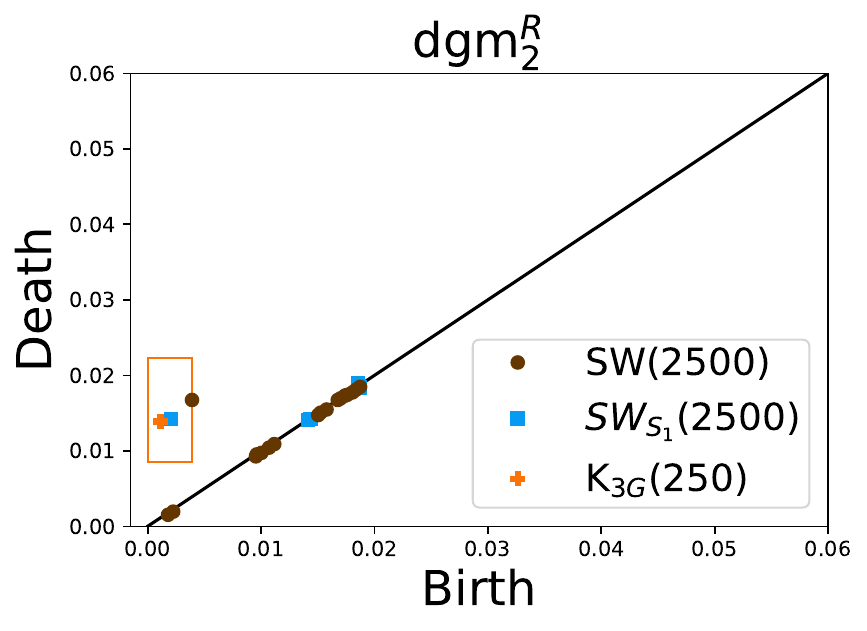}

    \caption{Top row: Phase portrait \((u,v)\) of the generalized Wilson–Cowan equations (left), the scalar trajectory \(u(t)\) used for sliding window embedding (center), and its Fourier power spectrum via FFT (right). Bottom row: Persistence diagrams of the Rips filtration on the sliding window point cloud derived from \(u(t)\).}
    \label{fig:QWplot}
\end{figure}

\subsubsection{Generalized Wilson-Cowan Equations}
\label{ssec:example4}
We consider a generalized version of the Wilson-Cowan equations. Traditionally, these equations are derived via a time-coarse graining technique that averages the response. They also restrict to a weak Gamma distribution of time delays. The extended model has been treated in \cite{kaslik2022stability}, it is given by
$$ \dot{u}(t) = -u(t) + f_1 \left( \theta_u + \int_{-\infty}^{t} h(t - s)(a u(s) + b v(s)) \, ds \right), $$
$$ \dot{v}(t) = -v(t) + f_2 \left( \theta_v + \int_{-\infty}^{t} h(t - s)(c u(s) + d v(s)) \, ds \right), $$
where $u(t)$ and $v(t)$ model the firing activity in two neuronal populations, $a, b, c,$ and $d$ are the connection weights, and $\theta_u$ and $\theta_v$ are background drives. The activation functions $f_1$ and $f_2$ are smooth and increasing on the real line. The authors considered three types of delay kernels $h: [0, \infty) \rightarrow [0, \infty)$, namely, weak Gamma, strong Gamma, and Dirac kernels. Their analysis indicates a preferable kernel based on the function of the model populations.

One can readily verify this system recovers the Wilson-Cowan equations in the case of a weak Gamma kernel. We consider the case of a Dirac kernel, see Figure \ref{fig:QWplot}. This system has been shown to exhibit quasiperiodicity \cite{kaslik2022stability,coombes2009delays}. Furthermore, it has applications to the subthalamic nucleus - globus pallidus network involved in Parkinson's Disease (guided by anatomical and electrophysiological research) \cite{holgado2010conditions}. The system of interest becomes
$$ \dot{u}(t) = -u(t) + f_1 \left( \theta_u + a u(t - \tau_1) + b v(t - \tau_2) \right), $$
$$ \dot{v}(t) = -v(t) + f_2 \left( \theta_v + c u(t - \tau_2) + d v(t - \tau_1) \right), $$
where
$$ f_1(x) = f_2(x) = \frac{1}{1 + e^{-\delta x}}. $$
The parameter values are $\theta_u = 0.1$, $\theta_v = 0.2$, $\tau_1 = \tau_2 = 0.152$, $a = d = -19$, $b = c = 10$, and $\delta = 10$. We solve the system with initial conditions $(u_0, v_0) = (0.05, 0.05)$ and $dt = 0.001$ up to $t = 50$. The sliding window is done using $f = u$, $d = 4$, $\tau = 1.712$, and $\epsilon = 0.01$.

\subsubsection{Electromagnetic Radiation on a Wilson Neuron Model}
\label{ssec:example5}
Wilson introduced a simplified model for a neocortical neuron by making assumptions on the Hodgkin-Huxley model \cite{wilson1999simplified}. The biophysics of these neurons is governed by the interplay of about a dozen ion currents. His model showed the need for only four ion currents to accurately replicate known spiking behavior. We analyze an extension of his model that incorporates the presence of electromagnetic radiation (EMR).

\begin{figure}[!b]
    \centering
    \includegraphics[width=0.27\textwidth, trim=2.7cm 2.7cm 2.7cm 2.7cm, clip]{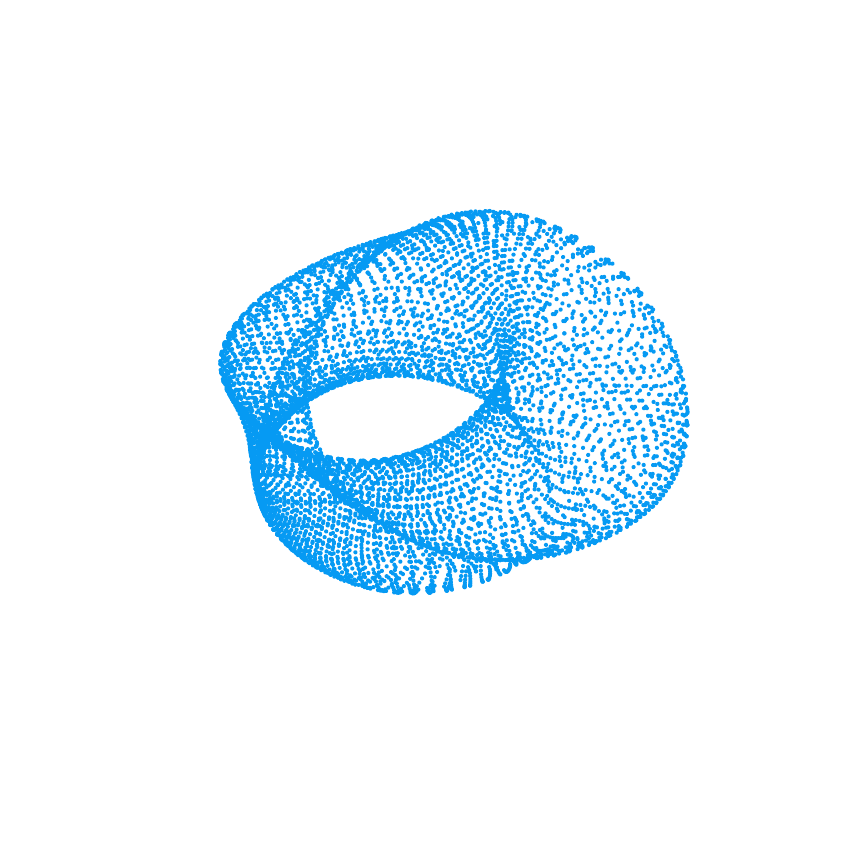}
    \hfill
    \includegraphics[width=0.32\textwidth]{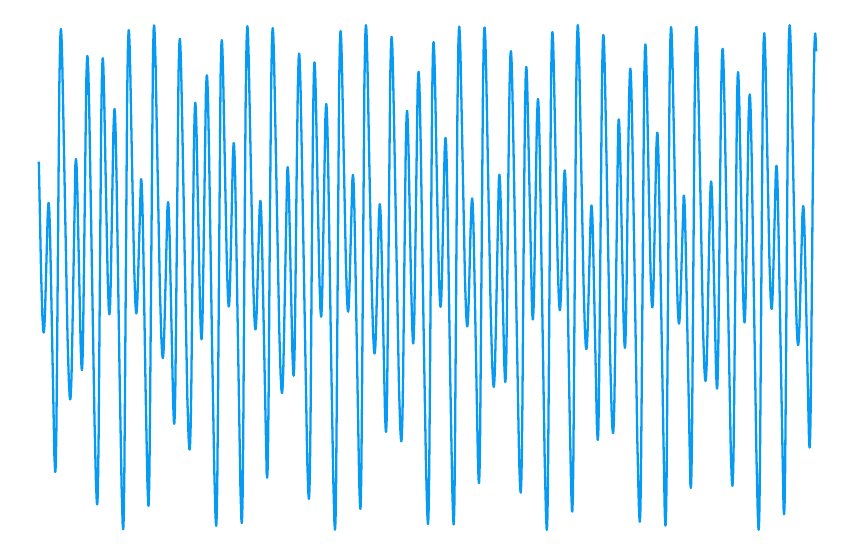}
    \hfill
    \includegraphics[width=0.32\textwidth]{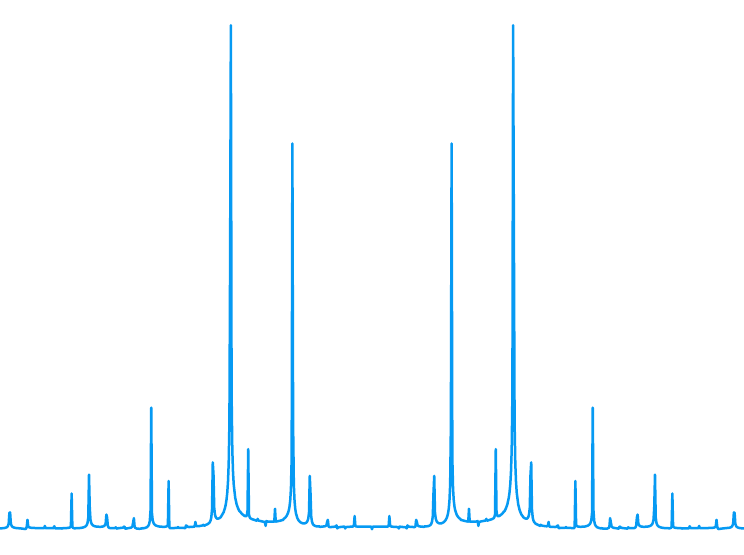}
    \hfill
    \includegraphics[width=0.48\textwidth]{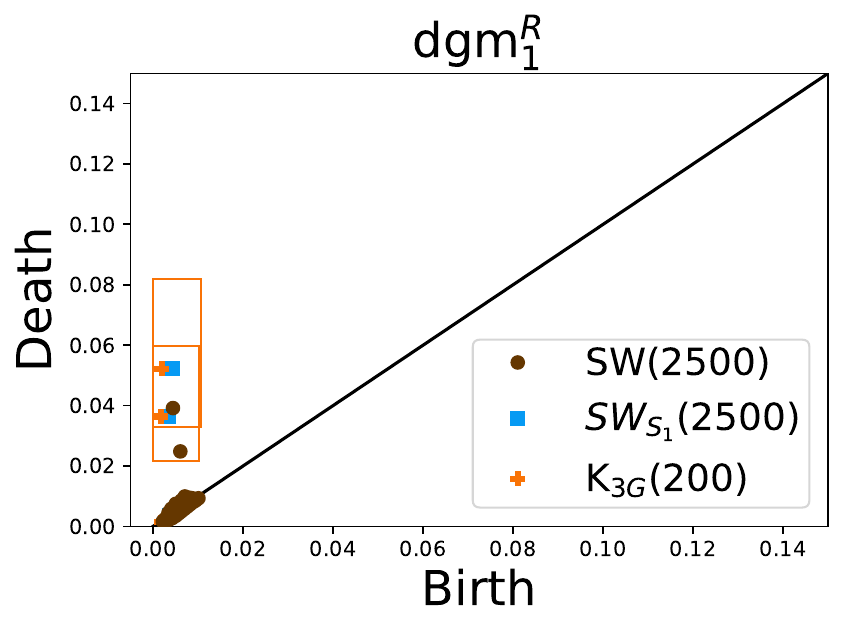}
    \hfill
    \includegraphics[width=0.48\textwidth]{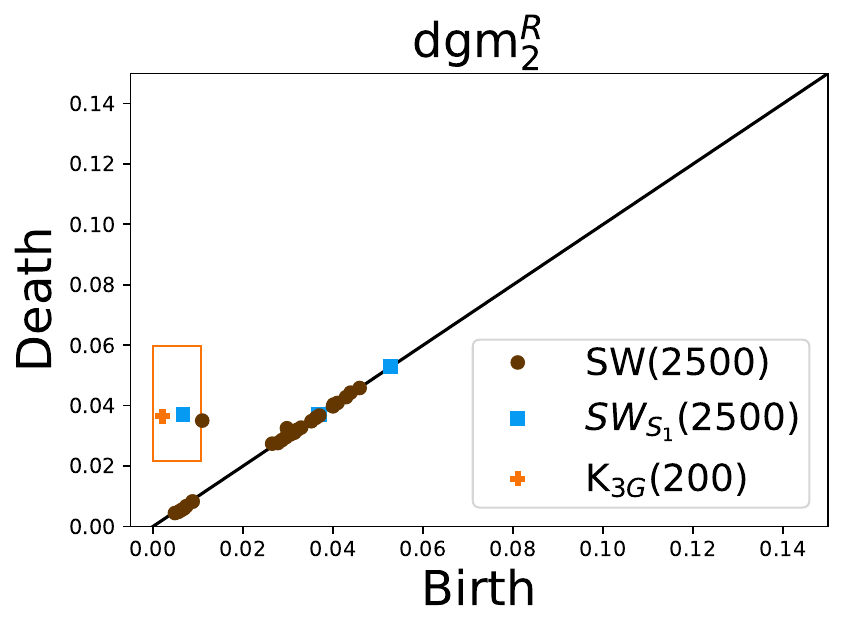}

    \caption{Top row: Phase portrait \((v,r,\phi)\) of the system (left), the radial coordinate trajectory \(r(t)\) used for sliding window embedding (center), and its Fourier power spectrum via FFT (right). Bottom row: Persistence diagrams of the Rips filtration on the sliding window point cloud derived from \(r(t)\).}
    \label{fig:EWplot}
\end{figure}

Commonplace presence of electronic devices introduces EMR exposure to neurons. The effects EMR has was imitated with the presence of a flux-controlled memristor in \cite{ju2022electromagnetic}. Their proposed model was shown to exhibit quasiperiodicity and their results were successfully replicated using a micro-controller unit-based hardware platform. Their mathematical model describes membrane potential $v$, recovery variable $r$, and inner state of the memristor $\phi$
$$ C_m \frac{dv}{dt} = -m_{\infty}(v)(v - E_{Na}) - g_K r(v - E_K) + I_{ext} - k_1 W(\phi) v, $$
$$ \frac{dr}{dt} = \frac{1}{\tau_r}(-r + r_{\infty}(v)), \quad \frac{d\phi}{dt} = v - k_2 \phi + \phi_{ext}, $$
where
$$ m_{\infty}(v) = 17.8 + 47.6v + 33.8v^2, $$
and
$$ r_{\infty}(v) = 1.24 + 3.7v + 3.2v^2. $$
We use typical model parameters for the membrane capacitor, $C_m = 1$, reverse potential for sodium and potassium, $E_{Na} = 0.5$ and $E_K = -0.95$, respectively, maximal conductance of potassium, $g_K = 26$, potassium ion channel activation time constant, $\tau_r = 5$, and external stimulus current $I_{ext} = 1$ \cite{xu2021analogy}. The EMR external contribution is given by $\phi_{ext} = A \sin(2\pi F t)$, the terms $k_1 W(\phi) v$ and $k_2 \phi$ denote the induction current caused by variation of magnetic flux and the leakage of magnetic flux. The memductance of the memristor is given by $W(\phi) = a - b \tanh(\phi)$ \cite{bao2017coexisting}. The remaining parameter values are $a = 1$, $b = 3$, $k_1 = 1.2$, $k_2 = 0.5$, $A = 0.35$, and $F = 0.22$. We solve the system using the initial conditions $(v_0, r_0, \phi_0) = (0, 0, 0)$ and $dt = 0.01$ up to $t = 500$. The sliding window was done with $f = r$, $d = 4$, $\tau = 72.21$, and $\epsilon = 0.07$, see Figure \ref{fig:EWplot}.

\begin{figure}[!b]
    \centering
    \includegraphics[width=0.25\textwidth, trim=2.1cm 2.1cm 2.1cm 2.1cm, clip]{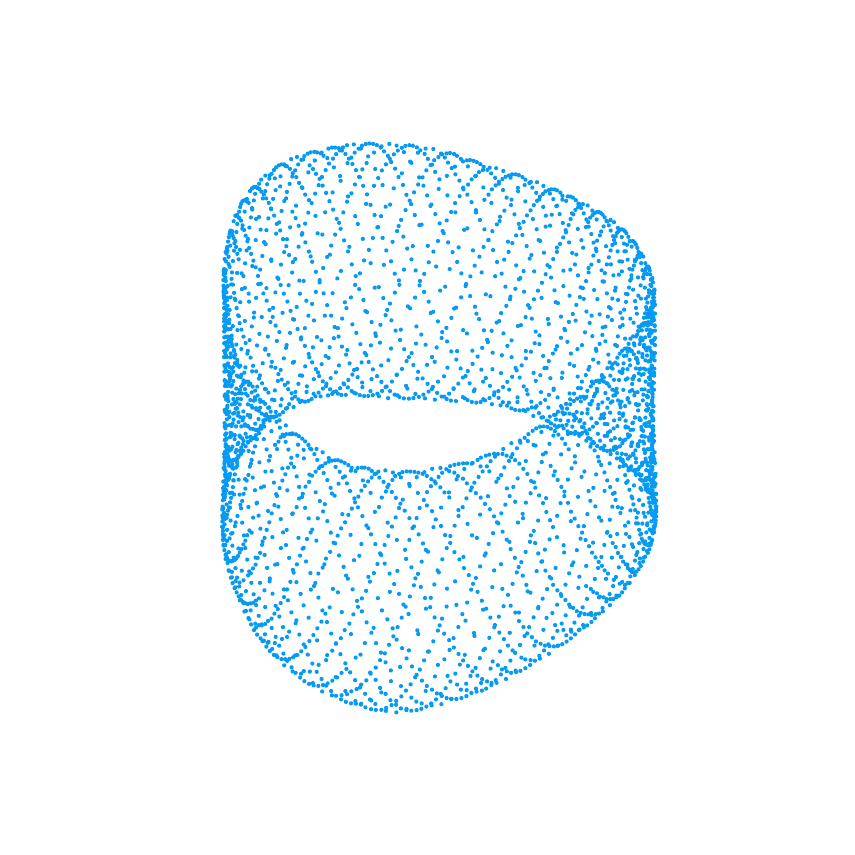}
    \hfill
    \includegraphics[width=0.32\textwidth]{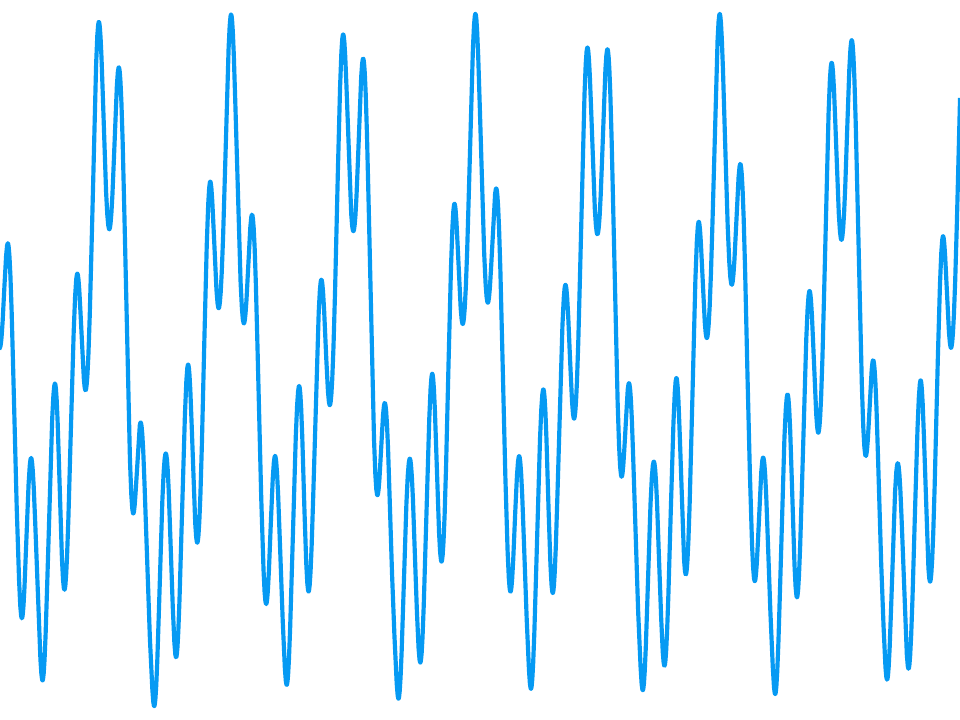}
    \hfill
    \includegraphics[width=0.32\textwidth]{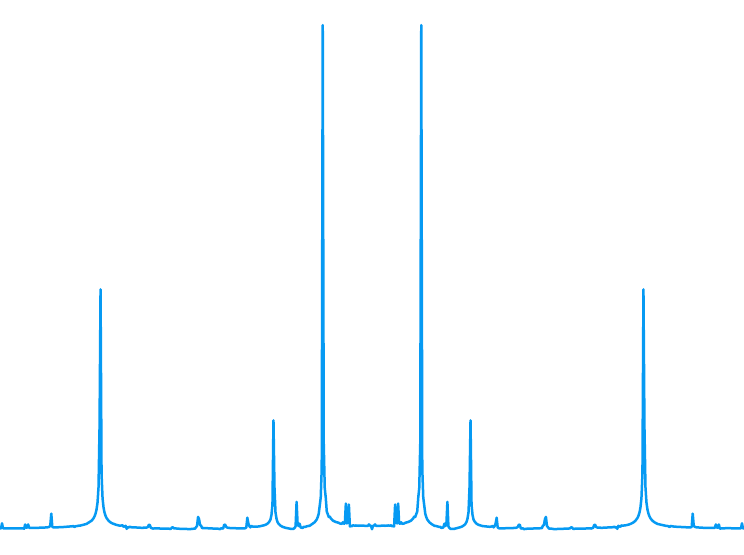}
    \hfill
    \includegraphics[width=0.48\textwidth]{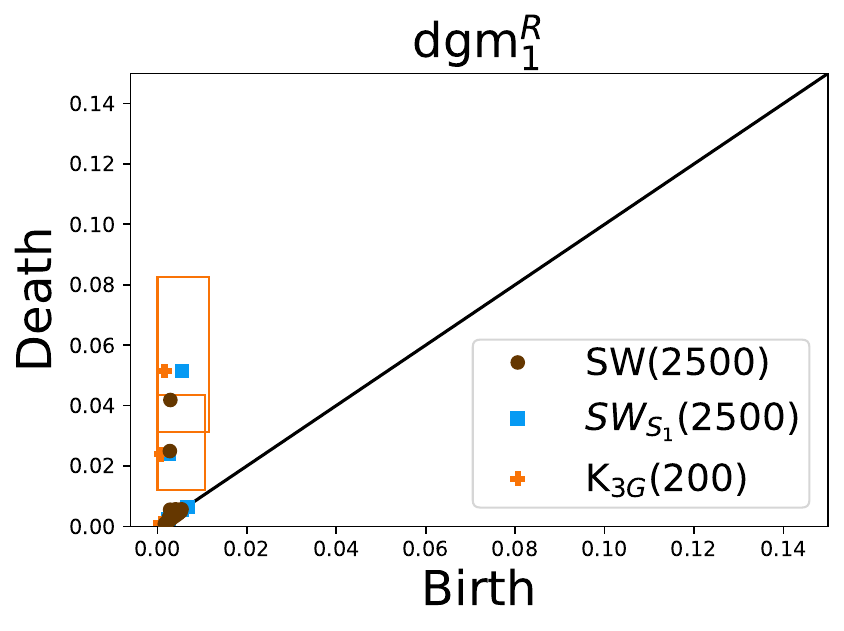}
    \hfill
    \includegraphics[width=0.48\textwidth]{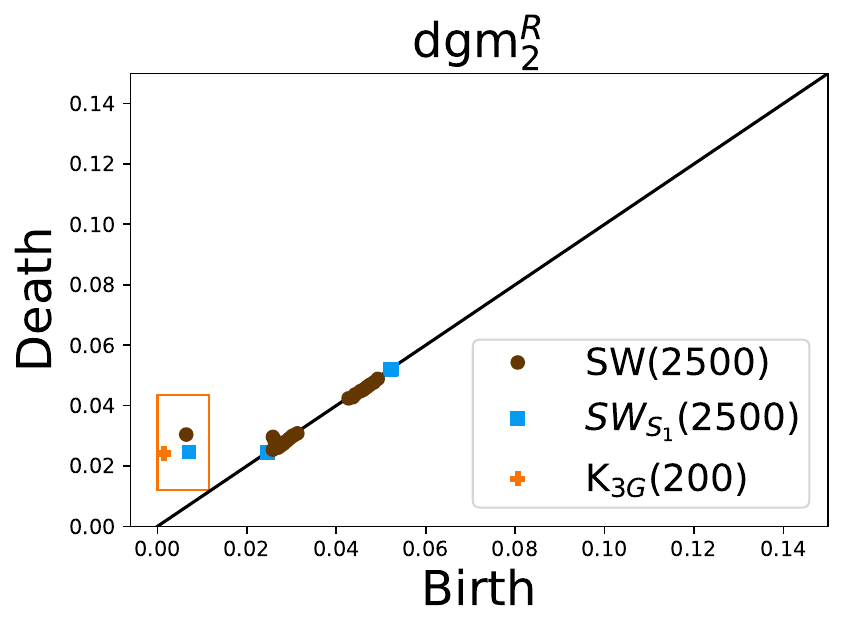}

    \caption{Top row: Three dimensional PCA projection of trajectories from the competitive TLN model (left), the scalar trajectory \(x_4(t)\) used for sliding window embedding (center), and its Fourier power spectrum via FFT (right). Bottom row: Persistence diagrams of the Rips filtration on the sliding window point cloud derived from \(x_4(t)\).}
    \label{fig:TLNplot}
\end{figure}

\subsubsection{Competitive Threshold-Linear Network}
\label{ssec:example6}
Threshold-linear networks (TLNs) provide an accessible model incorporating a threshold nonlinearity \cite{morrison2024diversity}. TLNs refine linear approximations of networks. The TLN model can be described by the equations:
$$ \frac{dx_i}{dt} = -x_i + \left[ \sum_{j=1}^{n} W_{ij}x_j + b_i \right]_+ $$
where $i = 1, \dots, n$ and $x_i$ represents the activity level of node $i$. The term $W_{ij}$ is the directed connection strength, $b_i$ is the external drive, and $[y]_+ = \max\{y, 0\}$ is the threshold nonlinearity.

We consider competitive TLNs, where $W_{ij} \leq 0$, $W_{ii} = 0$, and $b_i \geq 1$. Furthermore, we consider the connection strength:
$$ W_{ij} = \begin{cases}
0 & \text{if } i = j, \\
-1 + \lambda & \text{if node $j$ is connected to node $i$}, \\
-1 - \delta & \text{otherwise}.
\end{cases} $$

This model has shown complex behavior, including quasiperiodicity. We replicate the quasiperiodic behavior of Figure 2 in \cite{morrison2024diversity}. We refer to it for the initial condition and connection matrix. The remaining parameter values are $n = 25$, $b_i = 1$, $\lambda = 0.25$, and $\delta = 0.5$. We solve up to $t = 600$ and perform the sliding window with $f = x_4$, $d = 4$, $\tau = 27.53$, and $\epsilon = 0.07$, see Figure \ref{fig:TLNplot}.

\begin{figure}[!b]
    \centering
    \includegraphics[width=0.35\textwidth, trim=2.2cm 3.9cm 2.9cm 4.9cm, clip]{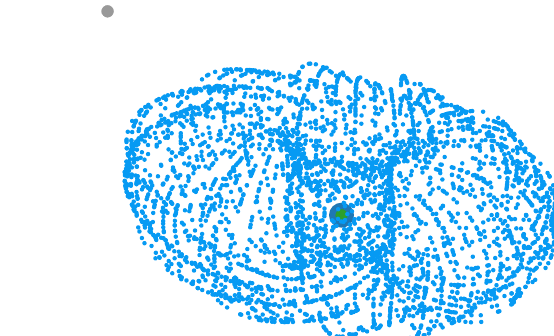}
    \hfill
    \includegraphics[width=0.27\textwidth]{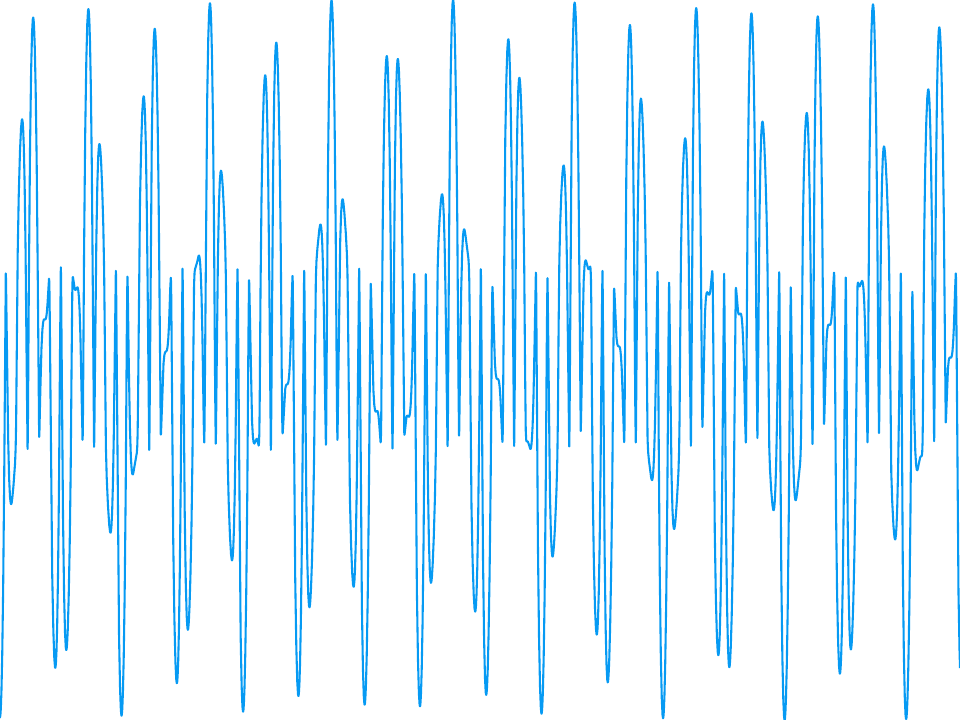}
    \hfill
    \includegraphics[width=0.27\textwidth]{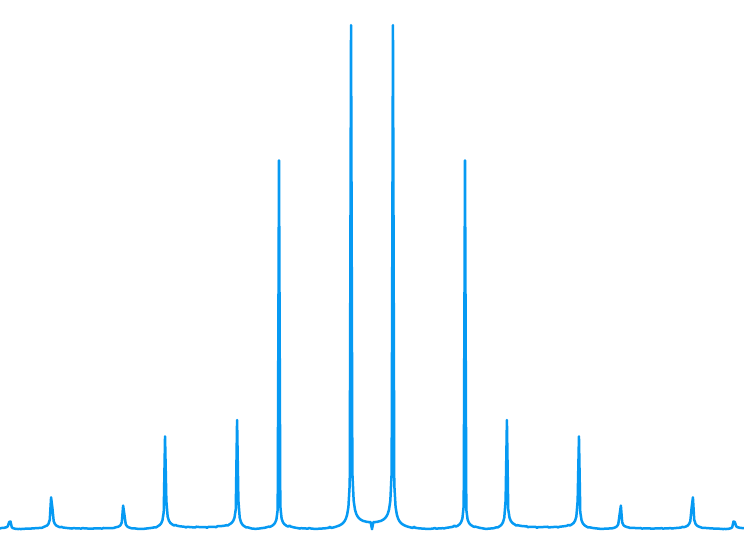}
    \hfill
    \includegraphics[width=0.47\textwidth]{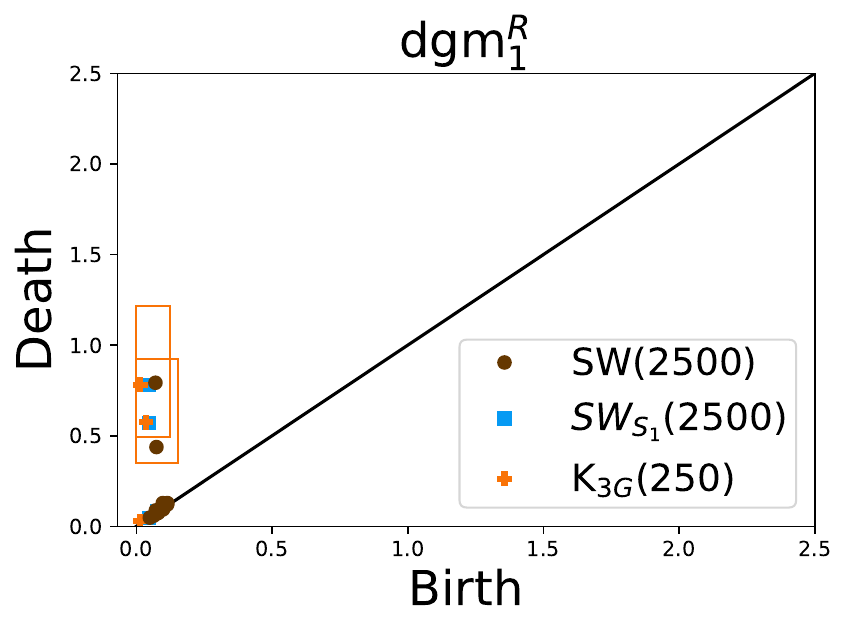}
    \hfill
    \includegraphics[width=0.47\textwidth]{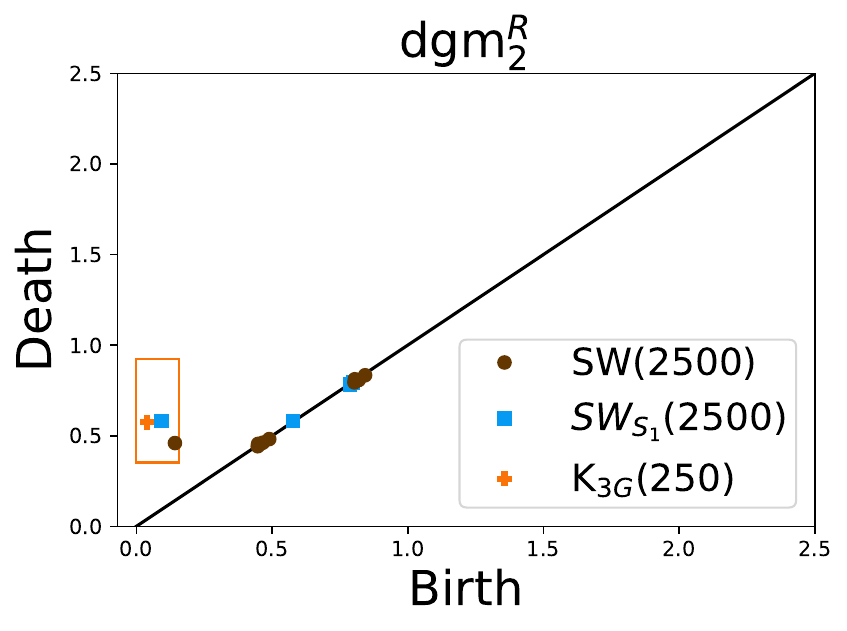}

    \caption{Top row: Phase portrait of the restricted three‐body problem in \((x,y,z)\)‐space, showing the Earth (blue/green) and the Moon (gray) as fixed primaries (left), the scalar trajectory \(x(t)\) used for sliding‐window embedding (center), and its Fourier power spectrum via FFT (right). Bottom row: Persistence diagrams of the Rips filtrations on the sliding window point cloud derived from \(x(t)\).}
    \label{fig:3BPplot}
\end{figure}

\subsubsection{Restricted Three Body Problem}
\label{ssec:example7}
In celestial mechanics, the restricted three body problem (RTBP) offers an accessible model with known equilibrium points \cite{koon2000heteroclinic}. The equations describe the progression of three celestial bodies in which one of them is considered a massless particle. The other two bodies, denotes as primaries, are assumed to move in circular orbits around their center of mass. By taking a coordinate system that rotates with the primaries and under the appropriate scaling, it can be assumed the primaries have masses $1-\mu$ and $\mu$, $\mu \in [0, 1/2]$, are fixed at $(\mu, 0, 0)$ and $(\mu - 1, 0, 0)$, respectively, and complete one revolution in $2\pi$ \cite{gomez2001dynamics}. This framework allows us to express the motion of the massless particle by the equations:
$$
\begin{aligned}
\ddot{x} &- 2\dot{y} = \Omega_x, \\
\ddot{y} &+ 2\dot{x} = \Omega_y, \\
\ddot{z} &= \Omega_z,
\end{aligned}
$$
where
$$
\Omega = \frac{1}{2}(x^2 + y^2) + \frac{1 - \mu}{r_1} + \frac{\mu}{r_2} + \frac{1}{2}\mu(1 - \mu),
$$
and
$$
r_1 = \sqrt{(x - \mu)^2 + y^2 + z^2}, \quad
r_2 = \sqrt{(x - \mu + 1)^2 + y^2 + z^2}
$$
are the distances from the particles to the primaries.

The case of the Earth-Moon system is of particular interest since it can aid in spacecraft missions interested in the Sun and the magnetosphere of the Earth \cite{gomez2001dynamics}. Indeed, near the equilibrium points of the system, quasiperiodicity is present which translates to nice trajectories for a mission. We replicate this behavior for the Earth-Moon system, see Figure \ref{fig:3BPplot}. In this case, $\mu = 0.0121506$. We solve the system with the initial condition $(x_0, y_0, z_0, \dot{x}, \dot{y}, \dot{z}) = (-0.5, 0, 0, 0, 0, 0.73)$ up to $t = 100$. The sliding window was done with $f = x$, $d = 4$, $\tau = 4.37$, and $\epsilon = 0.03$.

\begin{figure}[!b]
    \centering
    \includegraphics[width=0.30\textwidth, trim=2cm 2.5cm 2cm 2.3cm, clip]{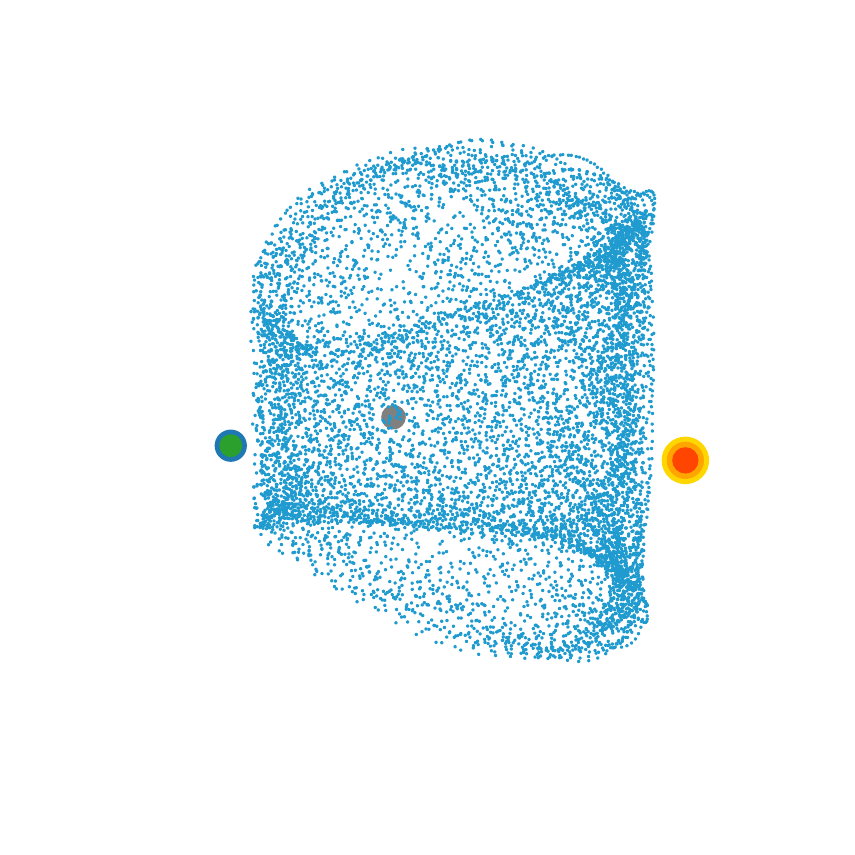}
    \hfill
    \includegraphics[width=0.32\textwidth]{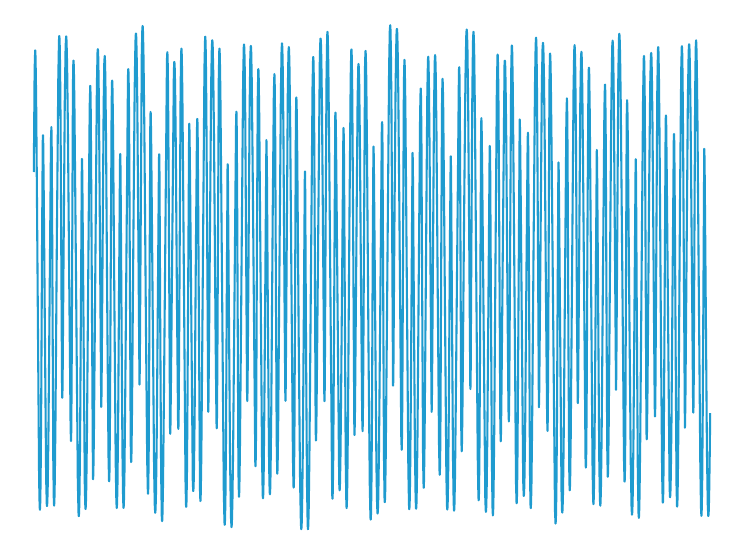}
    \hfill
    \includegraphics[width=0.30\textwidth]{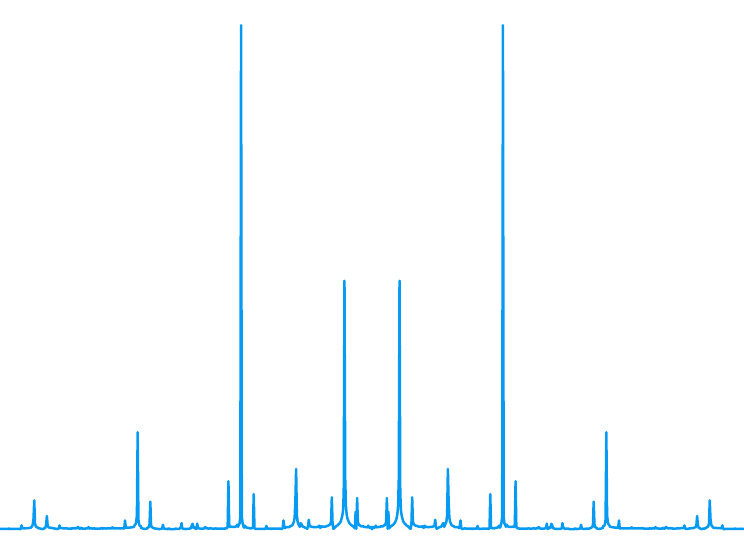}
    \hfill
    \includegraphics[width=0.48\textwidth]{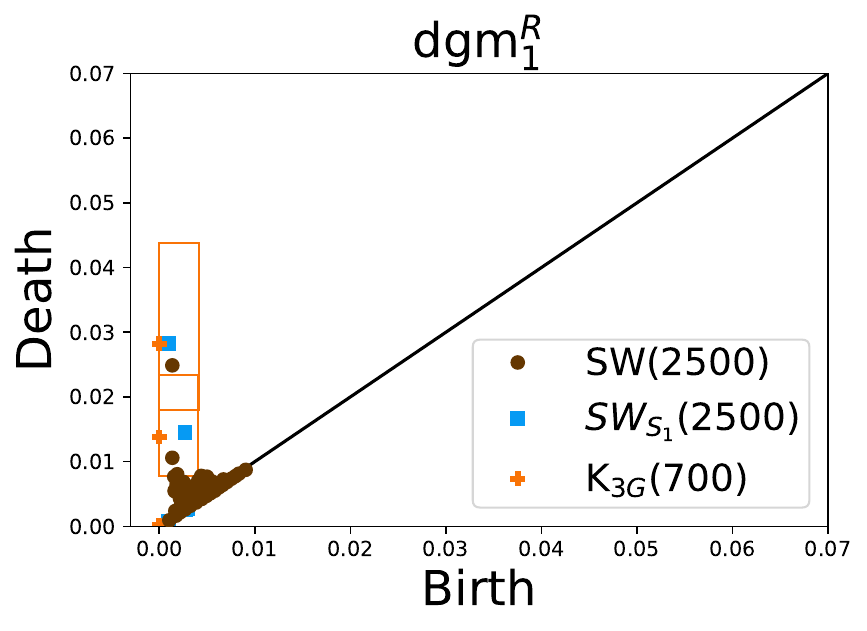}
    \hfill
    \includegraphics[width=0.48\textwidth]{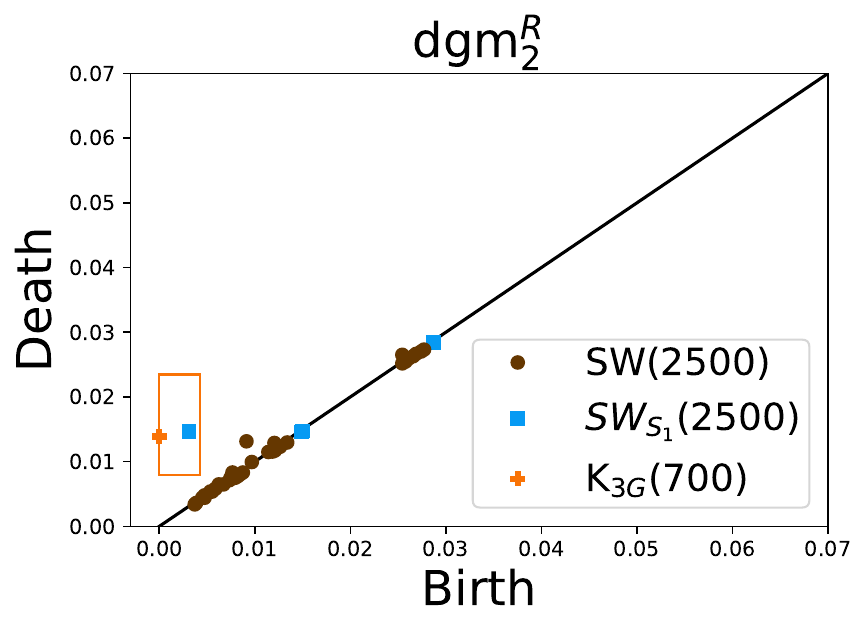}

    \caption{Top row: Phase portrait of the bicircular restricted four‐body problem in \((x,y,z)\)‐space, showing Earth (blue/green), Moon (gray), and Sun (red/yellow) as fixed primaries (positions not to scale). Center: the scalar trajectory \(z(t)\) used for sliding window embedding. Right: its Fourier power spectrum via FFT. Bottom row: Persistence diagrams of the Rips filtrations on the sliding window point cloud derived from \(z(t)\).}
    \label{fig:4BPplot}
\end{figure}

\subsubsection{Bicircular Restricted Four Body Problem}
\label{ssec:example9}
Considering the effects of the sun in the RTBP model gives rise to the bicircular restricted four body problem (BCR4BP). The model is of importance for exploiting the force from the sun in trajectory designs for lunar missions and has found applications for ballistic lunar transfers to the lunar region \cite{mccarthy2023four}. The new model uses the same coordinate axes as the RTBP and the same circular motion assumptions of the Earth and Moon but it now includes terms pertaining to the influence of the Sun. The Sun is assumed to lie in the x-y plane at a fixed distance to the origin, $a_4$, and moving with a constant angular velocity, $\dot{\theta}_S$, in circular motion. Furthermore, the model assumes the Earth and Moon are not perturbed by solar gravity. Under these assumptions, the equations are \cite{mccarthy2023four}
$$
\begin{aligned}
\ddot{x} &= 2\dot{y} + \frac{\partial \Upsilon}{\partial x}, \\
\ddot{y} &= -2\dot{x} + \frac{\partial \Upsilon}{\partial y}, \\
\ddot{z} &= \frac{\partial \Upsilon}{\partial z},
\end{aligned}
$$
where
$$
\Upsilon = \frac{1 - \mu}{r_{13}} + \frac{\mu}{r_{23}} + \frac{x^2 + y^2}{2}
+ \lambda \left( \frac{m_4}{r_{43}} - \frac{m_4}{a_4^3} (x_4 x + y_4 y + z_4 z) \right).
$$
and $(x_4, y_4, z_4)$ denotes the position of the sun, $r_{13}, r_{23}, r_{43}$ the distance of the Earth to the massless object, the Moon to the massless object, and the Sun to the massless object, respectively, and $m_4$ is the non-dimensional mass of the Sun. We note the case $\lambda = 0$ reduces to the RTBP. A study of the system as $\lambda$ increases to $1$ is done in \cite{mccarthy2023four}, where they showed the model exhibited quasiperiodicity. We replicate the quasiperiodic behavior, see Figure \ref{fig:4BPplot}, for parameter values $\lambda = 1$, $\mu = 0.012155$, $a_4 = 388.84$, $m_4 = 328950.69$, and $\dot{\theta}_S = -0.9251986$. The system was solved with the initial condition $(x_0, y_0, z_0, \dot{x}, \dot{y}, \dot{z}, \theta_S) = (1.09, 0, 0, 0, 0.19, 0.05, 2)$ up to $t = 200$. The sliding window was done with $f = z$, $d = 4$, $\tau = 39.99$, and $\epsilon = 0.009$.

\definecolor{lightgray}{RGB}{230, 230, 230}
\definecolor{darkgray}{RGB}{204, 204, 204}
\setlength{\arrayrulewidth}{0.5pt} 
\definecolor{navyblue}{RGB}{137, 207, 240}

\subsection{Concluding Remarks}
\label{sec:Con}

\begin{table}[ht]
\centering
\caption{Running Times}
\begin{tabular}{|c|c|c|c|}
\hline
\rowcolor{navyblue}
Example  & SW  & SW(S1) & K(3G) \\ \hline
\rowcolor{lightgray}
5.2.1 & 7008.66 sec & 7672.85 sec & 0.87 sec \\ \hline
\rowcolor{darkgray}
5.2.2 & 4351.15 sec & 5351.15 sec & 0.42 sec \\ \hline
\rowcolor{lightgray}
5.2.3 & 3126.81 sec & 3218.73 sec & 0.50 sec \\ \hline
\rowcolor{darkgray}
5.2.4 & 5556.86 sec & 7628.48 sec & 0.98 sec \\ \hline
\rowcolor{lightgray}
5.2.5 & 5137.54 sec & 6505.54 sec & 0.66 sec \\ \hline
\rowcolor{darkgray}
5.2.6 & 5805.28 sec & 7193.97 sec & 0.22 sec \\ \hline
\rowcolor{lightgray}
5.2.7 & 6900.29 sec & 7937.03 sec & 2.09 sec \\ \hline
\rowcolor{darkgray}
5.2.8 & 6615.33 sec & 8231.25 sec & 0.59 sec \\ \hline
\end{tabular}
\label{tab:RT}
\end{table}

In this section, we detailed how 3G can be implemented on general quasiperiodic functions. Our approximation method was used on dynamical systems known to exhibit quasiperiodicity. As our figures show, we successfully approximated the diagrams obtained from SW. Error bounds can also be computed for our method, illustrated in the plots as rectangles. Furthermore, our approach significantly reduces computational time, as shown in Table \ref{tab:RT}. We hope our work contributes to the implementation of the sliding window embedding technique on large data sets.

\section{Conclusions and Future Work}

We detailed how SW is a vital method used for the better understanding of quasiperiodic signals. The correspondence between a quasiperiodic function of $N$ incommensurate frequencies and an $N$-dimensional torus makes persistent homology an ideal component. However, the exponential computational cost of the latter limits the datasets that can be impacted by SW. Our contribution is providing an approximation to the persistence diagrams of interest (Section \ref{sec:AM}). Our method, 3G, achieves this within known error bounds. On average, our computation took less than 1 second to run for the examples depicted in Section \ref{sec: App}. This contrasts with an average computation time of 5,563 seconds when using a standard library, Ripser \cite{bauer2016ripser}, to compute the persistence diagrams of interest (Table \ref{tab:RT}).

Future work for our project would involve improving our approximation by tightening the error bounds. A more granular approach would be to inspect the specific transformation matrix instead of relying on a singular value argument. Another direction is to apply our work to real-world data, such as electrocardiogram datasets.



\begin{thebibliography}{50}

\bibitem{bao2017coexisting}
{\sc Bao, B., Qian, H., Xu, Q., Chen, M., Wang, J., and Yu, Y.}
\newblock Coexisting behaviors of asymmetric attractors in hyperbolic-type
  memristor based hopfield neural network.
\newblock {\em Frontiers in Computational Neuroscience 11\/} (2017), 81.

\bibitem{bauer2016ripser}
{\sc Bauer, U.}
\newblock {\em Ripser}, vol.~514.
\newblock URL: \url{https://github.com/Ripser/ripser}, 2016.

\bibitem{bauer2021ripser}
{\sc Bauer, U.}
\newblock Ripser: a lean {C++} code for the computation of vietoris--rips
  persistence barcodes.
\newblock {\em Journal of Applied and Computational Topology 5}, 3 (2021),
  391--423.

\bibitem{belloy2017dynamic}
{\sc Belloy, M., Naeyaert, M., Keliris, G., Abbas, A., Keilholz, S., Van
  Der~Linden, A., and Verhoye, M.}
\newblock Dynamic resting state fmri in mice: detection of quasi-periodic
  patterns.
\newblock {\em Proceeding of the International Soc. Magn. Reson. Med 961\/}
  (2017).

\bibitem{beresnevich2017sums}
{\sc Beresnevich, V., and Leong, N.}
\newblock Sums of reciprocals and the three distance theorem.
\newblock {\em arXiv preprint arXiv:1712.03758\/} (2017).

\bibitem{charo2020topology}
{\sc Char{\'o}, G.~D., Artana, G., and Sciamarella, D.}
\newblock Topology of dynamical reconstructions from lagrangian data.
\newblock {\em Physica D: Nonlinear Phenomena 405\/} (2020), 132371.

\bibitem{chazal2016structure}
{\sc Chazal, F., De~Silva, V., Glisse, M., and Oudot, S.}
\newblock {\em The structure and stability of persistence modules}, vol.~10.
\newblock Springer, 2016.

\bibitem{cohen2007stability}
{\sc Cohen-Steiner, D., Edelsbrunner, H., and Harer, J.}
\newblock Stability of persistence diagrams.
\newblock {\em Discrete \& Computational Geometry 37}, 1 (2007), 103.

\bibitem{coombes2009delays}
{\sc Coombes, S., and Laing, C.}
\newblock Delays in activity-based neural networks.
\newblock {\em Philosophical Transactions of the Royal Society A: Mathematical,
  Physical and Engineering Sciences 367}, 1891 (2009), 1117--1129.

\bibitem{crawley2015decomposition}
{\sc Crawley-Boevey, W.}
\newblock Decomposition of pointwise finite-dimensional persistence modules.
\newblock {\em Journal of Algebra and its Applications 14}, 05 (2015), 1550066.

\bibitem{detroux2015performance}
{\sc Detroux, T., Habib, G., Masset, L., and Kerschen, G.}
\newblock Performance, robustness and sensitivity analysis of the nonlinear
  tuned vibration absorber.
\newblock {\em Mechanical Systems and Signal Processing 60\/} (2015), 799--809.

\bibitem{gakhar2019k}
{\sc Gakhar, H., and Perea, J.~A.}
\newblock K{\"u}nneth formulae in persistent homology.
\newblock {\em arXiv preprint arXiv:1910.05656\/} (2019).

\bibitem{gakhar2021sliding}
{\sc Gakhar, H., and Perea, J.~A.}
\newblock Sliding window persistence of quasiperiodic functions.
\newblock {\em Journal of Applied and Computational Topology 8}, 1 (2024),
  55--92.

\bibitem{gomez2001dynamics}
{\sc G{\'o}mez, G., and Mondelo, J.~M.}
\newblock The dynamics around the collinear equilibrium points of the rtbp.
\newblock {\em Physica D: Nonlinear Phenomena 157}, 4 (2001), 283--321.

\bibitem{MR1867354}
{\sc Hatcher, A.}
\newblock {\em Algebraic topology}.
\newblock Cambridge University Press, Cambridge, 2002.

\bibitem{hlawka2012geometric}
{\sc Hlawka, E., Thomas, C., Schoi{\ss}engeier, J., and Taschner, R.}
\newblock {\em Geometric and Analytic Number Theory}.
\newblock Universitext. Springer Berlin Heidelberg, 2012.

\bibitem{holgado2010conditions}
{\sc Holgado, A. J.~N., Terry, J.~R., and Bogacz, R.}
\newblock Conditions for the generation of beta oscillations in the subthalamic
  nucleus--globus pallidus network.
\newblock {\em Journal of Neuroscience 30}, 37 (2010), 12340--12352.

\bibitem{ju2022electromagnetic}
{\sc Ju, Z., Lin, Y., Chen, B., Wu, H., Chen, M., and Xu, Q.}
\newblock Electromagnetic radiation induced non-chaotic behaviors in a wilson
  neuron model.
\newblock {\em Chinese Journal of Physics 77\/} (2022), 214--222.

\bibitem{kaslik2022stability}
{\sc Kaslik, E., Kokovics, E.-A., and R{\u{a}}dulescu, A.}
\newblock Stability and bifurcations in wilson-cowan systems with distributed
  delays, and an application to basal ganglia interactions.
\newblock {\em Communications in Nonlinear Science and Numerical Simulation
  104\/} (2022), 105984.

\bibitem{khasawneh2016chatter}
{\sc Khasawneh, F.~A., and Munch, E.}
\newblock Chatter detection in turning using persistent homology.
\newblock {\em Mechanical Systems and Signal Processing 70\/} (2016), 527--541.

\bibitem{kim2022exact}
{\sc Kim, K., and Jung, J.-H.}
\newblock Exact multi-parameter persistent homology of time-series data: Fast
  and variable one-dimensional reduction of multi-parameter persistence theory.
\newblock {\em arXiv preprint arXiv:2211.03337\/} (2022).

\bibitem{koon2000heteroclinic}
{\sc Koon, W.~S., Lo, M.~W., Marsden, J.~E., and Ross, S.~D.}
\newblock Heteroclinic connections between periodic orbits and resonance
  transitions in celestial mechanics.
\newblock {\em Chaos: An Interdisciplinary Journal of Nonlinear Science 10}, 2
  (2000), 427--469.

\bibitem{leong2017sums}
{\sc Leong, N.}
\newblock {\em Sums of Reciprocals and the Three Distance Theorem}.
\newblock PhD thesis, University of York, 2017.

\bibitem{lesnick2015theory}
{\sc Lesnick, M.}
\newblock The theory of the interleaving distance on multidimensional
  persistence modules.
\newblock {\em Foundations of Computational Mathematics 15}, 3 (2015),
  613--650.

\bibitem{levine1984quasicrystals}
{\sc Levine, D., and Steinhardt, P.~J.}
\newblock Quasicrystals: a new class of ordered structures.
\newblock {\em Physical review letters 53}, 26 (1984), 2477.

\bibitem{mccarthy2023four}
{\sc McCarthy, B.~P., and Howell, K.~C.}
\newblock Four-body cislunar quasi-periodic orbits and their application to
  ballistic lunar transfer design.
\newblock {\em Advances in Space Research 71}, 1 (2023), 556--584.

\bibitem{morozov2005}
{\sc Morozov, D.}
\newblock Persistence algorithm takes cubic time in worst case.
\newblock {\em BioGeometry News, Department of Computer Science, Duke
  University, Durham, NC\/} (2005).

\bibitem{morrison2024diversity}
{\sc Morrison, K., Degeratu, A., Itskov, V., and Curto, C.}
\newblock Diversity of emergent dynamics in competitive threshold-linear
  networks.
\newblock {\em SIAM Journal on Applied Dynamical Systems 23}, 1 (2024),
  855--884.

\bibitem{murudi2004seismic}
{\sc Murudi, M.~M., and Mane, S.~M.}
\newblock Seismic effectiveness of tuned mass damper (tmd) for different ground
  motion parameters.
\newblock In {\em 13th World Conference on Earthquake Engineering\/} (2004),
  vol.~2, pp.~1--8.

\bibitem{myers2022damping}
{\sc Myers, A.~D., and Khasawneh, F.~A.}
\newblock Damping parameter estimation using topological signal processing.
\newblock {\em Mechanical Systems and Signal Processing 174\/} (2022), 109042.

\bibitem{fractions1963cd}
{\sc Olds, C.}
\newblock Continued fractions.
\newblock {\em Mathematical Association of America\/} (1963).

\bibitem{penalva2018takens}
{\sc Penalva~Vadell, J.}
\newblock Takens' theorem: Proof and applications.
\newblock Master's thesis, Universitat de les Illes Balears, 2018.

\bibitem{perea2019topological}
{\sc Perea, J.~A.}
\newblock Topological time series analysis.
\newblock {\em Notices of the American Mathematical Society 66}, 5 (2019),
  686--694.

\bibitem{perea2015sw1pers}
{\sc Perea, J.~A., Deckard, A., Haase, S.~B., and Harer, J.}
\newblock Sw1pers: Sliding windows and 1-persistence scoring; discovering
  periodicity in gene expression time series data.
\newblock {\em BMC bioinformatics 16}, 1 (2015), 1--12.

\bibitem{PhysOrg2013Quake}
{\sc {Phys.org}}.
\newblock Japanese companies look to quake damping pendulums.
\newblock
  \url{https://phys.org/news/2013-08-japanese-companies-quake-damping-pendulums.html},
  August 2013.
\newblock Accessed: 2024-03-20.

\bibitem{rockett1992continued}
{\sc Rockett, A.~M., et~al.}
\newblock {\em Continued fractions}.
\newblock World Scientific, 1992.

\bibitem{shadden2005definition}
{\sc Shadden, S.~C., Lekien, F., and Marsden, J.~E.}
\newblock Definition and properties of lagrangian coherent structures from
  finite-time lyapunov exponents in two-dimensional aperiodic flows.
\newblock {\em Physica D: Nonlinear Phenomena 212}, 3-4 (2005), 271--304.

\bibitem{shannon1949communication}
{\sc Shannon, C.~E.}
\newblock Communication in the presence of noise.
\newblock {\em Proceedings of the IRE 37}, 1 (1949), 10--21.

\bibitem{skraba2012topological}
{\sc Skraba, P., De~Silva, V., and Vejdemo-Johansson, M.}
\newblock Topological analysis of recurrent systems.
\newblock In {\em NIPS 2012 Workshop on Algebraic Topology and Machine
  Learning, December 8th, Lake Tahoe, Nevada\/} (2012), pp.~1--5.

\bibitem{takens1981detecting}
{\sc Takens, F.}
\newblock Detecting strange attractors in turbulence.
\newblock In {\em Dynamical systems and turbulence, Warwick 1980}. Springer,
  1981, pp.~366--381.

\bibitem{tralie2018quasi}
{\sc Tralie, C.~J., and Perea, J.~A.}
\newblock (quasi) periodicity quantification in video data, using topology.
\newblock {\em SIAM Journal on Imaging Sciences 11}, 2 (2018), 1049--1077.

\bibitem{van1988three}
{\sc Van~Ravenstein, T.}
\newblock The three gap theorem (steinhaus conjecture).
\newblock {\em Journal of the Australian Mathematical Society 45}, 3 (1988),
  360--370.

\bibitem{weixing1993quasiperiodic}
{\sc Weixing, D., Wei, H., Xiaodong, W., and Yu, C.}
\newblock {Quasiperiodic transition to chaos in a plasma}.
\newblock {\em Physical review letters 70}, 2 (1993), 170.

\bibitem{wen2019quasi}
{\sc Wen, R., Li, T., and Zhen, B.}
\newblock Quasi-periodic motions of a pendulum with vibrating suspension point.
\newblock {\em Journal of Vibration Engineering \& Technologies 7\/} (2019),
  519--532.

\bibitem{wilden1998subharmonics}
{\sc Wilden, I., Herzel, H., Peters, G., and Tembrock, G.}
\newblock Subharmonics, biphonation, and deterministic chaos in mammal
  vocalization.
\newblock {\em Bioacoustics 9}, 3 (1998), 171--196.

\bibitem{wilson1999simplified}
{\sc Wilson, H.~R.}
\newblock Simplified dynamics of human and mammalian neocortical neurons.
\newblock {\em Journal of theoretical biology 200}, 4 (1999), 375--388.

\bibitem{xu2019twisty}
{\sc Xu, B., Tralie, C.~J., Antia, A., Lin, M., and Perea, J.~A.}
\newblock Twisty takens: A geometric characterization of good observations on
  dense trajectories.
\newblock {\em Journal of Applied and Computational Topology 3\/} (2019),
  285--313.

\bibitem{xu2021analogy}
{\sc Xu, Q., Ju, Z., Feng, C., Wu, H., and Chen, M.}
\newblock Analogy circuit synthesis and dynamics confirmation of a bipolar
  pulse current-forced 2d wilson neuron model.
\newblock {\em The European Physical Journal Special Topics 230}, 7 (2021),
  1989--1997.

\bibitem{zhigljavskykronecker}
{\sc Zhigljavsky, A., and Aliev, I.}
\newblock Kronecker sequences: Asymptotic distributions of the partition
  lengths.
\newblock URL:
  \url{https://ssa.cf.ac.uk/zhigljavsky/pdfs/number\%20theory/Kroneker\%20sequences.pdf}.

\bibitem{zomorodian2005computing}
{\sc Zomorodian, A., and Carlsson, G.}
\newblock Computing persistent homology.
\newblock {\em Discrete \& Computational Geometry 33}, 2 (2005), 249.

\end{thebibliography}

\begin{appendices}

\section{Appendix}
\begin{proposition}
\label{A1}
Let $\sigma \in C_1(R_\epsilon(S_{\omega,T}))$ and $\sigma \neq 0.$ If $\partial(\sigma)=0$, then there exists $N \in \mathbb{N}$ and $\sigma_i \in C_1(R_\epsilon(S_{\omega,T}))$, $\sigma_i \neq 0$, such that $$\sigma = \sum_{i=0}^{N-1}\sigma_i,$$ $\partial(\sigma_i)=0$, and if $\tau \subsetneq \sigma_i,$ $$ \partial(\tau)\neq0. $$ Moreover, $\sigma_i$ can be expressed as a sum of the form $$\sigma_i=\sum_{j=0}^{N_i-1}[x_j,x_{j+1}],$$ where $x_i \in S_{\omega,T}, N_i \in \mathbb{N},$ and $x_{N_i}=x_0.$
\end{proposition}
\begin{proof}
Let $\tau_1  \subsetneq \sigma $ be such that $\partial(\tau_1)=0$ and there are no $\bar{\tau}  \subsetneq \sigma $ such that $\partial(\bar{\tau})=0$  and $\tau_1  \subsetneq \bar{\tau}.$ If no such $\tau_1$ exists, then let $\sigma_1=\sigma$ and the result follows. Otherwise, let $\sigma_1=\sigma/\tau_1.$ By the maximality of $\tau_1,$ $\sigma_1$ satisfies the stated conditions. Now, let $\tau_2  \subsetneq \tau_1 $ be such that $\partial(\tau_2)=0$ and there are no $\bar{\tau}  \subsetneq \tau_1 $ such that $\partial(\bar{\tau})=0$  and $\tau_2  \subsetneq \bar{\tau}.$ If no such $\tau_2$ exists, then let $\sigma_2=\tau_1$ and the result follows. Otherwise, let $\sigma_2=\tau_1/\tau_2$. As before, one can check $\sigma_2$ has the desired conditions. After finitely many steps, the process will stop resulting in the desired sum.

We now show $\sigma_i$ can be expressed as stated. Note that since $\sigma_i \neq 0,$ there exists $[x_0,x_1] \in \sigma_i.$ Furthermore, since $\partial(\sigma_i) =0,$ there must exist a $[x_1,x_2] \in \sigma_i$ to cancel the term $x_1$. Similarly, there must also exist a term $[x_2,x_3] \in \sigma_i$ to cancel the term $x_2.$ After repeating this process finitely many times, say $N_i$, we must reach a point with a term $[x_{N_i-1},x_0]$ to cancel the term $x_0.$ Clearly, $$\partial\big(\sum_{j=0}^{N_i-2}[x_j,x_{j+1}]+[x_{N_i-1},x_0]\big)=0,$$ thus by the properties of $\sigma_i$ all of its terms must have been accounted for. The result follows.

\end{proof}

For $[x,y] \in C_1(R_\epsilon(S_{\omega,T})) $, we let $$L_{x,y}=1$$ if in the circle representation of $S_{\omega,T}$, $\bar{d}(x,y) $ is the length of the arc starting at $x$ and ending at $y$ obtained by traversing clockwise. Similarly, we let $$L_{x,y} = -1$$ if $\bar{d}(x,y) $ is the length of the arc traversing counterclockwise. Since $\bar{d}(x,y) \neq 1/2,$ for $x,y \in S_{\omega,T},$ $L_{x,y}$ is well defined. We now define $$f([x,y]) = L_{x,y}\bar{d}(x,y).$$ Using linearity, we extend this function to all of $C_1(R_\epsilon(S_{\omega,T})),$ i.e. if $\sigma = \sum_i[x_i,y_i] \in C_1(R_\epsilon(S_{\omega,T})),$ $$ f(\sigma)= \sum_{i}f([x_i,y_{i}]) = \sum_{i}L_{x_i,y_i}\bar{d}(x_i,y_i).$$

\begin{proposition}
\label{A2}
Let $\sigma \in C_1(R_\epsilon(S_{\omega,T})), \ f(\sigma) \neq 0,$ and $N\in \mathbb{N}$ such that $$\sigma = \sum_{i=0}^N [x_i,x_{i+1}].$$ If there exists an $i_0$ for which $$ L_{x_i,x_{i+1}} = \pm 1 $$ for $0 \leq i \leq i_0$ and  $$ L_{x_i,x_{i+1}} = \mp 1 $$ for $i_0 < i \leq N,$ then $\sigma$ is homologous to a $$ \bar{\sigma} = \sum_{i=0}^{\bar{N}}[\bar{x}_i,\bar{x}_{i+1}] \in C_1(R_\epsilon(S_{\omega,T})),$$ where $\bar{N} \in \mathbb{N}, \ \bar{x}_{i+1}$ is adjacent to $\bar{x}_i,$  and $$ L_{\bar{x}_i,\bar{x}_{i+1}} = \operatorname{sgn}(f(\sigma)).$$ for $0 \leq i \leq \bar{N}.$
\end{proposition}
\begin{proof}
W.L.O.G. we let $$ L_{x_i,x_{i+1}} = 1 $$ for $0 \leq i \leq i_0$. Consider the circle representation of $S_{\omega,T}$. Let $\{x^0_j\}_{j=0}^{N_0}$ denote the points in the arc starting at $x_0$ and ending at $x_1$ traversing clockwise contained in $S_{\omega,T}$. We note $x^0_0=x_0$ and $x^0_{N_0}=x_1.$ If $N_0 > 1$, one can verify $$ \sum_{j=0}^{N_0-1}[x^0_j,x^0_{j+1}] - [x_0,x_1] = \partial( \sum_{j=0}^{N_0-2}[x^0_j,x^0_{j+1},x^0_{N_0}] ),$$ i.e.
$[x_0,x_1]$ is homologous to $$ \sigma_0=  \sum_{j=0}^{N_0-1}[x^0_j,x^0_{j+1}].$$ If $N_0 = 1$ we let $\sigma_0= [x_0^0,x^0_{N_0}]=[x_0,x_1].$ By repeating this process, we construct $\sigma_i$ for $0 \leq i \leq i_0 $ such that $$\sum_{i=0}^{i_0}\sigma_i= \sum_{i=0}^{i_0}\sum_{j=0}^{N_i-1}[x^i_j,x^i_{j+1}]=\sum_{i=0}^{A_1}[y_i,y_{i+1}],$$ where $y_0=x_0$ and $y_{A_1+1}=x_{i_0+1},$ is homologous to $$ \sum_{i=0}^{i_0}[x_i,x_{i+1}].$$
Similarly, by considering the points in the arc starting at $x_i$ and ending at $x_{i+1}$, for $i_0 < i \leq N$, traversing counter-clockwise contained in $S_{\omega,T}$, we can construct $\sigma_i$ as before such that $$ \sum_{i=i_0+1}^{N}\sigma_i = \sum_{i=i_0+1}^{N}\sum_{j=0}^{N_i-1}[x^i_j,x^i_{j+1}]=\sum_{i=0}^{A_2}[\bar{y}_i,\bar{y}_{i+1}], $$ where $\bar{y}_0=x_{i_0+1}$ and $\bar{y}_{A_2+1}=x_{N+1},$ is homologous to $$ \sum_{i=i_0+1}^{N}[x_i,x_{i+1}].$$
We note $y_{A_1+1}=\bar{y}_0=x_{i_0+1}$ and $y_{A_1}=\bar{y}_1$ since the points of the clockwise arc, $\sigma_{i_0}$, ended were the points of the counter-clockwise arc, $\sigma_{i_0+1},$ started. Thus, $[y_{A_1},y_{A_1-1}]=-[\bar{y}_0,\bar{y}_1]$ and $$ \sum_{i=0}^{N}\sigma_i = \sum_{i=0}^{A_1-1}[y_i,y_{i+1}] + \sum_{i=1}^{A_2}[\bar{y}_i,\bar{y}_{i+1}].$$ Repeating this argument $A = \min\{A_1+1,A_2+1 \} $ times, we conclude $[y_{A_1-i},y_{A_1-{i+1}}]=-[\bar{y}_i,\bar{y}_{i+1}]$ for $0 \leq i < A.$ We note that by construction $$ f(\sum_{i=0}^{A_1}[y_i,y_{i+1}] + \sum_{i=0}^{A_2}[\bar{y}_i,\bar{y}_{i+1}])=f(\sigma) \neq 0,$$ thus $A_1 \neq A_2.$  i.e. $$ \sum_{i=0}^{N}\sigma_i = \sum_{i=0}^{A_1-A_2-1}[y_i,y_{i+1}] $$ or $$ \sum_{i=0}^{N}\sigma_i = \sum_{i=A_1+1}^{A_2}[\bar{y}_i,\bar{y}_{i+1}].$$ The result follows.
\end{proof}

\begin{corollary}
Let $\sigma \in C_1(R_\epsilon(S_{\omega,T}))$ and $\partial(\sigma)=0.$ Then $\sigma$ is homologous to $0$ if and only if $f(\sigma) = 0.$
\end{corollary}
\begin{proof}
By Proposition \ref{A1}, $\sigma$ can be expressed as $$\sum_{i=0}^{N-1}[x_i,x_{i+1}],$$ where $x_i \in S_{\omega,T}, N \in \mathbb{N},$ and $x_{N}=x_0.$ The result follows by noting that in Proposition \ref{A2}, $f(\sigma)=0$ if and only if $A_1=A_2,$ which happens if and only if $\sigma$ is homologous to $0$.
\end{proof}
For the following result we let  $\{x_i\}_{i=0}^T$ denote the elements of $S_{\omega,T}$ in ascending order, $\oplus$ denote$\mod{T+1}$ addition, and $k\in \mathbb{N}.$
\begin{lemma}
\label{A4}
Let $\sigma \in C_1(R_\epsilon(S_{\omega,T}))$ be such that $\partial(\sigma)=0$ and $\sigma$ is not homologous to $0$. Then $\sigma$ is homologous to $$k \sum_{i=0}^T [x_i,x_{i\oplus 1}].$$
\end{lemma}
\begin{proof}
By Proposition \ref{A1}, we can express $\sigma$ as $$\sum_{i=0}^{N-1}[\bar{x}_i,\bar{x}_{i+1}],$$ where $\bar{x}_i \in S_{\omega,T}, N \in \mathbb{N},$ and $\bar{x}_{N}=\bar{x}_0.$ Furthermore, by repeated application of Proposition \ref{A2}, $\sigma$ is homologous to $$\bar{\sigma}=\sum_{i=0}^{N_0}[y_i,y_{i+1}],$$ where $y_i \in S_{\omega,T}, N_0 \in \mathbb{N}, \ y_{N_0+1}=y_0, \ y_{i+1}$ is adjacent to $y_i,$ and $$L_{y_i,y_{i+1}} =\operatorname{sgn}(f(\sigma))$$ for $0 \leq i \leq N_0.$ We note that a necessary condition to cancel a term $[y_i,y_{i+1}]$ is the existence of term $[y_{n_i},y_{n_i+1}]$ for which $$L_{y_{n_i},y_{n_i+1}} =-\operatorname{sgn}(f(\sigma)).$$ Since this can't happen, we conclude the expression for $\bar{\sigma}$ can't be reduced. Furthermore, we note that by construction $T \leq N_0.$ Indeed, the minimal number $j_0$ of terms needed to obtain $$ \partial(\sum_{i=0}^{j_0}[y_i,y_{i+1}])=0$$ is $j_0=T$ since this covers all of the points in $S_{\omega,T}.$ We conclude $$ \sum_{i=0}^{T}[y_i,y_{i+1}]=\sum_{i=0}^T [x_i,x_{i\oplus 1}].$$ In fact, an application of Proposition \ref{A1} to $\bar{\sigma},$ shows $$ \bar{\sigma} = \sum_{i=0}^{N_2-1}\sigma_i$$ for some $N_2 \in \mathbb{N}.$ By the properties of $\sigma_i$ and what has been shown, we conclude $$\sigma_i = \sum_{j=0}^T [x_j,x_{j\oplus 1}]$$ for all $0 \leq i \leq N_2-1$. The result follows.
\end{proof}
\begin{lemma}
\label{A5}
Let $T\in \mathbb{N}$ and $\omega\in \mathbb{R}\symbol{92}\mathbb{Q}$ with continued fraction expansion $[a_1,a_2,a_3,\cdots]$, $i$-th convergent $\frac{p_i}{q_i}$, and $k,r,s$ be the unique numbers for which
$$ q_k+q_{k-1} \leq T < q_k + q_{k+1} $$
and $$ T=rq_k+q_{k-1}+s, \ \ \ \ \ 1\leq r \leq a_{k+1}, \ \ \ \ \ 0 \leq s \leq q_k-1.$$ Let
\[
  D_i = q_i\,\omega - p_i.
\]
We assume
\[
  s + 1 < q_k,
  \quad\text{and}\quad
  \lvert D_{k+1}\rvert + (a_{k+1} - r + 1)\,\lvert D_k\rvert < \tfrac13.
\]
Let $ \Gamma $ be the set containing the pairs $(x,y), \ x,y\in S_{\omega,T} \ \text{and} \ x<y$, for which there exists a $z \in ([0,x)\cup (y,1)) \cap S_{\omega,T}$ such that $ 1-\bar{d}(x,y) \leq 2\max\{\bar{d}(y,z),\bar{d}(x,z)\} \leq 2\bar{d}(x,y)$. If $$ \lambda = \min \{\bar{d}(x,y) | \ (x,y) \in \Gamma\},$$ then for $\epsilon \in (|D_{k+1}|+(a_{k+1}-r+1)|D_k|,\lambda)$ $$[\sigma]=0 $$ in $H_1(R_{\epsilon}(S_{\omega,T},\bar{d});\mathbb{F}).$
\end{lemma}
\begin{proof}
Let us denote $\delta_C=|D_{k+1}|+(a_{k+1}-r+1)|D_k|.$ Since $\delta_C < 1/3$ and $\omega$ is irrational, $\lambda > \delta_C,$ thus $\epsilon$ is well defined. To show the result, it is sufficient to show there does not exists a $
\tau \in C_2(R_\epsilon(S_{\omega,T}))$ such that $$\sigma = \partial(\tau).$$
We will assume such $\tau$ exists and arrive at a contradiction. Thus, suppose there exists a $\tau \in C_2(R_\epsilon(S_{\omega,T}))$ such that $$\sigma = \partial(\tau).$$ Let $$\lambda_0=\max\{ \bar{d}(x,y),\bar{d}(y,z),\bar{d}(z,y)\} | \ \text{there exists} \ [x,y,z] \in \tau\}.$$
Let $y_1,y_2\in S_{\omega,T}, \ y_1 < y_2,$ be such that $\bar{d}(y_1,y_2)=\lambda_0.$ By assumption there exists $y_3 \in S_{\omega,T}$ such that $[y_1,y_2,y_3] \in \tau$ or $[y_1,y_3,y_2] \in \tau$. We note that by construction $$\max\{\bar{d}(y_2,y_3),\bar{d}(y_1,y_3)\} \leq \bar{d}(y_1,y_2).$$ Furthermore, W.L.O.G. we can assume $y_3 \in ([0,y_1) \cup (y_2,1)).$ Indeed, let us suppose $y_3 \in (y_1,y_2).$ There must exist a term of the form $[y_1,y_2,z_0]$ to cancel the term $-[y_1,y_2]$ in the case of the term $[y_1,y_3,y_2].$ In the case of the term $[y_1,y_2,y_3]$ take $z_0=y_3$. If $z_0 \in ([0,y_1) \cup (y_2,1))$ we are done. Thus, we assume $z_0 \in (y_1,y_2).$ This implies there must exists a term $[y_1,z_0,z_1] \in \tau $
to cancel the term $-[y_1,z_0]$ from $\partial([y_1,y_2,z_0]).$ If $z_1 \in (y_1,y_2),$ there must be a term of the form $[y_1,z_1,z_2] \in \tau$ to cancel the term $-[y_1,z_1]$ from $\partial([y_1,z_0,z_1]).$ On the other hand, if $z_1 \in ([0,y_1) \cup (y_2,1))$ we note $$L_{y_1,z_1}=-L_{y_1,y_2}$$ or otherwise $$\bar{d}(y_1,z_1) > \bar{d}(y_1,y_2),$$ a contradiction. As before, there must exists a $[z_0,z_2,z_1]$ to cancel the term $[z_0,z_1]$ from $\partial([y_1,z_0,z_1]).$ Since in both scenarios we end up in a case analogues to the term $[y_1,y_2,z_0],$ we conclude this process won't end. Thus, W.L.O.G. we assume $y_3 \in ([0,y_1) \cup (y_2,1))$ as stated.
We note $$L_{y_1,y_2}=L_{y_2,y_3} $$ since otherwise we would have $\bar{d}(y_2,y_3) > \bar{d}(y_1,y_2) = \lambda_0,$ a contradiction. Similarly, it must be the case that $$L_{y_1,y_2}=L_{y_3,y_1} $$ since otherwise we would have $\bar{d}(y_3,y_1) > \bar{d}(y_1,y_2) = \lambda_0$. These two conditions imply $$1-\bar{d}(y_1,y_2) = \bar{d}(y_2,y_3) + \bar{d}(y_3,y_1) \leq 2\max\{\bar{d}(y_2,y_3),\bar{d}(y_1,y_3)\},$$ i.e. $(y_1,y_2) \in \Gamma,$ but since $$\lambda_0 \leq \epsilon < \lambda$$ this is impossible by the minimality of $\lambda.$ Thus, no such $\tau$ exists.
\end{proof}

\end{appendices}

\end{document}